\documentclass[11pt]{article}
\usepackage{amsfonts,amsmath,latexsym}
\usepackage{pgfplots}
\DeclareUnicodeCharacter{2212}{−}
\usepgfplotslibrary{groupplots,dateplot}
\usetikzlibrary{patterns,shapes.arrows}
\pgfplotsset{compat=newest}
\usepackage{bm}
\usepackage[T1]{fontenc}
\usepackage[utf8]{inputenc}
\usepackage{verbatim}
\usepackage{algorithm}
\usepackage{algpseudocode}
\usepackage{eurosym}
\usepackage{enumitem}
\usepackage{dsfont}
\usepackage{hyperref}
\usepackage{graphicx}
\usepackage{epsf}
\usepackage{amsthm}
\usepackage{color}
\usepackage{amssymb}

\usepackage{comment}

\usepackage{caption}
\usepackage{subfigure}

\usepackage{fancyhdr} 

\usepackage{palatino}

\usepackage{mathrsfs}
\usepackage{stmaryrd}

\newtheorem{theorem}{Theorem}[section]
\newtheorem{prop}[theorem]{Proposition}
\newtheorem{procedure}[theorem]{Procedure}
\newtheorem{coro}[theorem]{Corollary}
\newtheorem{remark}[theorem]{Remark}

\newtheorem{lemma}[theorem]{Lemma}
\newtheorem{hyp}[theorem]{Hypothesis}
\newtheorem{condition}[theorem]{Condition}
\newtheorem{definition}[theorem]{Definition}

\renewcommand{\theequation}{\arabic{section}.\arabic{equation}}

\newcommand{\argmin}{\mathop{\mathrm{arg\,min}}}

\def\R{\mathbb R}
\def\N{\mathbb N}

\def\E{\mathbb E}
\def\P{\mathbb P}
\def\Q{\mathbb Q}
\def\U{\mathbb U}
\def\R{{\mathbb R}}
\def\E{{\mathbb E}}
\def\P{{\mathbb P}}
\def\N{{\mathbb N}}
\def\sha{{\cal A}}
\def\shb{{\cal B}}
\def\shc{{\cal C}}

\def\she{{\cal E}}
\def\shf{{\cal F}}
\def\shg{{\cal G}}

\def\shj{{\cal J}}
\def\shk{{\cal K}}

\def\shn{{\cal N}}
\def\shl{{\cal L}}
\def\shp{{\cal P}}

\def\sht{{\cal T}}
\def\shu{{\cal U}}

\def\shy{{\cal Y}}

\def\scrp{{\mathscr P}}

\def\scrs{{\mathscr S}}

\def\1{{\mathds 1}}
\def\vphi{{\varphi}}

\author{
	{\sc Thibaut BOURDAIS}
	\thanks{ENSTA Paris, Institut Polytechnique de Paris.
		Unit\'e de Math\'ematiques Appliqu\'ees (UMA).
		E-mail:{ \tt thibaut.bourdais@ensta-paris.fr}} 
	{\sc,}\ {\sc Nadia OUDJANE}
	\thanks{EDF R\&D,   and FiME (Laboratoire de Finance des March\'es de l'Energie
		(Dauphine, CREST,  EDF R\&D) www.fime-lab.org). 
		E-mail:{\tt  
			nadia.oudjane@edf.fr}}
	\ {\sc and}\ {\sc Francesco RUSSO} 
	\thanks{ENSTA Paris, Institut Polytechnique de Paris.
		Unit\'e de Math\'ematiques Appliqu\'ees (UMA). 
		E-mail:{\tt  francesco.russo@ensta-paris.fr}.
}}

\date{March 2025}

\title{An entropy penalized approach for stochastic optimization
  with marginal law constraints. Complete version}

\newcommand{\MBFigure}[6]{
	$\left. \right.$ \\
	\refstepcounter{figure}
	\addcontentsline{lof}{figure}{\numberline{\thefigure}{\ignorespaces #5}}
	\begin{center}
		\begin{minipage}{#1cm}
			\centerline{\includegraphics[width=#2cm,angle=#3]{#4}}
			\begin{center}
				\upshape{F\textsc{ig} \normal
				\end{center}
				size{\thefigure}. $-$} #5
		\end{center}
		\label{#6}
	\end{minipage}
\end{center}
$\left. \right.$ \\}

\begin{document}
\maketitle
	\begin{abstract} 
          This paper focuses on stochastic optimal control problems with constraints in law, which are rewritten as
          optimization (minimization)
          of probability measures problem on the canonical space. We introduce
          a penalized version of this type of problems by splitting the optimization variable and adding an entropic penalization term.
          We prove that this penalized version constitutes a good approximation of the original control problem and we provide an alternating procedure which converges, under a so called \textit{Stability Condition}, to an approximate solution of the original problem.
          We extend the approach introduced in a previous paper
of the same authors
          including a jump dynamics, non-convex costs and constraints on the marginal
          laws of the controlled process.
          The interest of our approach is illustrated by numerical simulations related to demand-side management problems arising in
          power systems.

	\end{abstract}

\medskip\noindent  {\bf 2020  AMS-classification}: 49M99; 49J99; 60H10; 
 60J60; 65C05.

	\section{Introduction}

        
        Stochastic optimal control problems with constraints in law have constituted a very active field of
        research during the recent years.
For instance
\cite{DPPExpectationConstraints, PfeifferTan, DaudinTerminal, DaudinFokkerPlanck}
treat   expectation constraints. Other references 
 prescribe marginal laws constraints,  related for example to
        optimal transport \cite{ThieullenMikami, TanTouzi, LabordereMartingaleOT, SonerMartingaleOT} or to the
        Schr\"odinger Bridge problems \cite{LeonardSchrodinger, SchrodingerBridgeData}.
        They  offer both  theoretical and practical perspectives.
	On a theoretical point of view, (semi)martingale optimal transport
        and the Schr\"odinger bridge problem can
        be seen as stochastic versions of the classical optimal transport problem \cite{SchrodingerBridgeStochasticControl, LeonardSchrodinger, MarinoSchrodingerOT, ComputationalOT}.
        
        On a practical point of view, expectation constraints are used for instance in demand side management applications in power system, typically to control the smart charging of a fleet of electrical vehicles while ensuring a minimum expected charge level for the batteries at the end of each day \cite{SeguretSmartChargingI, SeguretSmartChargingII}. In financial applications, the martingale optimal transport formulation introduced in \cite{LabordereMartingaleOT} in the discrete time setting and in \cite{ThieullenMikami, TanTouzi} for the continuous time setting provides powerful tools to study model-free hedging of derivatives \cite{LabordereMartingaleOT, SonerMartingaleOT}. As for the Schr\"odinger bridge, it is used for instance to estimate a model interpolating two sampled distributions \cite{SchrodingerBridgeData} or to generate synthetic data \cite{SchrodingerScoreBased, SchrodingerGenerating} that can be used to train neural networks.
                
		In this paper we are interested in optimization problems of the form
		\begin{equation}
			\label{eq:introInitialProblemJump}
			J^* := \inf_{\P \in \shk \cap \shp_\U(\eta_0)} J(\P) \quad \text{with} \quad J(\P) := \E^\P\left[\int_0^T f(r, X_r, \nu_r^\P)dr + g(X_T)\right],
		\end{equation}
		precisely described in Section \ref{sec:descriptionProblem}. More specifically,
              $\shk$ is a convex subset of
the space of probability measures on $\shp(\Omega)$,
where $\Omega = D([0, T], \R^d)$
 will be the canonical space of c\`adl\`ag trajectories with values in $\R^d$
 equipped with the Skorohod metric.
  $\eta_0$ will be a given initial Borel probability on $\R^d$, $\U$ will be a closed subset of
  $\R^p$ for some $p \in \N^*$.
 $\shp_\U(\eta_0)$ is a subset of $\shp(\Omega)$
 defined in Definition \ref{def:Pu} such that under $\P \in \shp_\U(\eta_0)$, the canonical process $X$ decomposes as
		\begin{equation*}
			\left\{
			\begin{aligned}
				& X_t = X_0 + \int_0^t b(r, X_r, \nu_r^\P)dr + \int_0^t \sigma(r, X_r)dB_r + \left(q\1_{\{|q| > 1 \}}\right)*\mu^X_t + \left(q\1_{\{|q| \le 1\}}\right)*(\mu^X - \mu^L)_t\\
				& X_0 \sim \eta_0,
			\end{aligned}
			\right.
		\end{equation*}
		where $\nu^\P : [0, T] \times \Omega \rightarrow \U$ is
                a progressively measurable process.
                $\mu^X$ is the jump measure of $X$ and $\mu^L = dtL(t, X_t, dq)$ its $\P$-compensator characterized by the Lévy kernel $L$ (see Definition \ref{def:levyKernel}) and $B$ is a $\P$-Brownian
                motion. Problem \eqref{eq:introInitialProblemJump} differs
                from a classical stochastic optimal control problem (in the weak formulation) in that it includes an additional constraint $\P \in \shk$. Indeed, when $\shk = \shp(\Omega)$, \eqref{eq:introInitialProblemJump} is a stochastic control problem where only the drift is controlled, while the volatility and the jump intensity are left unchanged. When $\shk \subset \shp(\Omega)$, 
		the formulation \eqref{eq:introInitialProblemJump} recovers many types of aforementioned stochastic control problems with constraints in law. For example, when $\shk = \left\{\P \in \shp(\Omega)~:~\P_T = \eta_T\right\}$ for some prescribed terminal distribution $\eta_T \in \shp(\R^d)$,
                Problem \eqref{eq:introInitialProblemJump} can be interpreted as a
                optimal transport problem in the spirit of \cite{ThieullenMikami, TanTouzi}. In the more specific case where $L = 0$, $b(t, x, u) = u$, $\sigma(t, x) = I_d$, $g = 0$ and $f(t, x, u) = \frac{1}{2}|u|^2$, Problem \eqref{eq:introInitialProblemJump} comes down to
               a
               Schr\"odinger bridge problem, see e.g. \cite{LeonardSchrodinger}.

               We discuss now our methodology  which extends the techniques developed in
                previous paper \cite{BOROptimi2023}.
               Indeed, in that paper, we considered stochastic control problems of the form \eqref{eq:introInitialProblemJump}
               in the case 
               $\shk = \shp(\Omega)$, $ \eta_0$ is deterministic and $L = 0$, i.e. continuous trajectories case.
               Thereby, we introduced a penalized version
of that problem
by splitting  the decision variables and  by adding a relative entropy divergence.
		Following this approach, we propose in this article to approximate Problem \eqref{eq:introInitialProblemJump} by considering the entropy penalized formulation
		\begin{equation}
			\label{eq:introPenalizedProblemJump}
			\shj_\epsilon^* := \inf_{(\P, \Q) \in \sha} \shj_\epsilon(\Q, \P) \quad \text{with} \quad \shj_\epsilon(\Q, \P) := \E^\Q\left[\int_0^T f(r, X_r, \nu_r^\P)dr + g(X_T)\right] + \frac{1}{\epsilon}H(\Q | \P),
		\end{equation}
		where $\sha$ is a subset of $\shp_\U(\eta_0) \times \shk$ (see Definition \ref{def:AJump}), $H$ is the relative entropy (see Definition \ref{def:relativeEntropy}), and $\epsilon > 0$ is a penalization parameter intended to go to $0$ in order to impose $\Q = \P$. The advantage of the penalized version \eqref{eq:introPenalizedProblemJump} is that it allows to distribute the constraint $\shk \cap \shp_\U(\eta_0)$ in \eqref{eq:introInitialProblemJump} either on $\P$ or $\Q$ by considering an alternating minimization method leading to two subproblems
		\begin{equation}
			\label{eq:subpb}
			\underset{\P \in \shp_\U(\eta_0)}{\inf} ~\shj_\epsilon(\Q, \P) \quad \text{and} \quad \underset{\Q \in \shk}{\inf} ~\shj_\epsilon(\Q, \P),
		\end{equation}
		that are easier to solve.
                For example when $\shk = \left\{\P \in \shp(\Omega)~:~\P_T = \eta_T\right\}$, the constraint on the initial law and the dynamics is supported by the variable $\P \in \shp_\U(\eta_0)$, whereas the constraint $\P_T = \eta_T$ is
              reported on $\Q$ by imposing $\Q_T \in \shk$.

                The contribution of the present paper is twofold. First we prove that the entropy penalized formulation \eqref{eq:introPenalizedProblemJump} indeed approximates the original Problem \eqref{eq:introInitialProblemJump} in a sense to be specified, 
                when $\epsilon$ vanishes, see Proposition \ref{prop:goodApproximation}. Second, taking advantage of the simplicity of each subproblem \eqref{eq:subpb}, we propose an alternating algorithm producing a sequence of iterates (by updating alternatively the variables $\P$ and $\Q$) which is proved to provide a good approximate solution of Problem
                \eqref{eq:introPenalizedProblemJump},
                when the number of iterations is sufficiently large.
                We emphasize that the present paper 
                extends \cite{BOROptimi2023} in several directions. First we consider the case of jump 
                (instead of continuous) diffusions. Second, as already mentioned, we take into account additional constraints
                by introducing the subset $\shk$, including the case of a prescribed terminal distribution.
                Third, we are able to deal with non-convex running cost w.r.t. the control variable by the introduction of a (so called)
                Mixed Variational Inequality (MVI)  assumption on $f$, $b$ and $\U$, similarly to \cite{MVIEconomics, MixedVariationalExistence, MVINumerics}.
                The proofs of our results follow the same lines as \cite{BOROptimi2023}, but raise additional technical difficulties, mainly due to the presence of jumps in
                the dynamics and to the MVI assumption. 
		
We now present a short overview of the literature to solve 
Problem \eqref{eq:introInitialProblemJump} in some particular cases.
In the case $\shk = \left\{\P \in \shp(\Omega) \vert \P_T = \eta_T\right\}$, \cite{ThieullenMikami, TanTouzi, LeonardSchrodinger}
rely on Lagrangian dualization techniques.
Under some general conditions,
the solutions of Problem \eqref{eq:introInitialProblemJump} are  characterized by two Lagrange multipliers which verify a coupled system of non-linear differential equations,
which is 
Hamilton-Jacobi-Bellman (HJB) equations in \cite{ThieullenMikami, TanTouzi} or a Schr\"odinger system in \cite{LeonardSchrodinger, OptimalSteeringI}.
In \cite{TanTouzi},
the HJB characterization  of the solution of
optimal transport problems naturally leads to finite difference schemes method, which are very reminiscent of an 
algorithm to solve mass transport problem introduced in \cite{benamou}, whereas the solution to Schr\"odinger system
can be approximated using the so called Iterative Proportional Fitting Procedure, also referred to as Sinkhorn algorithm \cite{FortetSchrodinger, MarinoSchrodingerOT, SchrodingerHilbert}.

When considering expectation constraints corresponding to the case \\
$\shk = \{\P \in \shp(\Omega)~:~\E^\P[\psi(X_T)] \le 0\},$
for some measurable function $\psi$, two types of approaches have been proposed in the literature. In \cite{DaudinTerminal, DaudinFokkerPlanck}, the author considers the Markovian setting and uses a constraint penalization technique to prove that the solution to Problem \eqref{eq:introInitialProblemJump} is 
		characterized by a coupled Hamilton-Jacobi-Bellman system on the space of measures. In \cite{PfeifferTan}, the authors treat a non-Markovian setting by a Lagrangian dualization approach.
		Deep-learning schemes have also been developed to tackle Problem \eqref{eq:introInitialProblemJump} with very general constraints set $\shk$ in \cite{GermainTargetLaw}.
They
                extend to the mean-field setting an exact penalization approach developed in
                \cite{ZidaniStateConstraint} to solve stochastic control problems with state constraints by
                introducing a stochastic target problem.
To this aim they consider    a linear-convex model where the controlled process follows a linear dynamics without jumps ($L = 0$) and the cost is convex.
    Similarly to those techniques, our approach is based on a
    penalization method. However, the originality of our approach relies on the fact that
		we split the decision variable into two variables and introduce the penalization in order to force those two variables to be close.
		
		The paper is organized as follows. Section \ref{sec:penalization} is concerned with the analysis of the entropy penalized formulation \eqref{eq:introPenalizedProblemJump}. In Proposition \ref{prop:goodApproximation} we prove that Problem \eqref{eq:introPenalizedProblemJump} constitutes a good approximation of Problem \eqref{eq:introInitialProblemJump} in the sense of Definition \ref{def:epsilonAdmissible}. In Section \ref{sec:minimization} we propose an alternating minimization procedure to compute a minimizing sequence $(\P^k, \Q^k)_{k \ge 1}$ to Problem \eqref{eq:introPenalizedProblemJump} and prove the convergence of the iterates $(\shj_\epsilon(\Q^k, \P^k))_{k \ge 1}$ to the optimal value $\shj^*$ under a suitable Stability Condition \ref{cond:stability}, which will be related to natural assumptions on Problem \ref{eq:introInitialProblemJump}.
		We introduce in particular Hypothesis \ref{hyp:MVIPointwise} on the running cost $f$ and the admissible set of controls $\U$ allowing to alleviate standard convexity assumptions to prove the convergence of the sequence of the iterates $(\shj_\epsilon(\Q^k, \P^k))_{k \ge 1}$. More specifically, Hypothesis \ref{hyp:MVIPointwise} only requires the existence of a solution to a MVI allowing to tackle some non-convex settings. In Section \ref{sec:examples} we provide several examples showing that the method developed in Section \ref{sec:minimization} has a wide range of applications. In a first example we treat the case of stochastic optimal control with jumps with some settings where the cost function $f$ is not convex in the control variable $u$, see Theorem \ref{th:convergenceControlSto}. In a second example, see Theorem \ref{th:convergenceTerminalLaw},
                we focus on the case where we prescribe the terminal distribution of the state process in the framework of continuous diffusions ($L = 0$).
             Finally in Section \ref{sec:numerics}, numerical simulations illustrate the interest of our approach to deal with demand side management problems arising in power systems.
		
	\section{Notations, definitions and classical results}
		\setcounter{equation}{0}

		
		\begin{itemize}
		\item All vectors $x \in \R^d$ are column vectors. Given $x \in \R^d$, $|x|$ will denote its Euclidean norm.
		\item Given a matrix $A \in \R^{d \times d}$, $\|A\| := \sqrt {Tr[AA^\top]}$ will denote its Frobenius norm.
                \item  $S(\R^d)$ (resp. $S^+(\R^d), S^{++}(\R^d)$) will denote
   the space of symmetric (resp. positive definite, strictly positive
      definite) matrices on $\R^d$.
		\item For any $x \in \R^d$, $\delta_x$ will denote the Dirac mass in $x$.
                \item The function $x \mapsto x \log(x)$ defined on $]0;+\infty[$ will be extended
                  (without further mention) to $x= 0$ by continuity.
	\item $\U$ will denote a closed subset of $\R^p$ where $p \in \N^*$.

                \item Given a measurable space $(E, \she)$,
                  $\shp(E)$
                  will denote the set of probability measures $\P$ on $\she$.
		\item Equality between stochastic processes on some probability space are in the sense of \textit{indistinguishability}.
		\item Given $0 \le t \le T$, $D([t, T], \R^d)$ will denote  of c\`adl\`ag functions defined on $[t, T]$ with values in $\R^d$. In the whole paper $\Omega$ will denote space $D([0,T], \R^d)$. For any $t \in [0, T] $ we denote by $X_t : \omega \in \Omega \mapsto \omega_t$ the coordinate mapping on $\Omega.$ We introduce the $\sigma$-field $\shf := \sigma(X_r, 0 \le r \le T)$. On the measurable space $(\Omega, \shf),$ we introduce the \textit{canonical process} $X : \omega \in ([0, T] \times \Omega, \shb([0, T])\otimes \shf) \mapsto X_t(\omega) = \omega_t \in (\R^d, \shb(\R^d))$.

                We endow $(\Omega, \shf)$ with the right-continuous filtration $\shf_t := \underset{t < s < T}{\bigcap}\sigma(X_r, t \le r \le T).$ The filtered space
		$(\Omega, \shf, (\shf_t))$ will be called the \textit{canonical space} (for the sake of brevity, we denote $(\shf_t)_{t \in [0, T]}$ by $(\shf_t)$).
\item $\shk$ will denote a convex subset of $\shp(\Omega)$.
		
              \item $\shp_{pred}$ will denote the predictable $\sigma$-algebra of $\Omega \times [0,T]$ 
with respect to the filtration
                $(\shf_t)_{t \in [0, T]})$. We also set $\tilde \shp := \shp_{pred}\otimes \shb(\R^d)$.
		\item Given $\P \in \shp(\Omega)$ and $t \in [0, T]$, $\P_t$ will denote the marginal at time $t$ of $\P$, that is the law of $X_t$ under $\P$.
		
		
		\item Given an $(\P, \shf_t)$-local martingale $M$ for $\P \in \shp(\Omega)$, $[M]$ will denote its \textit{quadratic variation}. If moreover $M$ is locally square integrable under $\P$, $\langle M\rangle$ will denote its \textit{predictable quadratic variation}. Note that if $M$ is continuous, $[M] = \langle M \rangle$.


		\item Given a progressively measurable process $Y$ and a stopping time $\tau$, $Y^\tau$ will denote the stopped process $Y_{\cdot \wedge \tau}$. If $\P$ is a probability measure on $(\Omega, \shf)$, $\P^\tau$ will denote the restriction of $\P$ to $\shf_\tau$.

		\item Given a probability measure $\P$ on $(\Omega, \shf)$, $\sha_{loc}^+(\P)$ will denote the space of real-valued c\`adl\`ag adapted process $A$ with non-decreasing path and locally integrable variation.

                \item
Throughout the paper we will use the notion of random measures and their associated {\it compensator}. For a detailed discussion on this topic as well as some unexplained notations we refer to Chapter II and Chapter III in \cite{JacodShiryaev}.
In particular, the \textit{compensator} of a 
 random measure is introduced in
 Theorem 1.8, Chapter II in 
\cite{JacodShiryaev}.

	\item Given an integer-valued random measure $\mu$ and a probability measure $\P$ on $(\Omega, \shf)$, $\shg_{loc}^\P(\mu)$ will denote the set of $\tilde \shp$-measurable functions for which a stochastic integral can be defined in the sense of Definition 1.27, Chapter II of \cite{JacodShiryaev}.

                \item Also, given an integer-valued random measure $\mu$ and a $\tilde \shp$-measurable function $W : [0, T] \times \Omega \times \R^d \rightarrow \R$, we denote $W*\mu := \int_{]0, \cdot] \times \R^d}W(r, X, q)\mu(drdq)$ (when it exists) the integral of $W$ with respect to (w.r.t.) $\mu$.

                  \item
  Given a probability measure $\P \in \shp(\Omega)$ and an integer random measure $\mu$ with compensator $\tilde \mu$, $\shg_{loc}^\P(\mu)$ will denote the set of $\tilde \shp$-measurable function $W$ such that the stochastic integral $W*(\mu - \tilde \mu)$ w.r.t. the compensated measure $\mu - \tilde \mu$ is well-defined.

	\end{itemize}

	\begin{definition}
		\label{def:stochasticInterval}
		(Stochastic interval). Let $\tau$ be a stopping time. We set
		$
		\llbracket 0, \tau \rrbracket := \{(t, \omega) \in [0, T] \times \Omega~:~0 \le t \le \tau(\omega)\}.
		$
	\end{definition}
	\begin{definition}
		\label{def:levyKernel}
		(Lévy kernel). Let $L : [0, T] \times \Omega \times \shb(\R^d)$ be a Borel function. We say that $L$ is a Lévy kernel if we have the following.
		\begin{enumerate}
			\item $L(t, X, .)$ is a $\sigma$-finite non-negative measure on $\R^d$ such that $L(t, X, \{0\}) = 0$.
			
			\item $(t, X) \mapsto L(t, X, A)$ is predictable for all $A \in \shb(\R^d)$.
			
			\item $(t, X) \mapsto \int_A (1 \wedge |q|^2)L(t, X, dq)$ is Borel and bounded for all $A \in \shb(\R^d)$.
		\end{enumerate}
	\end{definition}
	
	\begin{definition}
		\label{def:relativeEntropy}
		(Relative entropy). Let $\P, \Q$ be two elements of $\shp(\Omega)$. The relative entropy $H(\Q | \P)$ of $\Q$ with respect to $\P$ is given by
		\begin{equation*}
			\left\{
			\begin{aligned}
				& \E^\Q\left[\log \frac{d\Q}{d\P}\right] & \text{if}\quad \Q\ll\P\\
				& + \infty &\text{otherwise.}
			\end{aligned}
			\right.
		\end{equation*}
	\end{definition}

	\begin{remark}
		\label{rmk:relativeEntropy}
		\begin{enumerate}
                \item Let $\P, \Q \in \shp(\Omega)$ such that $\Q \ll \P$. Since $t \in [0, 1] \mapsto -t\log(t)$ is bounded by $1/e$, the entropy
                  $$H(\Q\vert \P) = \int_{\Omega}\frac{d\Q}{d\P}\log\frac{d\Q}{d\P}\P(d\omega)= \E^\Q\left[\log \frac{d\Q}{d\P}\right],$$
                  is well-defined, strictly greater than $-\infty$. In fact it is even always non-negative 
                  by Jensen's inequality, the function $x \mapsto x \log (x)$ being convex on $\R_+$. It could be $+ \infty$. 
                \item The relative entropy $H$ is \textbf{non-negative} and \textbf{jointly strictly convex}, that is for all $\P_1, \P_2, \Q_1, \Q_2 \in \shp(E)$, for all $\lambda \in ]0, 1[$,
                  $H(\lambda \Q_1 + (1 - \lambda) \Q_2 | \lambda \P_1 + (1 - \lambda)\P_2) < \lambda H(\Q_1 | \P_1) + (1 - \lambda)H(\Q_2| \P_2)$. Moreover, $(\P, \Q) \mapsto H(\Q | \P)$ is
                  lower semicontinuous with respect to the weak convergence of $\shp(E)$,
                 which corresponds to the  weak* convergence on Polish spaces. 
                  We refer to Lemma 1.4.3
                  in  \cite{DupuisEllisLargeDeviations}  for a proof of those properties.
			
			\item Let $\P, \Q \in \shp(\Omega)$ and let $t \in [0, T]$. Lemma 2.3 in \cite{VaradhanAsymptotic} applied with $\shf_2 = \shf$ and $\shf_1 = \sigma(X_t)$ yields $H(\Q_t | \P_t) \le H(\Q | \P)$.
		\end{enumerate}
	\end{remark}
	
	\begin{definition}
		\label{def:markovProba}
		(Markovian probability measure).
		A probability measure $\P$ is said to be Markovian if for all $t \in [0, T]$, for all $F \in \shb_b(C([t, T], \R^d), \R)$,
		\begin{equation}
			\label{eq:markovProp}
			\E^\P\left[F\left(\left(X_r\right)_{r \in [t, T]}\right)\middle | \shf_t\right] = \E^\P\left[F\left(\left(X_r\right)_{r \in [t, T]}\right)\middle | X_t\right].
		\end{equation}
	\end{definition}
	The result below  is Proposition 3.68, Chapter III in \cite{JacodCalculSto}.
	\begin{prop}
		\label{prop:characGloc}
		Let $\P \in \shp(\Omega)$. Let $\mu$ be a random measure with $\P$-compensator $\mu^L := dtL(t, X, dq),$ where $L$ is a Lévy kernel in the sense of Definition \ref{def:levyKernel}.
        Let $W$ be a $\tilde \shp$-measurable function and let $b_0 > 0$. The following statements are equivalent.
		\begin{enumerate}
                \item $W \in \shg_{loc}^\P(\mu)$. 
			\item $\left(W^2\1_{\{|W| \le b_0\}} + |W|\1_{\{|W| > b_0\}}\right)*\mu^L \in \sha_{loc}^+(\P)$.
                        \end{enumerate}
              \end{prop}
              Below we assume given two progressively measurable functions $b : [0, T] \times \Omega \rightarrow \R^d$ and $a : [0, T] \times \Omega \rightarrow S(\R^d)$ , as well as a Lévy kernel $L$ in the sense of Definition \ref{def:levyKernel}. $(b,a,L)$ represents a classical triplet of
              ''characteristics'' in the sense of
              Jacod-Shiryaev, see Section 2, Chapter II in \cite{JacodShiryaev}.             
              Let also $\eta_0 \in \shp(\R^d)$ .
        \begin{definition}
        	(Martingale problem).
        	\label{def:martProb}
                 Given the triplet $(b, a, L)$,
                a probability measure $\P \in \shp(\Omega)$ is called
                solution of the martingale problem with characteristics $(b, a, L)$ and with initial condition $\eta_0$ if under $\P$ the canonical process
                decomposes as
        	\begin{equation}\label{eq:martProbLin}
        		\left\{
        		\begin{aligned}
        	& X_t = X_0 + \int_0^t b_rdr + M_t^\P + \left(q\1_{\{|q| > 1\}}\right)*\mu^X_t + \left(q\1_{\{|q| \le 1\}}\right)*(\mu^X - \mu^L)_t\\
        	& X_0 \sim \eta_0,
        	\end{aligned}
        	\right.
        	\end{equation}
        	where $M^\P$ is a continuous $\P$-local martingale verifying $[M^\P] = \int_0^\cdot a_rdr$.
        \end{definition}
    \begin{remark}
    	\label{rmk:MartProbStroock}
    	By classical stochastic calculus arguments, see e.g. Theorem 2.42, Chapter II in \cite{JacodShiryaev}, we can state that the two following properties are equivalent.
    	\begin{enumerate}
    		\item $\P$ is solution to the martingale problem associated to $(b, a, L)$ with initial condition $\eta_0$.
    		\item We have $X_0 \sim \eta_0$ under $\P$ and for all
                  bounded functions $\phi \in \shc^{1, 2}$, the process $\phi(\cdot, X_\cdot) - \phi(0, X_0) - \int_0^\cdot \shl(\phi)(r , X)dr$ is a $\P$-local martingale, where
    		\begin{equation}
    			\label{eq:generatorJumps}
    			\begin{aligned}
                          \shl(\phi)(t, X) & = \partial_t \phi(t, X_t) + \langle b(t, X), \nabla_x \phi(t, X_t)\rangle + 	\frac{1}{2}Tr[\nabla_x^2\phi(t, X_t)a(t, X)] \\ 
    				& + \int_{\R^d}\left(\phi(t, X_{t-} + q) - \phi(t, X_{t-}) - \1_{\{|q| \le 1\}}\langle \nabla_x\phi(t, X_{t-}), q\rangle\right)L(t, X, dq).
    			\end{aligned}
    		\end{equation}
    	\end{enumerate}
    \end{remark}
              The definition below postulates that, given a triplet of characteristics $(b,a,L)$
              the ''martingale problem'' with any initial condition at any time $s$ and trajectory position
              $\omega$ restricted to $[0,s]$ admits existence and measurability.
	\begin{definition}
          \label{def:flowExistence}
          (Flow existence property).
   We will say that a triplet of characteristics  $(b, a, L)$ satisfies the
          flow existence property if there exists a
          kernel $(\P^{s, \omega})_{(s, \omega) \in [0, T] \times \Omega}$ (path-dependent class) of elements of $\shp(\Omega)$ verifying the following assertions.
		\begin{enumerate}
			\item For all $(s, \omega) \in [0, T] \times \Omega$, $\P^{s, \omega}(X_r = \omega_r, 0 \le r \le s) = 1$.
			\item For all $(s, \omega) \in [0, T] \times \Omega$, under $\P^{s, \omega}$ the canonical process decomposes for all $t \in [s, T]$ as
			\begin{equation}
				\label{eq:decompSOmega}
				X_t = \omega_s + \int_s^t b_rdr + M^{s, \omega}_t + \left(q\1_{\{|q| > 1\}}\right)*\mu^X_t + \left(q\1_{\{|q| \le 1\}}\right)*(\mu^X - \mu^L)_t,
			\end{equation}
			where $M^{s, \omega}$ is a continuous $\P^{s, \omega}$-local martingale such that $[M^{s, \omega}] = \int_s^\cdot a_rdr$.
                      \item The function $(s, \omega) \in [0, T] \times \Omega \mapsto \P^{s, \omega} \in \shp(\Omega)$ is progressively measurable
                        in the sense that
                        $(s,\omega) \mapsto \P^{s,\omega}(A)$
                   is progressively measurable for every $A \in \shf$.
		\end{enumerate}
	\end{definition}
%
	\begin{definition} \label{def:MinSeq} (Minimizing sequence, solution and
		$\epsilon$-solution).
		Let $(Y, d)$ be a metric space. Let $J : Y \mapsto \R_+$ be a function.
                Let $Z \subset Y$. Let $J^* := \underset{x \in Z}{\inf} J(x)$.
                \begin{enumerate}
                \item A \textit{minimizing sequence} for $J$ is a sequence $(x_n)_{n \ge 0}$ of elements of $Z$ such that $J(x_n) \underset{n \rightarrow + \infty}{\longrightarrow} J^*$.
                  We remark that $J^*$ is finite if
                  and only if $Z$ is non empty.
     	\item We will say that $x^* \in Z$ is a solution
                          to (or minimizer of) 
                          the optimization Problem
			\begin{equation} \label{eq:MinSeq}
				\underset{x \in Z}{\inf} J(x),
			\end{equation}
			if $J(x^*) = J^*.$ In this case, $J^* = \underset{x \in Z}{\min} J(x)$.
			\item For $\epsilon \ge 0$, we will say that $x_\epsilon \in Z$ is an $\epsilon$-solution to the optimization Problem \eqref{eq:MinSeq} if $0 \le J(x_\epsilon) - J^* \le \epsilon$. We also say that $x_\epsilon$ is $\epsilon$-optimal for the (optimization) Problem \eqref{eq:MinSeq}.
		\end{enumerate}
		
	\end{definition}
	The following is Definition 17.1 in \cite{chara}.
	\begin{definition}
		\label{def:correspondence}
		(Correspondence and measurable selector).
		\begin{enumerate}
                \item A correspondence $\sht$ from a set $S$ to a set $Y$
is an application $S \rightarrow 2^Y$.
                 \item
Let $(S,\scrs)$ and $(Y, \shy)$ be two measurable spaces.
Given a correspondence $\sht$ from $S$ to $Y$,
a measurable selector of $\sht$ is a $(\scrs, \shy)$-measurable map
$m : S \rightarrow Y$ such that for all $s \in S$, $m(s) \in \sht(s)$.
		\end{enumerate}
		
	\end{definition}
	

	\section{From the original problem to the penalized problem}
	\label{sec:penalization}
	\setcounter{equation}{0}
	In the rest of the paper, we fix a probability measure $\eta_0 \in \shp(\R^d)$.
	
	In this section we formulate a minimization problem on the space of probability measures $\shp(\Omega)$ and we introduce a regularized version of this problem by splitting the variables and adding a relative entropy penalization. We prove in Proposition \ref{prop:goodApproximation} that the regularized problem approximates in some sense the original optimization program.
	
	\subsection{Description of the optimization problem on the space of probability measures}
	\label{sec:descriptionProblem}

	We will consider  in the whole paper a (controlled) drift $b \in \shb([0, T] \times \R^d \times \U, \R^d)$, a diffusion matrix $\sigma \in \shb([0, T] \times \R^d, \R^{d \times d})$ and a Lévy kernel $L$ in the sense of Definition \ref{def:levyKernel}. We assume in the rest of the paper that $\sigma\sigma^\top$ is invertible and we introduce the notation
	\begin{equation}
		\label{eq:defSigma}
		\Sigma := \sigma\sigma^\top, \quad \sigma^{-1} := \sigma^\top \Sigma^{-1},
	\end{equation}
	$\sigma^{-1}$ being the \textit{generalized inverse} of $\sigma$.
	For the moment we do not make any other assumption on these coefficients. For the rest of the article, we will denote
	\begin{equation}
		\label{eq:defMuL}
		\mu^L(X_{t-}, dt, dq) := L(t, X_{t-}, dq)dt.
	\end{equation}
Later we will make use of the following assumption on the coefficients $b$ and $\sigma$.
	\begin{hyp}
		\label{hyp:boundedCoef}
		(Bounded drift and ellipticity).
		\begin{enumerate}
			\item $b$ is bounded.
			\item There exists $c_\sigma > 0$ such that for all $(t, x) \in [0, T] \times \R^d$, $\xi \in \R^d$,
			$
			\xi^\top\Sigma(t, x)\xi \ge c_\sigma |\xi|^2.
			$
		\end{enumerate}
	\end{hyp}
	We now define the set of admissible dynamics $\shp_\U(\eta_0)$ for our optimization problem.
	\begin{definition}
		\label{def:Pu}
		(Admissible dynamics).
		A probability measure $\P$ is an element of $\shp_\U(\eta_0)$ if there exists a progressively measurable process $\nu^\P : [0, T] \times \Omega \rightarrow \U$ such that, under $\P$, the canonical process
                $X$ has decomposition
		\begin{equation}
			\label{eq:decompPu}
			\left\{
			\begin{aligned}
	& X_t = X_0 + \int_0^t b(r, X_r, \nu_r^\P)dr + M_t^\P + \left(q\1_{\{|q| > 1\}}\right)*\mu^X_t + \left(q\1_{\{|q| \le 1\}}\right)*(\mu^X - \mu^L)_t\\
				& X_0 \sim \eta_0,
                        \end{aligned}
			\right.
		\end{equation}
		where $M^\P$ is a continuous local martingale such that $[M^\P] = \int_0^\cdot \Sigma(r, X_r)dr$ and $\mu^L (=\mu^{L, \P})$ is the $\P$-compensator of $\mu^X$ given by \eqref{eq:defMuL}. If $\nu^\P_t = u^\P(t, X_t)$ for a function $u^\P \in \shb([0, T] \times \R^d, \U)$, we will denote $\P \in \shp_\U^{Markov}(\eta_0)$.
              \end{definition}
            
	In the following, for any $\P \in \shp_\U(\eta_0)$, we will use the notation
	\begin{equation}
		\label{eq:simplifiedNotations1}
		b_t^\P := b(t, X_t, \nu_t^\P).
              \end{equation}
              \begin{remark}\label{rmk:MartProb}
                Given  a progressively measurable function $u : [0, T] \times \Omega \rightarrow \U$,
                setting $\nu^\P = u(\cdot, X)$, $\P$ is a solution to the martingale problem
		\begin{equation}
			\label{eq:decompPuMart}
			\left\{
			\begin{aligned}
				& X_t = X_0 + \int_0^t b(r, X_r, u(r, X))dr + M_t^\P + \left(q\1_{\{|q| > 1\}}\right)*\mu^X_t + \left(q\1_{\{|q| \le 1\}}\right)*(\mu^X - \mu^L)_t\\
				& X_0 \sim \eta_0,
			\end{aligned}
			\right.
		\end{equation}
		where $M^\P$ is a continuous local martingale such that
		$[M^\P] = \int_0^\cdot \Sigma(r, X_r)dr$ and $\mu^L$ is given by \eqref{eq:defMuL}.
                \begin{itemize}
\item
                In particular $\P$ is a solution to the martingale problem
                in the sense of Definition \ref{def:martProb} with characteristics
                $(b(t,X_t,u(t,X)), \Sigma(t,X_t), L)$.
              \item If given $u$ there is only one probability measure  $\P$
                such that $\P$ is a solution of previous
                martingale problem we will
                write $\P= \P^u$.
                \end{itemize}
              \end{remark}

	We also define a running (resp. terminal) cost $f \in \shb([0, T] \times \R^d \times \U, \R)$ (resp. $g \in \shb(\R^d, \R)$),
       which satisfy the following assumption.
	\begin{hyp}
		\label{hyp:runningCostJumps}
		(Running and terminal costs).
		\begin{enumerate}
			\item \label{item:fPositive} $f, g \ge 0$.
			\item \label{item:bounded} $f, g$ are bounded.
		\end{enumerate}
	\end{hyp}
	\begin{remark}
		\begin{enumerate}
                \item
From now on, for simplicity of the exposition we will set $g=0$.
\item Hypothesis \ref{hyp:runningCostJumps} item \ref{item:bounded} could be replaced by a polynomial growth assumption on $f$ but this would require strong regularity of the Lévy kernel $L$. It would be difficult to ensure the existence of a probability measure $\P \in \shp_\U(\eta_0)$ such that $J(\P) < + \infty$, where $J$ was defined
in \eqref{eq:introInitialProblemJump},
  without the existence of moments for $X$ under $\P$. We chose to have bounded costs in order to consider more complex dynamics.
		\end{enumerate}
		
	\end{remark}
	We are interested in the optimization Problem \eqref{eq:introInitialProblemJump} on the space of probability measures. Recall that the admissible set is of the form $\shp_\U(\eta_0) \cap \shk$,
	where $\shk$ is a convex subset of $\shp(\Omega)$.
	\begin{remark}
		Typical examples of $\shk$ include the following.
		\begin{enumerate}
			\item $\shk = \shp(\Omega)$.
			\item $\shk = \left\{\Q \in \shp(\Omega)~:~\E^\Q\left[\psi(X_T)\right] \le 0\right\}$, for some $\psi \in \shb(\R^d, \R)$.
			\item $\shk = \left\{\Q \in \shp(\Omega)~:~\Q_T = \eta_T\right\}$ for a fixed $\eta_T \in \shp(\R^d)$.
		\end{enumerate}
		We will examine in detail the cases related to $1.$ and $3.$ in Section \ref{sec:examples} and provide numerical simulations for $3.$ in Section \ref{sec:numerics}. 
	\end{remark}
	\begin{remark} \label{rmk:deal}
          \begin{enumerate}
          \item	The formulation \eqref{eq:introInitialProblemJump} allows to deal with problems falling outside the classical framework of the stochastic optimal control, which corresponds to $\shk = \shp(\Omega)$,
            see item 1.
\item            The optimization problem \eqref{eq:introInitialProblemJump} with $\shk = \shp(\Omega)$, without jumps and with convex costs was the object of  \cite{BOROptimi2023}.
\end{enumerate}
\end{remark}

	The admissible set $\shk \cap \shp_\U(\eta_0)$ can a priori be empty. In this case $J^* = + \infty$ and Problem \eqref{eq:introInitialProblemJump} has no interest. We will then often assume the following.
	\begin{hyp}
		\label{hyp:notEmpty}
                (Feasibility).
                $\shk \cap \shp_\U(\eta_0) \neq \emptyset$.
	\end{hyp}
	\begin{remark} \label{rmk:J*}
          Assume Hypotheses \ref{hyp:runningCostJumps} and \ref{hyp:notEmpty}.
          Then $J^*  < + \infty$. Indeed, $f$ being bounded,
          we have $J^* \le J(\P) < \infty$, for every
          $\P \in  \shp_\U(\eta_0) \cap \shk$.
	\end{remark}

	\subsection{Entropic penalization}

	We formulate an approximation of Problem \eqref{eq:introInitialProblemJump} by doubling the optimization variables and adding a relative entropy penalization. We therefore obtain Problem \eqref{eq:introPenalizedProblemJump},
	where $\epsilon > 0$ is the penalization parameter and the admissible set $\sha$ is the subset of $\shp(\Omega)^2$ defined below.
	\begin{definition}
		\label{def:AJump}
		(Regularized set). Let $\sha$ be the subset of probability measures $(\P, \Q) \in \shp(\Omega)^2$ such that
		\begin{enumerate}
			\item $\P \in \shp_\U(\eta_0)$ and $\Q \in \shk$.
			\item $H(\Q | \P) < + \infty$.
		\end{enumerate}
	\end{definition}
	We can characterize the elements of the set $\sha$ in terms of the decomposition of the canonical process $X$.
	\begin{lemma}
          \label{lemma:decompQ}
		Let $(\P, \Q) \in \sha$. We have the following properties.
		\begin{enumerate}
			\item Under $\P$ the canonical process decomposes as \eqref{eq:decompPu}, and, under $\Q$, the canonical process decomposes as
			\begin{equation}
				\label{eq:decompQ}
				X_t = X_0 + \int_0^t \beta^\Q_rdr + \left(q\1_{\{|q| \le 1\}}(Y^\Q - 1)\right)*\mu^L_t + M_t^\Q + \left(q\1_{\{|q| > 1\}}\right)*\mu^X_t + \left(q\1_{\{|q| \le 1\}}\right)*(\mu^X - Y^\Q.\mu^L)_t,
			\end{equation}
			where $\beta^\Q : [0, T] \times \Omega \rightarrow \R^d$ is a progressively measurable function,
			$Y^\Q : [0, T] \times \Omega \times \R^d \rightarrow \R^+$ is $\tilde \shp$-measurable and $M^\Q$ is a local martingale verifying $[M^\Q] = \int_0^\cdot \Sigma(r, X_r)dr$.
			\item Taking into account \eqref{eq:simplifiedNotations1}, it holds that
			\begin{equation}
				\label{eq:entropyEstimate}
				H(\Q | \P) \ge H(\Q_0 | \P_0) + \frac{1}{2}\E^{\Q}\left[\int_0^T|\sigma^{-1}_r(\beta^\Q_r - b_r^\P)|^2dr\right] +  \E^\Q\left[(Y^\Q\log(Y^\Q) - Y^\Q + 1)*\mu^L_T\right],
			\end{equation}
			with equality in \eqref{eq:entropyEstimate} if the martingale problem \eqref{eq:martProbLin}
                        with characteristics $(b,a,L)$  
with
                        $b_r = b(r,X_r,u(r,X))$, $a_r = \Sigma(r,X_r)$,
			verified by $\P$, given a functional $u$, has a unique solution. 
			\item Assume moreover Hypothesis \ref{hyp:boundedCoef}. Then $\E^\Q\left[\int_0^T |\sigma_r^{-1}\beta_r^\Q|^2dr\right] < + \infty$.
			
		\end{enumerate}
	\end{lemma}
	\begin{proof}
          Let $(\P, \Q) \in \sha$.
          The fact that \eqref{eq:decompPu} holds under $\P$ holds 
          by item 1. of  Definition \ref{def:AJump} and Definition  \ref{def:Pu}
          
          By Definition \ref{def:AJump} item 2., $H(\Q | \P) < + \infty$. Since $\P \in \shp_\U(\eta_0)$ verifies decomposition \eqref{eq:decompPEntropy},
        the rest of item 1.  and item 2. follow essentially by Theorems 2.1, 2.3, 2.6 and 2.9 in \cite{GirsanovEntropy}. 
In particular
item 1. is an application of Theorem \ref{th:entropyJumps}
in the Appendix,
setting
          $\beta^\Q_t := b(t, X_t, \nu^\P_t) + \Sigma(t, X_t)\alpha_t$ in Theorem \ref{th:entropyJumps}; item 2. then follows directly from \eqref{eq:lowerBoundEntropy} in Theorem \ref{th:entropyJumps}.
          
    Suppose now Hypothesis \ref{hyp:boundedCoef}, it holds that, for all $t \in [0, T]$,
          $$
          |\sigma^{-1}_t\beta_t^\Q|^2 \le 2|\sigma^{-1}_tb_t^\P|^2 + 2|\sigma_t^\top\alpha_t|^2 \le \frac{2}{c_\sigma}|b|_\infty^2 + 2|\sigma^\top_t\alpha_t|^2,
          $$
          hence
          $$
          \E^\Q\left[\int_0^T |\sigma^{-1}_r\beta_r^\Q|^2dr\right] \le \frac{2T}{c_\sigma}|b|_\infty^2 + 2\E^\Q\left[\int_0^T |\sigma^\top_r\alpha_r|^2dr\right],
          $$
          and item 3. follows by \eqref{eq:lowerBoundEntropy},
          remarking that  $y \mapsto  y \log(y) - y +  1$
          is a non-negative function.
        \end{proof}
     
	\begin{remark}
		\label{rmk:linkPb1Pb2}
		Assume Hypothesis \ref{hyp:notEmpty} and let $\P \in \shk \cap \shp_\U(\eta_0)$. Then $(\P, \P) \in \shp(\Omega)^2$ clearly satisfies item 1. and 2. of Definition \ref{def:AJump}. Hence $(\P, \P) \in \sha$ and we have $\shj_\epsilon(\P, \P) = J(\P)$.
	\end{remark}
	We want Problem \eqref{eq:introPenalizedProblemJump} to be a "good approximation" of the original Problem \eqref{eq:introInitialProblemJump}. To this aim the penalization parameter $\epsilon$ is intended to go to $0$. Intuitively, provided that Hypothesis \ref{hyp:notEmpty} is verified, the relative entropy penalization $\frac{1}{\epsilon}H(\Q | \P)$ increases when $\epsilon$ vanishes, forcing the probability measure $\Q$ and $\P$ to get closer in the sense of the relative entropy to keep the cost $\shj_\epsilon(\Q, \P)$ small. At the limit $\epsilon = 0$, any solution $(\Q^*, \P^*)$ of Problem \eqref{eq:introPenalizedProblemJump} is expected to satisfy $H(\Q^* | \P^*) = 0$, hence $\Q^* = \P^*$, $\P^* \in \shk \cap \shp_\U(\eta_0)$ and $\P^*$ is solution of Problem \eqref{eq:introInitialProblemJump}. The aim of this paper is then to use the penalized version \eqref{eq:introPenalizedProblemJump} of Problem \eqref{eq:introInitialProblemJump} to provide approximate solutions to Problem \eqref{eq:introInitialProblemJump} in the sense of Definition \ref{def:epsilonAdmissible} below.
	
	\begin{definition}
	 	\label{def:epsilonAdmissible}
                ($\epsilon_1$-admissible and $(\epsilon_1, \epsilon_2)$-approximate solution).
   Let  $\epsilon_1, \epsilon_2 > 0$ and $\P \in \shp_\U(\eta_0)$.
	 	\begin{enumerate}
                \item We say that $\P$ is $\epsilon_1$-admissible
              if $\underset{\Q \in \shk}{\inf} H(\Q | \P) \le \epsilon_1$.
            \item We say that $\P$ is an $(\epsilon_1, \epsilon_2)$-approximate solution of Problem \eqref{eq:introInitialProblemJump} if $\P$ is $\epsilon_1$-admissible
         and $J(\P) \le J^* + \epsilon_2$, where $J^*$ is the infimum defined in Problem \eqref{eq:introInitialProblemJump}.
	 	\end{enumerate}
	\end{definition}
        Note that an $\epsilon_1$-admissible solution $\P$ of Problem \eqref{eq:introInitialProblemJump} always belongs to  $\shp_\U(\eta_0)$ but
      not  necessarily to $\shk$.
	\begin{remark} \label{rmk:IndP21}
        
\begin{enumerate}
\item          The constraint on the dynamics
          $\P \in \shk \cap \shp_\U(\eta_0)$ in Problem \eqref{eq:introInitialProblemJump} has been decoupled in the penalized Problem \eqref{eq:introPenalizedProblemJump} and the constraints $\P \in \shp_\U(\eta_0)$, $\Q \in \shk$ and $H(\Q | \P) < + \infty$ are much easier to satisfy.
\item   Problem \eqref{eq:introPenalizedProblemJump} is defined independently of
          Problem \eqref{eq:introInitialProblemJump}.
          In particular the admissible set $\sha$ can be non-empty even if Hypothesis \ref{hyp:notEmpty} is not verified, i.e.
          $\shk \cap \shp_\U(\eta_0) = \emptyset$.
          \end{enumerate}
	\end{remark}

	\begin{prop} 
		\label{prop:goodApproximation}
		Assume Hypotheses \ref{hyp:runningCostJumps} and \ref{hyp:notEmpty}. Let $\epsilon' > 0$ and let $(\P_\epsilon^{\epsilon'}, \Q_\epsilon^{\epsilon'}) \in \sha$ be an $\epsilon'$-solution of Problem \eqref{eq:introPenalizedProblemJump} in the sense of Definition \ref{def:MinSeq} item 3.
                with $Z = \sha$.
                We have the following.
		\begin{enumerate}
			\item $\underset{\Q \in \shk}{\inf} H(\Q |  \P_\epsilon^{\epsilon'}) \le \epsilon(J^* + \epsilon')$, where $J^*$ is the infimum of Problem \eqref{eq:introInitialProblemJump}.
			\item Set $Y_{\epsilon}^{\epsilon'} := \int_0^T f(r, X_r, \nu_r^{\P_\epsilon^{\epsilon'}})dr,$
      where $\nu^\P$ refers to \eqref{eq:decompPu} in  Definition \ref{def:Pu}. 
                          Then
			\begin{equation}
				\label{eq:goodApproximation}
				J(\P_{\epsilon}^{\epsilon'}) - J^* \le \frac{\epsilon}{2}Var[Y_{\epsilon}^{\epsilon'}] + \epsilon',
			\end{equation}
                      \end{enumerate}
    where the variance {\it Var} refers to the probability
        $\P^{\epsilon'}_\epsilon$.
	\end{prop}
	\begin{remark} \label{rmk:goodApproximation}
          Proposition \ref{prop:goodApproximation} shows that any $\epsilon'$-solution $(\P_\epsilon^{\epsilon'}, \Q_\epsilon^{\epsilon'})$ of Problem \eqref{eq:introPenalizedProblemJump}
          with $Z = \sha$
          provides an
		$\left(\epsilon(J^* + \epsilon'), \frac{\epsilon}{2}Var[Y_{\epsilon}^{\epsilon'}] + \epsilon'\right)$-approximate solution $\P_\epsilon^{\epsilon'}$ to Problem \eqref{eq:introInitialProblemJump} in the sense of Definition \ref{def:epsilonAdmissible},
            
              \end{remark}


	\begin{proof} [Proof of Proposition \ref{prop:goodApproximation}]
		Let $\tilde \P \in \shk \cap \shp_\U(\eta_0)$. Then by Remark \ref{rmk:linkPb1Pb2}, $(\tilde \P, \tilde \P) \in \sha$ and we have
		\begin{equation}
			\label{eq:ineqInf}
			\inf_{(\P, \Q) \in \sha} \shj_\epsilon(\Q, \P) \le \shj_\epsilon(\tilde \P, \tilde \P) = J(\tilde \P).
		\end{equation}
		Since   inequality \eqref{eq:ineqInf} holds for any $\tilde \P \in \shk \cap \shp_\U(\eta_0)$, we get that $\shj_\epsilon^* \le J^*$. Let then $(\P_\epsilon^{\epsilon'}, \Q_\epsilon^{\epsilon'})$ be an $\epsilon'$-solution to Problem \eqref{eq:introPenalizedProblemJump}. By definition $\shj(\Q_\epsilon^{\epsilon'}, \P_\epsilon^{\epsilon'}) \le \shj_\epsilon^* + \epsilon' \le J^* + \epsilon'$ and since $f \ge 0$, the previous inequality yields
		$$
		\frac{1}{\epsilon}	H(\Q_\epsilon^{\epsilon'} |  \P_\epsilon^{\epsilon'}) \le \E^{\Q_\epsilon^{\epsilon'}}\left[\int_0^T f(r, X_r, \nu_r^{\P_\epsilon^{\epsilon'}})dr\right]
                + \frac{1}{\epsilon}	H(\Q_\epsilon^{\epsilon'} |  \P_\epsilon^{\epsilon'}) \le J^* + \epsilon',
		$$
		that is $H(\Q_\epsilon^{\epsilon'} |  \P_\epsilon^{\epsilon'}) \le \epsilon(J^* + \epsilon'),$ hence item 1.

                We now prove item 2. Note first that since $f$ is bounded, then $\E^{\P_\epsilon^{\epsilon'}}[(Y_\epsilon^{\epsilon'})^2] < + \infty$
		and $Y_\epsilon^{\epsilon'} \in L^2(\P_\epsilon^{\epsilon'})$. A direct application of Lemma \ref{lemma:squareIntVar} in the Appendix with $\eta = Y_\epsilon^{\epsilon'}$ then yields
		$$
		0 \le \E^{\P_\epsilon^{\epsilon'}}\left[Y_\epsilon^{\epsilon'}\right] - \left(-\frac{1}{\epsilon}\log \E^{\P_\epsilon^{\epsilon'}}\left[\exp\left(-\epsilon Y_\epsilon^{\epsilon'}\right)\right]\right) \le \frac{\epsilon}{2}Var\left[Y_\epsilon^{\epsilon'}\right],
		$$
		which rewrites
		\begin{equation}
			\label{eq:varIneq}
			0 \le J(\P_\epsilon^{\epsilon'}) - \left(-\frac{1}{\epsilon}\log \E^{\P_\epsilon^{\epsilon'}}\left[\exp\left(-\epsilon Y_\epsilon^{\epsilon'}\right)\right]\right) \le \frac{\epsilon}{2}Var\left[Y_\epsilon^{\epsilon'}\right].
                      \end{equation}
                      Let us denote $u^{\P_\epsilon^{\epsilon'}}:[0,T] \times \Omega \rightarrow \R$ a progressively measurable function
                      such that $\nu_r^{\P_\epsilon^{\epsilon'}} =  u^{\P_\epsilon^{\epsilon'}}(r, X)$.
		Let then $\tilde \Q_\epsilon^{\epsilon'}$ be the probability $\Q^*$ provided by
                Proposition \ref{prop:minimizerUnconstrained} applied with $\vphi(X) = \int_0^T f(r, X_r, u^{\P_\epsilon^{\epsilon'}}(r, X))dr$ and $\P = \P_\epsilon^{\epsilon'}$. By \eqref{eq:klOpti},
\begin{equation}\label{eq:klOptiAppl}
  \shj_\epsilon(\tilde \Q_\epsilon^{\epsilon'}, \P_\epsilon^{\epsilon'}) = -\frac{1}{\epsilon}\log \E^{\P_\epsilon^{\epsilon'}}\left[\exp\left(-\epsilon Y_\epsilon^{\epsilon'}\right)\right],
\end{equation}
  and by definition of $\tilde \Q_\epsilon^{\epsilon'}$, it holds
		\begin{equation}
			\label{eq:ineqInfTrivial}
			\shj_\epsilon(\tilde \Q_\epsilon^{\epsilon'}, \P_\epsilon^{\epsilon'}) \le \shj_{\epsilon}(\Q_\epsilon^{\epsilon'}, \P_\epsilon^{\epsilon'}).
		\end{equation}
		Moreover, since $(\P_\epsilon^{\epsilon'}, \Q_{\epsilon}^{\epsilon'})$ is an $\epsilon'$-solution of Problem \eqref{eq:introPenalizedProblemJump} we have
                \begin{equation}
			\label{eq:QPEpsilonSolution}
			\shj_\epsilon(\Q_\epsilon^{\epsilon'}, \P_\epsilon^{\epsilon'}) - \shj_\epsilon^* \le \epsilon'.
		\end{equation}
		Then, by \eqref{eq:klOptiAppl} we have
		\begin{equation}
			\label{eq:ineqInfimum}
			\begin{aligned}
				-\frac{1}{\epsilon}\log \E^{\P_\epsilon^{\epsilon'}}\left[\exp\left(-\epsilon Y_\epsilon^{\epsilon'}\right)\right] - \shj_\epsilon^* & = \shj_\epsilon(\tilde \Q_\epsilon^{\epsilon'}, \P_\epsilon^{\epsilon'}) - \shj_\epsilon^*\\
				& = \shj_\epsilon(\tilde \Q_\epsilon^{\epsilon'}, \P_\epsilon^{\epsilon'}) - \shj_\epsilon(\Q_\epsilon^{\epsilon'}, \P_\epsilon^{\epsilon'}) + \shj_\epsilon(\Q_\epsilon^{\epsilon'}, \P_\epsilon^{\epsilon'}) - \shj_\epsilon^*\\
                & \le \epsilon',
			\end{aligned}
		\end{equation}
		where we have used \eqref{eq:ineqInfTrivial} and \eqref{eq:QPEpsilonSolution} for the latter inequality. Recall that $\shj_\epsilon^* \le J^*$. Then by \eqref{eq:varIneq}, \eqref{eq:ineqInfimum} and the fact that $\shj_\epsilon^* \le J^*$, we get
		\begin{equation*}
			\begin{aligned}
				J(\P_\epsilon^{\epsilon}) - J^* & =  J(\P_\epsilon^\epsilon) - \left(-\frac{1}{\epsilon}\log \E^{\P_\epsilon^{\epsilon'}}\left[\exp\left(-\epsilon Y_\epsilon^{\epsilon'}\right)\right]\right) - \frac{1}{\epsilon}\log \E^{\P_\epsilon^{\epsilon'}}\left[\exp\left(-\epsilon Y_\epsilon^{\epsilon'}\right)\right] - \shj_\epsilon^* + \shj_\epsilon^* - J^*\\
				& \le \frac{\epsilon}{2}Var\left[Y_\epsilon^{\epsilon'}\right] + \epsilon'.
			\end{aligned}
		\end{equation*}
	\end{proof}

	\section{Minimization of the entropy penalized functional}
	\label{sec:minimization}
	\setcounter{equation}{0}
	This section focuses on the resolution of the entropy penalized problem \eqref{eq:introPenalizedProblemJump}. From now on, $\epsilon$ will be implicit in the cost function $\shj_\varepsilon$,
        to alleviate notations. In this section we aim at building a minimizing sequence $(\P^k, \Q^k)_{k \ge 1}$ of the functional $\shj$. Our construction requires an assumption on the existence of solutions to a Mixed Variational Inequalities (MVI) involving the running cost $f$ and the drift coefficient $b$.
	\begin{definition}
		\label{def:solutionSet}
		(MVI, Solution set).
		For all $(t, x, \delta) \in [0, T] \times \R^d \times \R^d$, let $\scrs(t, x, \delta)$ be the set of elements $\bar u \in \U$ verifying
                the Mixed Variational Inequality, indexed by
                $(t, x, \delta)$,
		\begin{equation}
			\label{eq:MVIPointwise}
			\tag{$(MVI)_{t, x, \delta}$}
			f(t, x, u) - f(t, x, \bar u) + \frac{1}{\epsilon}\langle \Sigma^{-1}(t, x)(b(t, x, \bar u) - \delta), b(t, x, u) - b(t, x, \bar u)\rangle \ge 0, ~\forall u \in \U.
                      \end{equation}
                    \end{definition} 
In the rest of the section, we will assume that Hypothesis \ref{hyp:MVIPointwise} below is verified.
	\begin{hyp}(MVI).
		\label{hyp:MVIPointwise}
		For all $(t, x, \delta) \in [0, T] \times \R^d \times \R^d$, we have $\scrs(t, x, \delta) \neq \emptyset$.
	\end{hyp}
	\begin{remark} \label{rmk:MVIPointwise}
          \cite{MVIEconomics, MixedVariationalExistence} provide conditions ensuring
the validity of Hypothesis \ref{hyp:MVIPointwise}.
\begin{itemize}
\item For example, this holds under
  Hypothesis \ref{hyp:KtxConvex},
  see Lemma \ref{lemma:solutionMVIFullConvex} in Section \ref{sec:mixedVar}.
  This case was treated in \cite{BOROptimi2023}. 
\item It is also worth noting that under stronger assumptions on the drift $b$
  (see Hypothesis \ref{hyp:linear}), \eqref{eq:MVIPointwise} admits solutions even when $f(t, x, \cdot)$ is not convex, see Lemma \ref{lemma:solutionMVILinear}
  below.
\end{itemize}
\end{remark}

We now introduce two crucial properties. Property \ref{item:min} below concerns a probability measure $\P \in \shp(\Omega)$.
\begin{enumerate}[label = (Min)]
		\item \label{item:min} There exists a unique solution $\Q^*$ to the optimization problem $\underset{\Q \in \shk}{\inf} \shj(\Q, \P)$.
	\end{enumerate}
	Property \ref{item:selec} below deals with a progressively measurable functional $\beta : [0, T] \times \Omega \rightarrow \R^d$
        in relation to the dynamics
        \begin{equation*}
        	X_t = X_0 + \int_0^t \beta(r, X)dr + M_t + \left(q\1_{\{|q| > 1\}}\right)*\mu^X_t + \left(q\1_{\{|q| \le 1\}}\right)*\left(\mu^X - \mu^L\right)_t.
        \end{equation*}
	\begin{enumerate}[label = (Selec)]
        \item \label{item:selec} Let $\sht : (t, X) \mapsto \scrs(t, X_t, \beta(t, X))$ be the correspondence (see Definition \ref{def:correspondence}) associating to $(t, X)$ the solution set $\scrs(t, X_t, \beta(t, X))$ introduced in Definition \ref{def:solutionSet}.
         There exists a progressively measurable
         selector $(t, X) \mapsto u^*(t, X)$ of the correspondence $\sht$
         from $[0,T] \times \Omega$ to $\U$,
         such that
         the following holds.
         \begin{itemize}
           \item The triplet $(b(\cdot, \cdot, u^*(\cdot, X)), \Sigma(\cdot, X_\cdot), L)$ verifies the flow existence property in the sense of Definition \ref{def:flowExistence} with associated class $(\P^{s, \omega})_{(s, \omega) \in [0, T] \times \Omega}$.
             \item 
               The martingale problem defined in
 \eqref{eq:martProbLin}
                        with characteristics $(b,a,L)$  
with
$b_r = b(r,X_r,u^*(r,X))$, $a_r = \Sigma(r,X_r)$,
admits uniqueness.
          \end{itemize}
          We will denote $\P^* : = \int \eta_0(dx) \P^{0,\omega_x}$,
          where $\omega_x \equiv x$.
          In particular $\P^*= \P^{u^*}$
          is the unique solution of the martingale  problem  \eqref{eq:decompPuMart}
  with $u = u^*$.
\end{enumerate}
        
	\begin{remark} \label{rmk:MinSelec}
		\begin{enumerate}
                \item 
        Uniqueness of a solution to the minimization problem $\underset{\Q \in \shk}{\inf} \shj(\Q, \P)$, requested in
        Property \ref{item:min},
        is ensured by the strict convexity of the objective function $\Q \in \shk \mapsto \shj(\Q, \P)$, see Remark \ref{rmk:relativeEntropy} item 2.
        For the existence we refer to Section \ref{sec:convergenceCases}.

			\item Under Hypothesis \ref{hyp:MVIPointwise} the solution set $\scrs(t, X_t, \beta(t, X))$ is non empty for all $(t, X) \in [0, T] \times \Omega$.
			
			\item Verifying the validity of Property \ref{item:selec} consists in proving two steps:
                          the existence of a measurable selector $\bar u$
of the correspondence $\sht$
and a flow well-posedness of the martingale problem
defined in
 \eqref{eq:martProbLin}
                        with characteristics $(b,a,L)$  
with
$b_r = b(r,X_r,\bar u(r,X))$, $a_r = \Sigma(r,X_r)$.

The first step can be performed using classical measurable selection arguments (see \ref{sec:measurableSelection}),
whereas the second step  can be done via standard results on well-posedness of Markovian martingale problems, see Section \ref{sec:wellPosed}.

\item Let $\beta^*$ verifying \ref{item:selec} and with
  associated selector $u^*$.
  Then for all $t \in [0, T], u \in \U $ we have
			\begin{equation}
				\label{eq:mixedVarHypothesis}
				f(t, X_t, u) - f(t, X_t, u^*(t, X)) + \frac{1}{\epsilon}\langle \Sigma^{-1}_t(b(t, X_t, u^*(t, X)) - \beta^*(t, X)), (b(t, X_t, u) - b(t, X_t, u^*(t, X)))\rangle \ge 0.
			\end{equation} 
		\end{enumerate}
	\end{remark}

	Properties \ref{item:min} and \ref{item:selec} allow to introduce the structural \textit{Stability Condition} verified by a subset $\scrp$ of $\shp_\U(\eta_0)$ associated to the initial optimization Problem \eqref{eq:introInitialProblemJump}.
	\begin{condition}
	(Stability Condition).
	\label{cond:stability}
	$\scrp \subset \shp_\U(\eta_0)$ verifies the Stability Condition if for any $\P \in \scrp$ we have the following.
	\begin{enumerate}
		\item $\P$ verifies Property \ref{item:min}.
	\end{enumerate}
	Let $\Q^* := \underset{\Q \in \shk}{\argmin} \shj(\Q, \P)$ be the associated probability measure.
	Since $\shj(\Q^*, \P)$ is finite, then
	$H(\Q^*\vert \P)$ is finite so that
	$(\P, \Q^*) \in \sha$.
	In particular we can apply Lemma \ref{lemma:decompQ} item 1. Let 
	then $\beta^* := \beta^{\Q^*}$ be given
	by \eqref{eq:decompQ}
	applied with $\Q = \Q^*$.
	\begin{enumerate}[start=2]		
		\item We suppose $\beta^*$ verifies \ref{item:selec} with
		selector $u^*$ with associated probability measure $\P^*:= \P^{u^*}$, in relation to the dynamics described in  Definition \ref{def:Pu} with $u = u^*$.
		
		\item
		$\P^*$    belongs to $\scrp$.
		
	\end{enumerate}
\end{condition}
	\begin{remark}
          Condition \ref{cond:stability} is called \it{Stability Condition} because, starting from a probability measure $\P \in \scrp$, we can build a new probability measure $\P^*$
          still belonging to $\scrp$. We will provide in Section \ref{sec:convergenceCases} conditions on the features of Problem \eqref{eq:introInitialProblemJump} such that this condition is satisfied by some typical subset $\scrp \subset \shp_\U(\eta_0)$.
	\end{remark}
	In the following, $\scrp \subset \shp_\U(\eta_0)$ is assumed to verify the Stability Condition \ref{cond:stability}.
	We then propose a construction by induction of a sequence $(\P^k, \Q^k)_{k \ge 1}$ of elements of $\sha$.
        \begin{procedure} \label{alternating-sequence} 
	\begin{itemize} 
\item Let $\P^0 \in \scrp$. Let $k \ge 0$.
        \item By the Stability Condition item 1., $\P^k$ verifies \ref{item:min} and we set $\Q^{k + 1} := \underset{\Q \in \shk}{\argmin} \shj(\Q, \P^k)$.
          We can apply Lemma \ref{lemma:decompQ} item 1. with
          $\P = \P^k$ and $\Q = \Q^{k + 1}$ since
          $(\P^k, \Q^{k + 1}) \in \sha$ and we set $\beta^{k + 1} :=
          \beta^{\Q^{k + 1}}$, where $\beta^{\Q^{k + 1}}$ is given by \eqref{eq:decompQ} applied with $\Q = \Q^{k + 1}$.
        \item By the Stability Condition item 2., $\beta^{k + 1}$ verifies \ref{item:selec} and we consider the associated selector $u^{k+1} (= u^*)$ and the associated probability measure $\P^{k + 1} (= \P^*) \in \scrp$.
	\end{itemize}
      \end{procedure}

	\begin{remark} \label{rmk:Selec}
          The uniqueness of $\P^{k + 1} $, as solution of
          the martingale problem defined in
 \eqref{eq:martProbLin}
                        with characteristics $(b,a,L)$  
with
$b_r = b(r,X_r,u^{k+1}(r,X))$, $a_r = \Sigma(r,X_r)$,
          is not necessary to build a sequence $(\P^k, \Q^k)_{k \ge 1}$ as above, but it
            will play a crucial role in the proof of Theorem \ref{th:convergenceSequence} below.
	\end{remark}
	\begin{theorem}
		\label{th:convergenceSequence}
		Assume Hypotheses \ref{hyp:runningCostJumps} and \ref{hyp:boundedCoef}. Let $\scrp \subset \shp_\U(\eta_0)$ verifying the Stability Condition \ref{cond:stability}. Let $\P^0 \in \scrp$ and let $(\P^k, \Q^k)_{k \ge 1}$ be a sequence of elements of $\sha$
constructed by Procedure \ref{alternating-sequence}. 
                 Then $\shj(\Q^k, \P^k) \underset{k \rightarrow + \infty}{\searrow} \shj^*$, where $\shj^* = \shj^*_\epsilon$
was defined in                 \eqref{eq:introPenalizedProblemJump}
                
	\end{theorem}

	The proof of Theorem \ref{th:convergenceSequence} requires several
        preliminary results.
	\begin{lemma}
		\label{lemma:3PointsJump}
		(Three points property).
		Let $\P \in \shp_\U(\eta_0)$   verifying \ref{item:min} and let $\Q^* := \underset{\Q \in \shk}{\argmin} \shj(\Q, \P)$ be the associated probability measure.
		For all $\Q \in \shk$,
		\begin{equation}
			\label{eq:3PointsJump}
			\frac{1}{\epsilon}H(\Q | \Q^*) + \shj(\Q^*, \P) \le \shj(\Q, \P).
		\end{equation}
	\end{lemma}
	\begin{proof}
          Let $\Q \in \shk$. If $\shj(\Q, \P) = + \infty$, the inequality \eqref{eq:3PointsJump} is trivially verified. Assume now that $\shj(\Q, \P) < + \infty$. In particular, $H(\Q | \P) < + \infty$
          and therefore $\Q \ll \P$.
          Set $\vphi(X) := \int_0^T f(r, X_r, u^{\P}(r,X)) dr,$ where
$ \nu^\P_r$ and $u^{\P}$ are related as in
 Definition \ref{def:Pu}. We also
          define the probability measure $\tilde \P \in \shp(\Omega)$ by
		$$
		d\tilde \P := \frac{\exp(- \epsilon\vphi(X))}{\E^{\P}\left[\exp(-\epsilon\vphi(X))\right]}d\P.
		$$
		We first prove that $\Q^* = \underset{\tilde \Q \in \shk}{\argmin} ~ H(\tilde \Q | \tilde \P)$. Indeed, since $\tilde \P \sim \P$, for $\tilde \Q \ll \tilde \P$ we have
		$$
		\log \frac{d\tilde\Q}{d\tilde \P} = \log \frac{d\tilde\Q}{d\P} + \log \frac{d\P}{d\tilde \P} = \log \frac{d\tilde\Q}{d\P} + \epsilon \vphi(X) + \log \E^{\P}\left[\exp(-\epsilon\vphi(X))\right] \quad \tilde \Q\text{-a.s.}
		$$
                since $\P$-a.s. (and therefore  $\tilde \P$ a.s.) 
Taking the expectation under $\tilde \Q$ in the previous equality yields
		\begin{equation}
			\label{eq:3PointsInter}
			\begin{aligned}
				H(\tilde \Q | \tilde \P) & = H(\tilde \Q | \P) + \epsilon\E^{\tilde \Q}[\vphi(X)] + \log \E^{\P}\left[\exp(-\epsilon\vphi(X))\right]\\
				& = \epsilon \shj(\tilde \Q, \P) + \log \E^{\P}\left[\exp(-\epsilon\vphi(X))  \right].
			\end{aligned}
                      \end{equation}
Setting $\tilde \Q = \Q^*$ in \eqref{eq:3PointsInter}
and taking into account the fact that $\Q^* = \underset{\tilde \Q \in \shk}{\argmin} ~ \shj(\tilde \Q , \P)$,
\eqref{eq:3PointsInter} implies immediately
that $\Q^* = \underset{\tilde \Q \in \shk}{\argmin} ~ H(\tilde \Q | \tilde \P)$.
                Then by Theorem 2.2 in \cite{CsiszarEntropy} applied with $I = H$, $\she = \shk$, $R = \tilde \P$, $Q = \Q^*$ and $P = \Q$, we have for all $\Q \in \shk$,
                \begin{equation}
			\label{eq:projEntropy}
			H(\Q | \Q^*) + H(\Q^* | \tilde \P) \le H(\Q | \tilde \P).
		\end{equation}
		Replacing $H(\Q^* | \tilde \P)$ and $H(\Q | \tilde \P)$ in \eqref{eq:projEntropy} by the expression given by \eqref{eq:3PointsInter} applied with $\tilde \Q = \Q^*$ and $\tilde \Q = \Q$ respectively, we have
		$$
		H(\Q | \Q^*) + \epsilon \shj(\Q^*, \P) \le \epsilon \shj(\Q, \P),
		$$
		and we get \eqref{eq:3PointsJump} by dividing each side of the previous inequality by $\epsilon > 0$.
	\end{proof}
        We introduce some introductory lines to the {\it Four points property}
Lemma \ref{lemma:4PointsJump}.
Let any $\P \in \scrp$; let $\Q^* \in \shk$ and $\P^* \in \scrp$ defined
via  items 1. and 2. of the Stability Condition \ref{cond:stability}
combining
        \ref{item:min} and
        \ref{item:selec}.
        Consider $\beta^*$ verifying  \ref{item:selec} with selector $u^*$.
        We recall in particular that,
under $\Q^*$ the canonical process decomposes as
	\begin{equation}
		\label{eq:decompQk}
		X_t = X_0 + \int_0^t \beta^*(r,X)dr + \left(q\1_{\{|q| \le 1\}}(Y^{\Q^*} - 1)\right)*\mu^L_t +  M_t^{\Q^*} + \left(q\1_{\{|q| > 1\}}\right)*\mu^X_t + \left(q\1_{\{|q| \le 1\}}\right)*(\mu^X - (Y^{\Q^*}.\mu^L))_t,
                	\end{equation}
	where $M^{\Q^*}$ is a continuous $\Q^*$-local martingale verifying $[M^{\Q^*}] = \int_0^\cdot \Sigma(r, X_r)dr$ and
 $Y^{\Q^*} : [0, T] \times \Omega \times \R^d \rightarrow \R^+$ is $\tilde \shp$-measurable.     
 Moreover for all $(t, u) \in [0, T] \times \U$, remarking that
 $u^*(t,X_t) \in \scrs(t,X_t,\beta^*(t,X))$,
	\begin{equation}
		\label{eq:mixedVarIneq}
		f(t, X_t, u) - f(t, X_t, u^*(t, X)) + \frac{1}{\epsilon}\langle \sigma^{-1}_t(b^*_t - \beta^*_t), \sigma^{-1}_t(b^u_t - b^*_t)\rangle \ge 0,
	\end{equation}
	\begin{equation}
		\label{eq:simplifiedNotations2}
		\sigma_t^{-1} := \sigma^{-1}(t, X_t), \quad b^*_t := b(t, X_t, u^*(t, X)), \quad \text{and} \quad b^u_t := b(t, X_t, u),
              \end{equation}
              and   $\sigma^{-1}$ is the
generalized inverse
              defined by \eqref{eq:defSigma}.
	Besides, under $\P^*$ the canonical process has decomposition
	\begin{equation}
		\label{eq:decompPk}
		X_t = X_0 + \int_0^t b_r^*dr + M^{\P^*}_t + \left(q\1_{\{|q| > 1\}}\right)*\mu^X_t + \left(q\1_{\{|q| \le 1\}}\right)*(\mu^X - \mu^L)_t.
	\end{equation}
	We now state some entropic estimates.
	\begin{lemma}
		\label{lemma:entropyEstimates}
		Let $\Q$ (resp. $\Q^*$) with decomposition \eqref{eq:decompQ} (resp. \eqref{eq:decompQk}). 
                Then
		\begin{equation}
			\label{eq:entropyQQStar}
			H(\Q | \Q^*) \ge \frac{1}{2}\E^\Q\left[\int_0^T |\sigma^{- 1}_r(\beta^\Q_r - \beta^*(r,X))|^2dr\right].
		\end{equation}
	\end{lemma}

	\begin{proof}
          We can suppose    $H(\Q | \Q^*) < + \infty$, otherwise previous inequality is trivial. At this point
          we apply Theorem \ref{th:entropyJumps}  with $\P = \Q^*$
          and
        $b_r = \beta^*(r,X)$, 
        taking into account \eqref{eq:decompQk} instead of \eqref{eq:decompPEntropy}
          so that $\mu^L$ in \eqref{eq:decompPEntropy} becomes   $Y^{\Q^*} . \mu^L$.
This  provides the existence of a progressively measurable process $\alpha$ and a $\tilde \shp$-measurable function $Y$ such that under $\Q$
          the canonical process decomposes as
		\begin{equation}
			\label{eq:decompQfromQstar}
			\begin{aligned}
				X_t & = X_0 + \int_0^t \beta^*(r,X)dr + \left(q\1_{\{|q| \le 1\}}(Y^{\Q^*} - 1)\right)*\mu^L_t + \int_0^t \Sigma_r\alpha_rdr + \left(q\1_{\{|q| \le 1\}}(Y - 1)\right)*(Y^{\Q^*}\mu^L)_t\\
				& + \bar M_t + \left(q\1_{\{|q| > 1\}}\right)*\mu^X_t + \left(q\1_{\{|q| \le 1\}}\right)*(\mu^X - (YY^{\Q^*}).\mu^L)_t,
			\end{aligned}
                      \end{equation}
$\bar M$ being a continuous $\Q$-local martingale satisfying $[\bar M] = \int_0^\cdot \Sigma_rdr$. Moreover, \eqref{eq:lowerBoundEntropy} together with the positivity of $x \mapsto x\log(x) - x + 1$ yields
		\begin{equation}
			\label{eq:tempEntropQQstar}
			H(\Q | \Q^*) \ge \frac{1}{2}\E^\Q\left[\int_0^T \alpha_r^\top \Sigma_r \alpha_rdr\right] = \frac{1}{2}\E^\Q\left[\int_0^T |\sigma_r^\top\alpha_r|^2dr\right].
		\end{equation}
		Comparing decompositions \eqref{eq:decompQ} and \eqref{eq:decompQfromQstar}, uniqueness of the characteristics of a semimartingale (under $\Q$) with respect to the truncation function $q \mapsto q \1_{\{|q| \le 1\}}$ yields successively $Y^\Q = Y Y^{\Q^*}  $ and $\beta_t^* = \beta^\Q + \Sigma_r \alpha_r$ $dt \otimes d\Q$-a.e. It follows that $\sigma^\top_t\alpha_t = \beta_t^* - \beta_t^\Q$, and we get \eqref{eq:entropyQQStar} by injecting this expression in \eqref{eq:tempEntropQQstar}. 
	\end{proof}
	\begin{lemma}
		\label{lemma:4PointsJump}
                (Four points property).
                Assume Hypothesis \ref{hyp:runningCostJumps} item \ref{item:fPositive}.
                Let $\P, \Q^*,
 \P^*$ 
                as in considerations above decomposition \eqref{eq:decompQk}.
                Then for all $\Q \in \shk$
		\begin{equation}
			\label{eq:4PointsJump}
			\shj(\Q, \P^*) \le \frac{1}{\epsilon}H(\Q | \Q^*) + \shj(\Q, \P).
		\end{equation}
	\end{lemma}
	\begin{proof}
          Inequality \eqref{eq:4PointsJump} is trivial if either $H(\Q | \Q^*) = + \infty$ or $\shj(\Q, \P) = + \infty$ and we assume for the rest of the proof that $H(\Q | \Q^*) < + \infty$ and $\shj(\Q, \P) < + \infty$
          which is equivalent to $H(\Q | \P) < +\infty$. This implies that $(\P, \Q) \in \sha$.

          For the probability $\P$ we take into account the decomposition 
          \eqref{eq:decompPu} and notation \eqref{eq:simplifiedNotations1}.
          Concerning probability $\P^*$
          we keep in mind the decomposition  \eqref{eq:decompPuMart}
          with $u = u^*$, so that
          \begin{equation}
			\label{eq:decompPuMartustar}
			\left\{
			\begin{aligned}
				& X_t = X_0 + \int_0^t b(r, X_r, u^*(r, X))dr + M_t^\P + \left(q\1_{\{|q| > 1\}}\right)*\mu^X_t + \left(q\1_{\{|q| \le 1\}}\right)*(\mu^X - \mu^L)_t\\
				& X_0 \sim \eta_0.
			\end{aligned}
			\right.
		\end{equation}
              We also remind decomposition   
 \eqref{eq:decompQ} associated with $\Q$
 and decomposition \eqref{eq:decompQk} related to $\Q^*$.

          		Applying Lemma \ref{lemma:4PointsIneq} with 
		$a = \sigma^{-1}_r\beta^\Q_r$, $b = \sigma^{-1}_rb^\P_r$, $c = \sigma^{-1}_rb^*_r$, $d = \sigma^{-1}_r\beta_r^*$,
		we have for all $r \in [0, T]$
		\begin{equation*}
			\begin{aligned}
			& \frac{1}{2}|\sigma^{-1}_r(\beta^\Q_r - b^\P_r)|^2 - \frac{1}{2}|\sigma^{-1}_r(\beta^\Q_r - b^*_r)|^2 + \frac{1}{2}|\sigma^{-1}_r(\beta^\Q_r - \beta^*_r)|^2\\
			\ge & \langle \sigma^{- 1}_r(b^*_r - b_r^\P), \sigma^{- 1}_r(\beta^*_r - b^*_r)\rangle.
			\end{aligned}
		\end{equation*}
		This inequality then yields
		\begin{equation*}
			\begin{aligned}
				 f(r, X_r, \nu_r^\P) &+ \frac{1}{2\epsilon}|\sigma^{-1}_r(\beta^\Q_r - b^\P_r)|^2 - f(r, X_r, u^*_r) - \frac{1}{2\epsilon}|\sigma^{-1}_r(\beta^\Q_r - b^*_r)|^2 + \frac{1}{2\epsilon}|\sigma^{-1}_r(\beta^\Q_r - \beta^*_r)|^2\\
				& \ge  f(r, X_r, \nu_r^\P) - f(r, X_r, u^*_r) + \frac{1}{\epsilon}\langle\sigma^{- 1}_r(b^*_r - b_r^\P), \sigma^{- 1}_r(\beta^*_r - b^*_r)\rangle,
			\end{aligned}
		\end{equation*}
		and \eqref{eq:mixedVarIneq} applied with $u = \nu_r^\P$ implies that the right-hand term in the above inequality is non-negative, which gives
		\begin{equation}
			\label{eq:pointwiseIneq}
			f(r, X_r, \nu_r^\P) + \frac{1}{2\epsilon}|\sigma^{-1}_r(\beta^\Q_r - b^\P_r)|^2 + \frac{1}{2\epsilon}|\sigma^{-1}_r(\beta^\Q_r - \beta^*_r)|^2 \ge f(r, X_r, u^*_r) + \frac{1}{2\epsilon}|\sigma^{-1}_r(\beta^\Q_r - b^*_r)|^2.
                      \end{equation}
      Integrating each member of the inequality \eqref{eq:pointwiseIneq} between $0$ and $T$ and taking the expectation under $\Q$, we get
		\begin{equation}
			\label{eq:pointwiseIntegral}
			\begin{aligned}
                          \E^\Q\left[\int_0^Tf(r, X_r, \nu_r^\P)dr\right] &+ \frac{1}{2\epsilon}
             \E^\Q\left[\int_0^T|\sigma^{-1}_r(\beta^\Q_r - b^\P_r)|^2dr \right] + \frac{1}{2\epsilon}\E^\Q\left[\int_0^T|\sigma^{-1}_r(\beta^\Q_r - \beta^*_r)|^2dr \right]\\
      &\ge  \E^\Q\left[\int_0^Tf(r, X_r, u^*_r)dr\right] + \frac{1}{2\epsilon}\E^\Q\left[\int_0^T|\sigma^{-1}_r(\beta^\Q_r - b^*_r)|^2dr \right].
			\end{aligned}
		\end{equation}
	Since $H(\Q | \P)  < + \infty$, Lemma \ref{lemma:decompQ} item 2. implies that 
	$$
	H(\Q_0 | \P_0) + \E^\Q\left[(Y^\Q\log(Y^\Q) - Y^\Q + 1)*\mu^L_T\right] < + \infty.
	$$
	We can then add $\frac{1}{\epsilon}H(\Q_0 | \P_0) + \frac{1}{\epsilon}\E^\Q\left[(Y^\Q\log(Y^\Q) - Y^\Q + 1)*\mu^L_T\right]$ to each side of the inequality \eqref{eq:pointwiseIntegral}, so that
	\begin{equation}
		\label{eq:pointwiseIntegral2}
		\begin{aligned}
                  & \E^\Q\left[\int_0^Tf(r, X_r, \nu_r^\P)dr\right] + \frac{1}{\epsilon}\left(H(\Q_0 | \P_0) +
                    \frac{1}{2}\E^\Q\left[\int_0^T|\sigma^{-1}_r
                    (\beta^\Q_r - b^\P_r)\vert^2dr \right]
               + \E^\Q\left[(Y^\Q\log(Y^\Q) - Y^\Q + 1)*\mu^L_T\right]\right) \\
			& + \frac{1}{2\epsilon}\E^\Q\left[\int_0^T|\sigma^{-1}_r(\beta^\Q_r - \beta^*_r)|^2dr\right]\\
			\ge & \E^\Q\left[\int_0^Tf(r, X_r, u^*_r)dr\right] + \frac{1}{\epsilon}\left(H(\Q_0 | \P_0) + \frac{1}{2}\E^\Q\left[\int_0^T|\sigma^{-1}_r(\beta^\Q_r - b^*_r)|^2dr \right] +\E^\Q\left[(Y^\Q\log(Y^\Q) - Y^\Q + 1)*\mu^L_T\right] \right).
		\end{aligned}
	\end{equation}
		On the one hand, again by Lemma \ref{lemma:decompQ} item 2. we have that
		\begin{equation}
			\label{eq:4PointsEntropy1}
			\begin{aligned}
				\shj(\Q, \P) & = \E^\Q\left[\int_0^T f(r, X_r, \nu_r^\P)dr\right] + \frac{1}{\epsilon}H(\Q | \P)\\
				& \ge \E^\Q\left[\int_0^T f(r, X_r, \nu_r^\P)dr\right] + \frac{1}{\epsilon}\left(H(\Q_0 | \P_0) +  \frac{1}{2}\E^{\Q}\left[\int_0^T|\sigma^{-1}_r(\beta^\Q_r - b_r^\P)|^2dr\right]\right. \\
				& +  \left.\E^\Q\left[(Y^\Q\log(Y^\Q) - Y^\Q + 1)*\mu^L_T\right]\right).
			\end{aligned}
		\end{equation}
		On the other hand,
                by Lemma \ref{lemma:entropyEstimates} we have
		\begin{equation}
			\label{eq:applicationEntropyEstimate}
                        \frac{1}{\epsilon} H(\Q | \Q^*) \ge \frac{1}{2 \epsilon}
                        \E^\Q\left[\int_0^T |\sigma^{- 1}_r(\beta^\Q_r - \beta^*_r)|^2dr\right].
		\end{equation}
                Applying first \eqref{eq:4PointsEntropy1} and \eqref{eq:applicationEntropyEstimate} and then
                \eqref{eq:pointwiseIntegral2}
                yields
		\begin{equation}
			\label{eq:almost4Points}
			\begin{aligned}
				\shj(\Q, \P) + \frac{1}{\epsilon}H(\Q | \Q^*) & \ge \E^\Q\left[\int_0^Tf(r, X_r, u^*_r)dr\right] + \frac{1}{\epsilon}\left(H(\Q_0 | \P_0) + \frac{1}{2}\E^\Q\left[\int_0^T|\sigma^{-1}_r(\beta^\Q_r - b^*_r)|^2dr\right]\right.\\
				& + \left.\E^\Q\left[(Y^\Q\log(Y^\Q) - Y^\Q + 1)*\mu^L_T\right] \right).
			\end{aligned}
		\end{equation}
		Since $H(\Q_0 | \P_0) = H(\Q_0 | \P^*_0) = H(\Q_0 | \eta_0)$, $\shj(\Q, \P) < + \infty$ and $H(\Q | \Q^*) < + \infty$, inequality \eqref{eq:almost4Points} implies in particular that
		\begin{equation}
			\label{eq:condCorollary}
			H(\Q_0 | \P^*_0) + \frac{1}{2}\E^{\Q}\left[\int_0^T|\sigma^{-1}_r(\beta^\Q_r - b_r^*)|^2dr\right]
			+ \E^\Q\left[(Y^\Q\log(Y^\Q) - Y^\Q + 1)*\mu^L_T\right] < + \infty.
		\end{equation}
		Recall that, by Property \ref{item:selec}, the triplet $(b(\cdot, \cdot, u^*(\cdot, X)), \Sigma(\cdot, X_\cdot), L)$ has the flow existence property and the martingale problem \eqref{eq:decompPk}.
                Moreover $\P^*$ is the unique solution  to the martingale problem
                in the sense of Definition \ref{def:martProb} with the above characteristics.
      
                This together with  \eqref{eq:condCorollary} implies the validity of Hypotheses \ref{hyp:uniqueness} and \ref{hyp:almostSureBound} of Corollary \ref{coro:entropyConverse}   with $\P^1 = \P^*$, $\P^2 = \Q$, $b^1 = b^*$,
                $b^2 = \beta^\Q + \int_{\R^d}(q\1_{\{|q| \le 1\}})(Y^\Q - 1)L(r, X_{r-}, dq)$, $\lambda^1 = 1$ and $\lambda^2 = Y^\Q.$
At this point, that  corollary implies
\begin{equation}
			\label{eq:entropPStar}
			H(\Q | \P^*) = H(\Q_0 | \P^*_0) + \frac{1}{2}\E^{\Q}\left[\int_0^T|\sigma^{-1}_r(\beta^\Q_r - b_r^*)|^2dr\right] + \E^\Q\left[(Y^\Q\log(Y^\Q) - Y^\Q + 1)*\mu^L_T\right].
		\end{equation}
		Combining \eqref{eq:almost4Points} and \eqref{eq:entropPStar} gives
		\begin{equation*}
			\shj(\Q, \P) + \frac{1}{\epsilon}H(\Q | \Q^*) \ge \E^\Q\left[\int_0^Tf(r, X_r, u^*_r)dr\right] + \frac{1}{\epsilon}H(\Q | \P^*) = \shj(\Q, \P^*),
		\end{equation*}
		that is \eqref{eq:4PointsJump}. This concludes the proof.
              \end{proof}

              \begin{proof}[Proof of Theorem \ref{th:convergenceSequence}.]
We decompose the proof in two steps.
                \begin{enumerate}
                \item The first step of the proof is the following.
                  Let	$(\bar \P, \bar \Q) \in \sha$.
                  Note that $\shj(\bar \Q, \bar \P) < \infty$
                  since $f$ is bounded and $(\bar \P, \bar \Q) \in \sha$.                  
                Then $\shj(\bar \Q, \P^k) < + \infty$ for all $k \ge 1$.

            Indeed let $k \ge 1$.
          We apply decomposition \eqref{eq:decompQ} in Lemma \ref{lemma:decompQ} with  $\P = \bar \P$ and $\Q = \bar \Q$. On the one hand,
          setting $b^k_t := b(t,X_t,u^k)(t,X)$, we have
		\begin{equation}
			\label{eq:expDriftFiniteInt}
			\begin{aligned}
				\E^{\bar \Q}\left[\int_0^T |\sigma^{-1}_r(\beta_r^{\bar \Q} - b_r^k)|^2dr\right] & \le 2\E^{\bar \Q}\left[\int_0^T |\sigma_r^{-1}(\beta_r^{\bar \Q} - b_r^{\bar \P})|^2dr\right] + 2\E^{\bar \Q}\left[\int_0^T |\sigma^{-1}_r(b_r^{\bar \P} - b_r^k)|^2dr\right]\\
				& \le 2\E^{\bar \Q}\left[\int_0^T |\sigma_r^{-1}(\beta_r^{\bar \Q} - b_r^{\bar \P})|^2dr\right] + \frac{4T}{c_\sigma}|b|_\infty,
			\end{aligned}
		\end{equation}
		and it follows from \eqref{eq:expDriftFiniteInt} and Lemma \ref{lemma:decompQ} item 3. that
		\begin{equation}
			\label{eq:expDriftFinite}
			\E^{\bar \Q}\left[\int_0^T |\sigma^{-1}_r(\beta_r^{\bar \Q} - b_r^k)|^2dr\right] < + \infty.
		\end{equation}
		On the other hand, by item 2. of the aforementioned lemma
		\begin{equation}
			\label{eq:expJumpFinite}
			H(\bar \Q_0 | \P_0) + \E^{\bar \Q}\left[(Y^{\bar \Q}\log(Y^{\bar \Q}) - Y^{\bar \Q} + 1)*\mu^L_T\right] < + \infty.
		\end{equation}
		Since $H(\bar \Q_0 | \P_0) = H(\bar \Q_0 | \P^k_0) = H(\bar \Q_0 | \eta_0)$, inequalities \eqref{eq:expDriftFinite} and \eqref{eq:expJumpFinite} yields
		\begin{equation}
			\label{eq:condCorollary2}
			H(\bar \Q_0 | \P^k_0) + \frac{1}{2}\E^{\bar \Q}\left[\int_0^T|\sigma^{-1}_r(\beta^{\bar \Q}_r - b_r^k)|^2dr\right]
			+ \E^{\bar \Q}\left[(Y^{\bar \Q}\log(Y^{\bar \Q}) - Y^{\bar \Q} + 1)*\mu^L_T\right] < + \infty.
		\end{equation}
		Recall that, by Procedure \ref{alternating-sequence} and Property \ref{item:selec},
                $\P^k$ is a unique solution of the martingale problem   in the sense of Definition \ref{def:martProb} with characteristics
                $(b(t,X_t,u^k(t,X)), \Sigma(t,X_t), L)$.
                This together with  \eqref{eq:condCorollary2} implies the validity of Hypotheses \ref{hyp:uniqueness} and \ref{hyp:almostSureBound}
		of Corollary \ref{coro:entropyConverse}   with $\P^1 = \P^k$, $\P^2 = \bar \Q$, $b^1 = b^k$, $b^2 = \beta^{\bar \Q} + \int_{\R^d}(q\1_{\{|q| \le 1\}})(Y^{\bar \Q} - 1)L(r, X_{r-}, dq)$, $\lambda^1 = 1$ and $\lambda^2 = Y^{\bar \Q}.$ Hence $H(\bar \Q | \P^k) < + \infty$, and we conclude that $\shj(\bar \Q, \P^k) < + \infty$ since $f$ is bounded.

              \item Let us conclude the proof of
                Theorem \ref{th:convergenceSequence}.
We keep in mind the alternating sequence that we have constructed in Procedure \ref{alternating-sequence}.
We fix $k \in \N^*$,  $\P^k \in \scrp$ and we  apply Lemmata \ref{lemma:3PointsJump} and \ref{lemma:4PointsJump} to $\P = \P^k$, $\Q^* = \Q^{k + 1}$ and $\P^* = \P^{k + 1}$. More precisely, Lemma \ref{lemma:3PointsJump} applied with $\P = \P^k \in \scrp$ and $\Q^* = \Q^{k + 1}$ yields
		\begin{equation}
			\label{eq:3PointsSequence}
			\frac{1}{\epsilon}H(\Q | \Q^{k + 1}) + \shj(\Q^{k + 1}, \P^k) \le \shj(\Q, \P^k),
		\end{equation}
		for all $\Q \in \shk$, whereas by Lemma \ref{lemma:4PointsJump} applied with $\Q^* = \Q^{k + 1}$ and $\P^* = \P^{k + 1}$ we have
		\begin{equation}
			\label{eq:4PointsSequence}
			\shj(\Q, \P^{k + 1}) \le \frac{1}{\epsilon}H(\Q | \Q^{k + 1}) + \shj(\Q, \P),
		\end{equation}
		for all $(\P, \Q) \in \sha$.
		We recall that,
by the first item of Procedure \ref{alternating-sequence},
                $(\P^k,\Q^{k+1}) \in \sha.$ Evaluating \eqref{eq:4PointsSequence} in $\Q = \Q^{k + 1}$ and $\P = \P^k$ yields
 $ \shj(\Q^{k + 1}, \P^{k + 1}) \le \shj(\Q^{k + 1}, \P^k),$
  whereas by definition of $\Q^{k + 1}$, $\shj(\Q^{k + 1}, \P^k) \le \shj(\Q^k, \P^k)$. Hence
  \begin{equation} \label{eq:gendarme}
    \begin{aligned}
      \shj(\Q^{k + 1}, \P^{k + 1}) & \le \frac{1}{\epsilon} H(\Q^{k+1},\Q^{k+1})
                                     + \shj(\Q^{k + 1}, \P^k) =  \shj(\Q^{k + 1}, \P^k) \\
                                   &\le \shj(\Q^k, \P^k),
    \end{aligned}
      \end{equation}
and therefore the sequence $(\shj(\Q^k, \P^k))_{k \ge 0}$ is decreasing. Since $\shj(\Q, \P) \ge 0$ for all $(\P, \Q) \in \sha$, the sequence $(\shj(\Q^k, \P^k))_{k \ge 0}$ converges towards a limit $\ell \ge 0$.
By \eqref{eq:gendarme} the sequence $(\shj(\Q^{k + 1}, \P^k))_{k \ge 1}$
converges towards the same limit $\ell$. 
Let now $\gamma > 0$ and $(\bar \P, \bar \Q) \in \sha$
such that $\shj(\bar \Q, \bar \P) \le \shj^* + \gamma$, which exists by the definition of infimum.
By Step 1. of this proof,
 $\shj(\bar \Q, \P^k) < + \infty$ for all $k \ge 0$. At this point inequality \eqref{eq:3PointsSequence} applied with $\Q = \bar \Q$ yields $\frac{1}{\epsilon}H(\bar \Q | \Q^{k + 1}) < + \infty$ for all $k \ge 0$. We then apply successively \eqref{eq:4PointsSequence} and \eqref{eq:3PointsSequence} to $\Q = \bar \Q$ and $\P = \bar \P$ to obtain for every $k \ge 0$
\begin{equation}
	\label{eq:liminf}
	\begin{aligned}
		\shj(\bar \Q, \P^{k + 1}) &\le \shj(\bar \Q, \P^k) - \shj(\Q^{k + 1}, \P^k) + \shj(\bar \Q, \bar \P)\\
		& \le \shj(\bar \Q, \P^k) - \shj(\Q^{k + 1}, \P^k) + \shj^* + \gamma.
	\end{aligned}
\end{equation}
We denote $\underline{\ell} := \liminf \shj(\bar \Q, \P^k)$ which is finite since $\shj(\bar \Q, \P^k) < + \infty$ for all $k \ge 0$. Taking the $\liminf$ in \eqref{eq:liminf} yields
$$
\underline{\ell} \le \underline{\ell} - \ell + \shj^* + \gamma, \quad \text{that is} \quad \ell \le \shj^* + \gamma.
$$
Taking into account that $\gamma > 0$ is arbitrary, previous inequality implies that $\ell \le \shj^*$.
Since
the opposite inequality is trivial, we finally get $\ell = \shj^*$.
\end{enumerate}
              \end{proof}
	
	\begin{remark}
          \label{rmk:alternatingMinimization}
          The Procedure \ref{alternating-sequence} is an alternating minimization procedure in the following sense.
          \begin{itemize}
          \item          By construction, we recall that $\Q^{k + 1} = \underset{\Q \in \shk}{\argmin}~\shj(\Q, \P^k)$.
\item
  We also have $\P^{k + 1} \in \underset{\P \in \shp_\U(\eta_0)}{\argmin}~\shj(\Q^{k + 1}, \P)$.
  Indeed, by \eqref{eq:4PointsSequence} applied with $\Q = \Q^{k + 1}$, we have that $\shj(\Q^{k + 1}, \P^{k + 1}) \le \shj(\Q^{k + 1}, \P)$ for any $\P \in \shp_\U(\eta_0)$ such that $(\P, \Q_{k+ 1}) \in \sha$, and previous inequality still holds if $H(\Q_{k + 1} | \P) = + \infty$ and $(\P, \Q_{k + 1}) \notin \sha$.

  \end{itemize}
	\end{remark}
	\section{Examples in the Markovian setting}
	\label{sec:examples}
	\setcounter{equation}{0}

	The construction of the minimizing sequence $(\P^k, \Q^k)_{k \ge 1}$ by Procedure \ref{alternating-sequence} includes the case where the functions $\beta^{k + 1}$ and the
        controls $u^{k + 1}$ can
       be
        a priori path-dependent. In this section we focus on the Markovian setting since, for numerical reasons, it is interesting to look
        for \textit{Markovian controls}
        $u^{k+1}$, i.e. 
        of the form $u^{k + 1} : [0, T] \times \R^d \rightarrow \U$ (with some abuse of notation)
        which depend on the controlled process at time $t$ only through its current value $X_t$. They are indeed much easier to approximate and to compute.
	We will consider instances of Problem \eqref{eq:introPenalizedProblemJump} for different sets $\shk$ and we will provide subsets $\scrp \subset \shp_\U^{Markov}(\eta_0)$ which verify the Stability Condition \ref{cond:stability} in each situation.

        We will provide in Sections \ref{sec:mixedVar}, \ref{sec:measurableSelection} and \ref{sec:wellPosed} a general framework under which \ref{item:selec} is verified for any function $\beta \in \shb([0, T] \times \R^d, \R^d)$. We
        are concerned with the verification of the property \ref{item:min} case-by-case in Section \ref{sec:convergenceCases}. From now on 
        we will suppose the following.    
        
\begin{hyp}
  
		\label{hyp:continuityCoef}
		(Continuity). The running cost $f$ and
                the drift $b$ are continuous on $[0, T] \times \R^d \times \U$.
              \end{hyp}
              This framework can of course
              be extended to the case when we have
              a terminal cost $g$, which will
              also be supposed to be continuous.
	
	\subsection{Solutions to mixed variational inequalities (MVI)}
	\label{sec:mixedVar}
        
	In this section we verify Hypothesis \ref{hyp:MVIPointwise} in the Markovian setting.
	More precisely, we are interested in finding general conditions such that for any $(t, x, \delta) \in [0, T] \times \R^d \times \R^d$,
        there exists $\bar u \in \U$ verifying \eqref{eq:MVIPointwise},
        i.e. the set $\scrs(t,x,\delta)$ is non-empty.
	
	Indeed, we will verify Hypothesis \ref{hyp:MVIPointwise} under two different sets of assumptions on $f$ and $b$. We start by the convexity assumption below.
	\begin{hyp}
		\label{hyp:KtxConvex}
		(Compact and Convex).
		$\U$ is compact and for all $(t, x) \in [0, T] \times \R^d$,
                the subset 
		\begin{equation}
			\label{eq:Ktx}
			K(t, x) := \left\{\vphantom{e^{||}} (b(t, x, u), z)~\middle | ~u \in \U,~z \ge f(t, x, u)\right\}
		\end{equation}
of $\R^d \times \R_+$	is convex.
              \end{hyp}

              \begin{remark} \label{rmk:closed}
\begin{enumerate}
\item        Hypothesis \ref{hyp:KtxConvex} is for instance fulfilled in the case where $\U$ is compact convex, $b$ is linear in the control variable $u$ and the function $f(t, x, \cdot)$ is convex in $u \in \U$ for all $(t, x) \in [0, T] \times \R^d$.
\item	For all $(t, x) \in [0, T] \times \R^d$, the set $K(t, x)$ is closed
if $\U$ is compact
and Hypothesis \ref{hyp:continuityCoef} is fulfilled.

\end{enumerate}
\end{remark}

\begin{lemma}
		\label{lemma:solutionMVIFullConvex}
		Assume Hypotheses \ref{hyp:runningCostJumps} item 1., \ref{hyp:continuityCoef} and \ref{hyp:KtxConvex}. Then for all $(t, x, \delta) \in [0, T] \times \R^d \times \R^d$, there exists $\bar u \in \U$ such that \eqref{eq:MVIPointwise}, i.e. Hypothesis \ref{hyp:MVIPointwise} is verified.
	\end{lemma}
	\begin{proof}
		Let $\delta \in \R^d$. Let $(t, x) \in [0, T] \times \R^d$. We define
		\begin{equation}
			\label{eq:barFDelta}
			(y, z) \in \R^d \times \R_+
                        \mapsto F_\delta(y, z) = z + \frac{1}{2}|\sigma^{-1}(t, x)(\delta - y)|^2.
		\end{equation}
		Set $F_\delta^{\inf} := \underset{(y, z) \in K(t, x)}{\inf}~F_\delta(y, z)$,
                which is non-negative and finite.
                Let $(y_n, z_n)_{n \ge 1}$ be a minimizing sequence of $F_\delta$ on $K(t, x)$, i.e. $F_{\delta_n}(y_n,z_n)$ converges to $F_\delta^{\inf}$.
  Since $F_{\delta_n}(y_n,z_n)$ is bounded,
 the sequence
  $(y_n, z_n)_{n \ge 1}$
  is also bounded
  taking into account
  the expression
 \eqref{eq:barFDelta}.
 Therefore $(y_n, z_n)_{n \ge 1}$ admits a converging subsequence, that we still indicate with the same notation,
 to some point $(y^*, z^*)$.
 Since $K(t, x)$ is closed, see Remark \ref{rmk:closed}, $(y^*, z^*) \in K(t, x)$.
 Since the function $F_\delta$ is continuous
		$$
		F_\delta^{\inf} = \lim_{n \rightarrow + \infty} F_\delta(y_n, z_n) = F_\delta(y^*, z^*),
		$$
                so	that  $(y^*, z^*)$ is a minimum of $F_\delta$ on $K(t, x)$.
                Lemma \ref{lemma:convexOptimality} in the Appendix,
                applied with $U = K(t, x)$, $F = F_\delta$, $g(y, z) = z$ and $h(y, z) = \frac{1}{2\epsilon}|\sigma^{-1}(t, x)(y - \delta)|^2$ then gives
		\begin{equation}
			\label{eq:varEqStep1}
			z - z^* +\frac{1}{\epsilon} \langle \sigma^{-1}(t, x)(y^* - \delta), \sigma^{- 1}(t, x)(y - y^*)\rangle \ge 0,
		\end{equation}
		for all $(y, z) \in K(t, x)$.
                Since $(y^*,z^*) \in K(t,x)$ there is $\bar u $ such that $ y^* = f(t,x,\bar u)$ and $z^* = f(t,x,\bar u)$
                (clearly we cannot have $z^* > f(t,x,\bar u)$).
                This yields
\begin{equation}
			\label{eq:varEqStep2}
			z - f(t, x, \bar u) +\frac{1}{\epsilon} \langle \sigma^{-1}(t, x)(b(t, x, \bar u) - \delta), \sigma^{- 1}(t, x)(y - b(t, x, \bar u))\rangle \ge 0.
		\end{equation}
		Moreover inequality \eqref{eq:varEqStep2} applies for all $(y, z) = (b(t, x, u), f(t, x, u)) \in K(t, x)$, $u \in \U$, hence \eqref{eq:MVIPointwise} is verified for all $u \in \U$.
	\end{proof}
	We exhibit a second framework where Hypothesis \ref{hyp:MVIPointwise} is verified without requiring any convexity assumption on $f$.
	\begin{hyp}
		\label{hyp:linear}
		$\U$ is compact convex and for all $(t, x, u) \in [0, T] \times \R^d \times \U$, $b(t, x, u) = \gamma(t, x) + u$, where $\gamma \in \shb([0, T] \times \R^d, \R^d)$.
	\end{hyp}
	\begin{lemma}
		\label{lemma:solutionMVILinear}
		Assume Hypotheses \ref{hyp:runningCostJumps} item 1., \ref{hyp:continuityCoef} and \ref{hyp:linear}.
                Then for all $(t, x, \delta) \in [0, T] \times \R^d \times \R^d$, there exists $\bar u \in \U$ such that
                Hypothesis \ref{hyp:MVIPointwise} related
                to (MVI) inequality is verified.
	\end{lemma}
	\begin{proof}
          Let $(t, x) \in [0, T] \times \R^d$ and let $\delta \in \R^d$.
Under Hypothesis \ref{hyp:linear}, 
           the  mixed variational inequality
                \eqref{eq:MVIPointwise}
reduces to
               	\begin{equation}
\label{eq:mixedVarIneqLinear}
f(t, x, u) - f(t, x, \bar u) + \langle \Sigma^{-1}(t, x)\bar u + \Sigma^{-1}(t, x)(\gamma(t, x) - \delta), u - \bar u \rangle \ge 0.
\end{equation}
We recall that $f$ is continuous and $ \Sigma$ is positive definite.
The validity of Hypothesis \ref{hyp:MVIPointwise} consists in showing the existence of $\bar u \in \U$ such that
for all $u \in \U$, \eqref{eq:mixedVarIneqLinear} holds.
For all $(t, x) \in [0, T] \times \R^d$, Corollary 3.2 item $(ii)$ in \cite{MixedVariationalExistence} applied with $A = \Sigma^{-1}(t, x)$, $a = \Sigma^{-1}(t, x)(\gamma(t, x) - \delta)$, $y = u$, $\bar y = \bar u$ and $h = f(t, x, \cdot)$ shows that there is a unique
$\bar u \in \U$ fulfilling \eqref{eq:mixedVarIneqLinear},
i.e. $\scrs(t,x,\delta)$ is a singleton.
	\end{proof}

	\subsection{Measurable selection}
	\label{sec:measurableSelection}
        
	In this short section, we verify the first part of property \ref{item:selec}. The result below is a consequence of Proposition \ref{prop:measurableSelectionMVI}
        and it is proved in Appendix \ref{app:measurableSelection}.
        
	\begin{lemma}
		\label{lemma:existenceImpliesMeasurable}
		Assume Hypotheses \ref{hyp:MVIPointwise} and \ref{hyp:continuityCoef}. Then for all $\delta \in \shb([0, T] \times \R^d, \R^d)$ there exists a Borel function $\bar u : [0, T] \times \R^d \rightarrow \U$ such that $\bar u(t, x)$ verifies \eqref{eq:MVIPointwise} with $\delta = \delta(t, x)$
		for all $(t, x) \in [0, T] \times \R^d$.
	\end{lemma}
\begin{remark}
		\label{rmk:pathDependentMeasSelec}
		We can extend the results of Lemma \ref{lemma:existenceImpliesMeasurable} to the path-dependent case.
                Let $\Omega_0$ be the set of couples $(t,X) \in [0,T] \times \Omega$ such that
                $X_s = X_t$, if $t\ge s$. $\Omega_0$ is a metric space equipped with the
                topology of uniform convergence.
                If $\delta :  \Omega_0 \rightarrow \R^d$ is a Borel function,
                in particular  progressively measurable
                one can show that there exists a Borel  functional
                $\bar u : \Omega_0 \rightarrow \U$ such that $\bar u(t, X)$ verifies \eqref{eq:MVIPointwise} with
$x = X_t$ and
                $\delta = \delta(t, X)$,
               for all $(t, X) \in [0, T] \times \Omega$, see Remark \ref{rmk:MVI_PathDep}.
             \end{remark}

       From Section \ref{sec:mixedVar} we deduce the result below.
	\begin{coro}
		\label{coro:mviVerified}
		Assume Hypotheses \ref{hyp:runningCostJumps} item 1. and \ref{hyp:continuityCoef}. Assume moreover either Hypothesis \ref{hyp:KtxConvex} or \ref{hyp:linear}. Then for all $\delta \in \shb([0, T] \times \R^d, \R^d)$ there exists $\bar u \in \shb([0, T] \times \R^d, \U)$ such that $\bar u(t, x)$ verifies \eqref{eq:MVIPointwise}
                with $\delta = \delta(t, x)$ for all $(t, x) \in [0, T] \times \R^d$.
	\end{coro}
	\begin{proof}
          \begin{itemize}
          \item            Under Hypothesis \ref{hyp:KtxConvex},
            the result follows by Lemma \ref{lemma:solutionMVIFullConvex}, which implies that Hypothesis \ref{hyp:MVIPointwise} is verified,
            and by Lemma \ref{lemma:existenceImpliesMeasurable}.
          \item 
            If, instead, Hypothesis \ref{hyp:linear} is verified, the result is a consequence of  Lemma \ref{lemma:solutionMVILinear} and
            Lemma \ref{lemma:existenceImpliesMeasurable}.
            \end{itemize}
          \end{proof}

	\subsection{Well-posedness of Markovian martingale problems}
	\label{sec:wellPosed}

	In this short section we focus on the second part of Property \ref{item:selec}. For technical reasons,
        given a path-dependent functional $u^*$
        it is difficult to ensure existence and uniqueness of a probability measure $\P^* \in \shp_\U(\eta_0)$
        being solution of the martingale problem
                in the sense of Definition \ref{def:martProb} with characteristics
                $(b(t,X_t,u^*(t,X)), \Sigma(t,X_t), L)$.

                Indeed,
                when $u^* \in \shb([0, T] \times \R^d, \U)$,
Proposition \ref{prop:existenceWithJumps} below provides existence and uniqueness of $\P^*$ under the following assumptions on the coefficients $b$, $\Sigma$ and $L$.
	\begin{hyp}
		\label{hyp:coefDiffJumps}
		(Jump diffusion coefficients).
		\begin{enumerate}
			\item $b$ is bounded.
		
			\item $\Sigma$ is bounded and continuous.
			
                          
			\item There exists $c_\sigma > 0$ such that for all $(t, x) \in [0, T] \times \R^d$, $\xi \in \R^d$,
			$
			\xi^\top\Sigma(t, x)\xi \ge c_\sigma |\xi|^2.
			$
			
			\item There exists a measure $L_*$ on $\shb(\R^d \backslash \{0\})$ such that $\int_{\R^d} \left(1 \wedge |q|^2\right) L_*(dq) < + \infty$ and for all $(t, x) \in [0, T] \times \R^d$, $L_*(\cdot) - L(t, x, \cdot)$ is a non-negative measure.
			
		\end{enumerate}
	\end{hyp}
	The following statement is proved in \cite{BORMArtingaleProblems}. The verification of the flow property is a consequence of existence results for jump diffusions in \cite{KomatsuJumpsDiffusion}. Concerning item 2., if follows from a disintegration with respect to $\eta_0$ and a countable characterization of the solution of a martingale problem in the sense of Definition \ref{def:martProb}.
       
        \begin{prop}
          \label{prop:existenceWithJumps}
          Assume Hypothesis \ref{hyp:coefDiffJumps}. Let $u : [0, T] \times \R^d \rightarrow \U$ be a measurable function.
We set $(b,a,L) : = (b(\cdot, \cdot, u(\cdot, X_{\cdot})), \Sigma(\cdot, X_\cdot), L)$.
   We have the following.
          \begin{enumerate}
          \item The triplet $(b,a,L)$
            verifies the flow existence property in the sense of Definition \ref{def:flowExistence} with associated path-dependent class $(\P^{s, \omega})_{(s, \omega) \in [0, T] \times \Omega}$.
          \item The martingale problem defined in
\eqref{eq:martProbLin}
                        with characteristics $(b,a,L)$  
admits uniqueness.
\end{enumerate}
	\end{prop}
          We denote $\P^u := \int_{\R^d}\eta_0(dx)\P^{0, x}$,
          where $\P^{0,x}:= \P^{0,\omega_0}$ and
          $\omega_0 \equiv x$.
          $\P^u$ solves the martingale problem with respect
          to $(b,a,L)$ in the sense of Definition
          \ref{def:martProb}.

	\subsection{Verification of the Stability Condition \ref{cond:stability}}
	\label{sec:convergenceCases}
	 In the whole section we assume Hypotheses \ref{hyp:runningCostJumps}, \ref{hyp:continuityCoef}, and \ref{hyp:coefDiffJumps}. We provide in this section sufficient conditions for the Stability Condition \ref{cond:stability} to hold in two situations: when $\shk = \shp(\Omega)$ in Theorem \ref{th:convergenceControlSto} and when $\shk = \left\{\Q \in \shp(\Omega)~:~\Q_T = \eta_T\right\}$ in Theorem \ref{th:convergenceTerminalLaw}. 
	\begin{theorem}
          \label{th:convergenceControlSto}
		Let $\shk = \shp(\Omega)$. Assume Hypotheses \ref{hyp:runningCostJumps}, \ref{hyp:continuityCoef}, \ref{hyp:coefDiffJumps} and either Hypothesis \ref{hyp:KtxConvex} or \ref{hyp:linear}.
                Then $\scrp = \shp_\U^{Markov}(\eta_0)$ verifies the Stability Condition \ref{cond:stability}.
	\end{theorem}
	Before proving Theorem \ref{th:convergenceControlSto} we consider the following lemma concerning the verification of item \ref{item:min}
and item \ref{item:selec}
        of the Stability Condition \ref{cond:stability}.
	\begin{lemma}
		\label{lemma:markovDecompositionJump}
		Let $\shk = \shp(\Omega)$.
                Let $u \in \shb([0, T] \times \R^d, \U)$ and $\eta_0 \in \shp(\R^d)$
and  $\P^u$ given  by Proposition \ref{prop:existenceWithJumps}.
             
       \begin{enumerate}
\item There exists a unique $\Q^* := \underset{\Q \in \shk}{\argmin}~\shj(\Q, \P^u)$.
			
\item 
  There exists $\beta^* = \beta^{\Q^*} \in \shb([0, T] \times \R^d, \R^d)$ and $Y^* = Y^{\Q^*}  \in \shb([0, T] \times \R^d \times \R^d, \R_+)$
  such that, under $\Q^*$, the canonical process
	\begin{equation}
		\label{eq:decompQkBis}
		\begin{aligned}
			X_t & = X_0 + \int_0^t \beta^*(r,X_r)dr + \left(q\1_{\{|q| \le 1\}}(Y^* - 1)\right)*\mu^L_t +  M_t^{\Q^*}\\
			& + \left(q\1_{\{|q| > 1\}}\right)*\mu^X_t + \left(q\1_{\{|q| \le 1\}}\right)*(\mu^X - (Y^{*}.\mu^L))_t.
		\end{aligned}
	\end{equation}

		\end{enumerate}

              \end{lemma}
	\begin{proof}
 Item 1. is given by Proposition \ref{prop:minimizerUnconstrained} applied with $\vphi(X) := \int_0^T f(r, X_r, u(r, X_r))dr$.
  Indeed, by Proposition \ref{prop:minimizerUnconstrained}, $\Q^*$ is proved to be an \textit{exponential twist} probability measure verifying formula \eqref{eq:expoTwist} with reference measure $\P = \P^u$.

  Consequently, using item 1.
  of Remark 6.3
  in \cite{BORMarkov2023} together with the fact that Hypothesis \ref{hyp:coefDiffJumps} holds allows us to apply Proposition 6.4 in \cite{BORMarkov2023} which yields item 2.
   taking into account Remark \ref{rmk:MartProbStroock},
with
$$ b_r = \beta^*(r,X_r) + \int q\1_{\{|q| \le 1\}}(Y^* - 1) L(r, X_r, dq),
\
L := Y^* L,$$
where $Y^*(t, x, q) = \frac{v(t, x + q)}{v(t, q)}$, $ k = q\1_{\{ \vert q \vert \le 1 \}}$
and $\beta^* = \frac{\Gamma(v)}{v}$ as in (6.8) 
and $v$ as in (5.4) of \cite{BORMarkov2023}.

\end{proof}

\begin{proof}
          [Proof of Theorem \ref{th:convergenceControlSto}]
          Let $\P \in \scrp  =  \shp_{\U}^{Markov}(\eta_0).$
          We consider $u = u^\P \in \shb([0, T] \times \R^d, \U)$
          such that
          $\nu^\P = u(\cdot,X_\cdot)$
          in \eqref{eq:decompPu}.
          By Lemma \ref{lemma:markovDecompositionJump} there exists a unique solution $\Q^*$ to $\underset{\Q \in \shk}{\inf} \shj(\Q, \P)$ and a  function $\beta^* := \beta^{\Q^*}$ associated to the
          decomposition \eqref{eq:decompQkBis} of the canonical process $X$ under $\Q^*$. In particular $\P$ verifies Property \ref{item:min}.
          Setting $\delta:= \beta^*$, by Corollary \ref{coro:mviVerified} there exists a function
          $u^* (= \bar u) \in \shb([0, T] \times \R^d, \U)$ such that for all $(t, x) \in [0, T] \times \R^d$, $u^*(t, x) \in \scrs(t,x,\beta^*(t,x))$.
          Proposition \ref{prop:existenceWithJumps} applied with $u = u^*$ then implies that \ref{item:selec} is verified with $\beta = \beta^*$.
          Setting $\P^* := \P^{u^*}$,
          we have $\P^* \in \scrp$, so we can conclude that $\scrp$ verifies the Stability Condition \ref{cond:stability}.
	\end{proof}
        
        The theorem below considers the case of a fixed terminal law
but it  is situated in the context of continuous processes.

\begin{theorem}
		\label{th:convergenceTerminalLaw}
		Let $\eta_0 \in \shp(\R^d)$ with a second order moment
                and $\eta_T \in \shp(\R^d)$. Let $\shk = \left\{\Q \in \shp(\Omega)~:~\Q_T = \eta_T\right\}$. Assume $L = 0$. Assume moreover Hypotheses \ref{hyp:runningCostJumps}, \ref{hyp:continuityCoef}, \ref{hyp:coefDiffJumps} and either Hypothesis \ref{hyp:KtxConvex} or \ref{hyp:linear}. Then
		\begin{equation}
			\label{eq:PMarkovTerminal}
			\scrp = \left\{\P \in \shp_\U^{Markov}(\eta_0)~:~H(\eta_T | \P_T) < + \infty\right\}
		\end{equation}
		verifies the Stability Condition \ref{cond:stability}.
	\end{theorem}
	The proof of Theorem \ref{th:convergenceTerminalLaw} requires the following lemma concerning the verification of items \ref{item:min}
and  \ref{item:selec}
        of the Stability Condition \ref{cond:stability}.
	\begin{lemma} 
          \label{lemma:markovDecompositionTerminalLaw}
           Assume $L = 0$.
           Let $\eta_0 \in \shp(\R^d)$ and $\eta_T \in \shp(\R^d)$. Set $\shk = \left\{\Q \in \shp(\Omega)~:~\Q_T = \eta_T\right\}$. Let $u \in \shb([0, T] \times \R^d, \U)$. Let $\P^u$ be the probability
           given by Proposition \ref{prop:existenceWithJumps}. Assume that $H(\eta_T | \P^u_T) < + \infty$.
		\begin{enumerate}
			\item There exists a unique $\Q^* := \underset{\Q \in \shk}{\argmin}~\shj(\Q, \P^u)$.
			
			\item
                          Under $\Q^*$,
                          there exists $\beta^* = \beta^{\Q^*} \in \shb([0, T] \times \R^d, \R^d)$ such that, the canonical process has decomposition
     	\begin{equation}
          \label{eq:decompQkTer}
		X_t = X_0 + \int_0^t \beta^*(r,X_r)dr +
                M_t^{\Q^*},
             	\end{equation}
                where $M^{\Q^*}$ is a continuous local martingale
                such that
                $[M^{\Q^*}] = \int_0^\cdot \Sigma(r,X_r) dr$.
		\end{enumerate}
              \end{lemma}

	\begin{proof}
          Item 1. is given by Proposition \ref{prop:existenceMinQmu} applied with $\vphi(X) := \int_0^T f(r, X_r, u(r, X_r))dr$ which is obviously bounded
          since $f \ge 0$ is assumed to be bounded. 
		
          Let us prove item 2.
          We observe that $L=0$ implies that  the compensator
          of $\mu^X$ vanishes so that $\mu^X$ also vanishes
          so $X$ is $\P$-a.s. continuous.
          Taking into account Remark \ref{rmk:MartProbStroock},
          under
          $\P:= \P^u$, the canonical process decomposes as
		$$
		X_t = X_0 + \int_0^t b(r, X_r, u(r, X_r))dr + M_t^\P,
		$$
		where $M^\P$ is a continuous local martingale such that $[ M^\P ] = \int_0^\cdot \Sigma(r, X_r)dr$, $u \in \shb([0, T] \times \R^d, \U)$.
                $H(\Q^* \vert \P)$ is finite recalling the considerations after the statement of item 1. of the Stability Condition
                \ref{cond:stability}.
	Now	Theorem \ref{th:entropyJumps} provides a progressively measurable process $\alpha = \alpha(\cdot, X)$
                such that under $\Q^*$ the canonical process decomposes
		\begin{equation}
			\label{eq:decompMarkovianDrift}
			X_t = X_0 + \int_0^t b(r, X_r, u(r, X_r))dr + \int_0^t \Sigma(r, X_r)\alpha(r, X)dr + M_t^{\Q^*},
		\end{equation}
		where $M^{\Q^*}$ is a martingale such that $[ M^{\Q^*}] = \int_0^\cdot \Sigma(r, X_r)dr$ and
		\begin{equation}
			\label{eq:CAlpha}
			C_\alpha := \E^{\Q^*}\left[\int_0^T |\sigma^\top(r, X_r)\alpha(r, X)|^2dr\right] < + \infty,
                      \end{equation}
                      because of \eqref{eq:lowerBoundEntropy}.

              Setting $\beta_t := b(t, X_t, u(t, X_t)) + \Sigma(t, X_t)\alpha(t, X), t \in [0, T]$,
              by items 1. and 2. of Hypothesis \ref{hyp:coefDiffJumps},
		on the one hand, \eqref{eq:CAlpha} on the other hand, we get that
		$$
		\E^{\Q^*}\left[\int_0^T |\beta_r|^2dr\right] \le C\left(T\|b\|_\infty^2 + \|\Sigma\|_\infty \E^{\Q^*}\left[\int_0^T |\sigma^\top(r, X_r)\alpha(r, X)|^2dr\right]\right) < + \infty,
		$$
		for some constant $C > 0$. In particular, $\E^{\Q^*}\left[\int_0^T |\beta_r|dr\right] < + \infty$ and Lemma \ref{lemma:nelsonDerivative} gives
		$$
		\lim_{h \downarrow 0}\E^{\Q^*}\left[\frac{X_{t + h} - X_t}{h}\middle|\shf_t\right] = \beta_t~\text{in}~L^1(\Q^*),
		$$
		for
          almost all $t \in [0, T]$. Recall that the probability measure $\P^u$ is Markovian in the sense of Definition \ref{def:markovProba}. Then by Lemma \ref{lemma:QMarkovian} $\Q^*$ is also Markovian so that
		$$
		\E^{\Q^*}\left[\frac{X_{t + h} - X_t}{h}\middle|\shf_t\right] = \E^{\Q^*}\left[\frac{X_{t + h} - X_t}{h}\middle|X_t\right].
		$$
                		It follows that, for almost all $t \in [0, T]$, $\beta_t$ is $\sigma(X_t)$-measurable, hence $\Sigma(t, X_t)\alpha(t, X) = \beta_t - b(t, X_t, u(t, X_t))$ is $\sigma(X_t)$-measurable. Let then $\Gamma \in \shb([0, T] \times \R^d, \R^d)$ be a measurable function such that $\Gamma(t, X_t) = \E^{\Q^*}[\Sigma(t, X_t)\alpha(t, X)|X_t]$ $dt\otimes d\Q^*$-a.e. The existence of $\Gamma$ is guaranteed by Proposition 5.1 in \cite{MimickingItoGeneral}. Then $\Sigma(t, X_t)\alpha(t, X) =  \Gamma(t, X_t)$ $dt\otimes d\Q^*$-a.e. Setting $\beta^*(t, X_t) := b(t, X_t, u(t, X_t)) + \Gamma(t, X_t)$, the conclusion follows from \eqref{eq:decompMarkovianDrift}.
	\end{proof}
	\begin{proof}
		[Proof of Theorem \ref{th:convergenceTerminalLaw}]

		Let $\P \in \scrp$.
                Since $\P \in \shp_\U(\eta_0)$, 
                $\P $
                verifies the martingale problem
                \eqref{eq:decompPuMart},
where $u$ is some Borel function.
By Lemma \ref{lemma:markovDecompositionTerminalLaw} there exists a unique solution $\Q^*$ to $\underset{\Q \in \shk}{\inf} \shj(\Q, \P)$
so that the property \ref{item:min} is verified.

Moreover, as mentioned after item 1. of the Stability Condition \ref{cond:stability} $H(\Q^* \vert \P) < \infty$
so that $(\P,\Q^*)$ belongs to $\sha$. 
Consequently $\shj(\Q^*, \P) < \infty$.

Concerning the property \ref{item:selec},
the same Lemma  \ref{lemma:markovDecompositionTerminalLaw}  also implies the existence
of a ''Markovian'' function $\beta^* := \beta^{\Q^*}$ associated to the decomposition \eqref{eq:decompQkTer}
of the canonical process $X$ under $\Q^*$.
Setting $\delta:= \beta^*$, by Corollary \ref{coro:mviVerified} there exists a function
$u^* (= \bar u) \in \shb([0, T] \times \R^d, \U)$ such that for all $(t, x) \in [0, T] \times \R^d$, $u^*(t, x) \in \scrs(t,x,\beta^*(t, x))$. Proposition \ref{prop:existenceWithJumps} applied with $u = u^*$ then implies that \ref{item:selec} is verified with $\beta = \beta^*$. Setting $\P^* := \P^{u^*}$, it remains to prove that $H(\eta_T | \P^*_T)$ to conclude that $\P^* \in \scrp$ and that $\scrp$ verifies the Stability condition \ref{cond:stability}.

By assumption $H(\eta_T | \P_T) < + \infty$.
We remark that the probability measures $\Q^*, \P^*, \P$ are compatible with
the context introduced before \eqref{eq:decompQk}
which is given here by \eqref{eq:decompQkTer}
since $L = 0$. 
This allows us to 
apply Lemma \ref{lemma:4PointsJump}, in particular
inequality \eqref{eq:4PointsJump}
with $\Q = \Q^*$ yielding $\shj(\Q^*, \P^*) \le \shj(\Q^*, \P) < + \infty$. Hence $\shj(\Q^*, \P^*) < + \infty$, which implies $H(\Q^* | \P^*) < + \infty$ and in particular, $H(\eta_T | \P^*_T) = H(\Q^*_T | \P^*_T) \le H(\Q^* | \P^*) < + \infty$, see Remark \ref{rmk:relativeEntropy} item 3.
		Hence $\P^* \in \scrp$ and $\scrp$ verifies the Stability Condition \ref{cond:stability}.
	\end{proof}

	\begin{remark} \label{rmk:GirsPathDep}
          One can prove that the Stability Condition \ref{cond:stability} is verified in some path-dependent settings,
          where the control
          $          u :  \Omega_0 \rightarrow \U$ is a measurable functional, so not necessarily ''Markovian''.
          Let $\eta_0 \in \shp(\R^d)$ with a second order moment and $\eta_T \in \shp(\R^d)$ and $\shk = \left\{\Q \in \shp(\Omega)~:~\Q_T = \eta_T\right\}$ as in Theorem \ref{th:convergenceTerminalLaw}.
          Assume again  $L = 0$ and either Hypothesis \ref{hyp:KtxConvex} or Hypothesis \ref{hyp:linear}.
          Using Girsanov theorem, one can easily extend Proposition \ref{prop:existenceWithJumps} to martingale problems
          associated with the path-dependent triplet $(b, a, L) = (b(\cdot, \cdot, u(\cdot, X)), \Sigma(\cdot, X_\cdot), 0)$. Setting $\scrp := \{ \P \in \shp_\U(\eta_0) \vert H(\eta_T\vert \P_T) < +\infty  \}$, we
can adapt the proof of
Theorem  \ref{th:convergenceTerminalLaw} to the path-dependent case. 
$\scrp$ can be shown to satisfy   the Stability Condition: 
in particular the Property \ref{item:min} follows by  Proposition \ref{prop:existenceMinQmu}
whereas the verification of Property  \ref{item:selec}  
 relies on Remark \ref{rmk:pathDependentMeasSelec}.
	\end{remark}

	\section{Numerics}
	\label{sec:numerics}
	\setcounter{equation}{0}

	In this section we illustrate our approach in the framework
        of stochastic control with prescribed terminal law.
	We consider the problem of controlling a large, heterogeneous population of $N$ air-conditioners such that their overall consumption tracks a given target profile $r = (r_t)_{0\leq t\leq T}$ while imposing that they recover their initial state distribution at the end of the time horizon $[0, T]$. This application was already considered in \cite{FullyBackward} without the constraint on the terminal marginal distribution.
	Air-conditioners are aggregated in $d$ clusters indexed by $1 \le i \le d$ depending on their characteristics. We denote by $N_i$ the number of air-conditioners in the cluster $i$. Individually, the temperature $X^{i, j}$ in the room with air-conditioner $j$ in cluster $i$ is assumed to evolve according to the following linear dynamics
	\begin{equation}
		\label{eq:individualTempSDE}
		dX_t^{i, j} = -\theta^i(X_t^{i, j} - x_{out}^i)dt - \kappa^iP_{max}^iu^{i, j}_t dt + \sigma^{i, j}dW_t^{i, j}, ~X_0^{i, j} = x_0^{i, j}, 1 \le i \le d, 1 \le j \le N_i,
	\end{equation} 
	where  $x_{out}^i$ is the outdoor temperature; $\theta^i$ is a positive thermal constant;
	$\kappa^i$ is the heat exchange constant; $P^i_{max}$ is the maximal power consumption of an air-conditioner in cluster $i$. $W^{i, j}$ are independent Brownian motion that represent random temperature fluctuations inside the rooms, such as a window or a door opening. For each cluster, a \textit{local controller} decides at each time step to turn $ON$ or $OFF$ some conditioners in the cluster $i$ by setting $u^{i, j} = 1$ or $0$ in order to satisfy a \textit{prescribed proportion} of active air-conditioners. We are interested in the global planner problem which consists in computing the prescribed proportion $u^i = \frac{1}{N_i}\sum_{j = 1}^{N_i} u^{i, j}$ of air conditioners ON in each cluster in order to track the given target consumption profile $r = (r_t)_{0\leq t\leq T}$. For each $1 \le i \le d$ the average temperature $X^i = \frac{1}{N}\sum_{j = 1}^{N_i} X^{i, j}$ in the cluster $i$ follows the aggregated dynamics
	\begin{equation}
		\label{eq:aggregatedTempSDE}
		dX_t^{i} = -\theta^i(X_t^{i} - x_{out}^i)dt - \kappa^iP_{max}^iu^{i}_t dt + \sigma^{i}dW_t^{i},
	\end{equation}
	with
	$$
	W_t^i = \frac{1}{N_i}\sum_{j = 1}^{N_i}W_t^{i, j}, ~\sigma^i = \frac{1}{N_i}\sum_{j = 1}^{N_i} \sigma^{i, j}.
	$$
	We consider the stochastic control Problem \eqref{eq:introInitialProblemJump} on the time horizon $[0,T]$ with $\U = [0, 1]^d$ and $T$ being
      $2$ hours.  The running cost $f$ is defined for any $(t, x, u) \in [0, T] \times \R^d \times \U$ such that
	\begin{equation}
		\label{eq:runningCostExample}
		f(t, x, u) := C\left(\sum_{i=1}^d \rho_i u_i - r_t\right)^2 + \frac{1}{d}\sum_{i=1}^d\left(\gamma_i(\rho_i u_i)^2 + \lambda_i (x_i - x_{max}^i)^2_+ + \lambda_i (x_{min}^i - x_i)^2_+\right),
	\end{equation}
        where $\rho_i = \frac{P^i_{max}}{\sum_{i=1}^d P^i_{max}}, \lambda_i, \gamma_i \ge 0,  0 \le i \le d$.  

We refer to \cite{FullyBackward} for the precise signification of the parameters of the model. The difference with \cite{FullyBackward} is that we impose a terminal target distribution $\eta_T$ instead of a terminal cost $g$. In particular, we fix $\shk = \{\P \in \shp(\Omega) ~:~\P_T = \eta_T\}$. We illustrate the convergence of the alternating minimization procedure in the situation where $d = 5,$ where the initial and terminal laws are fixed to be
$\eta_0 = \eta_T = \shn(\mu, \Sigma),$ for some mean vector $\mu \in \R^d$ and a diagonal covariance matrix $\Sigma := diag(\nu^2_i)_{1 \le i \le d},$
The Monte-Carlo algorithm used is a slight modification 
of Algorithm 1. in Section 4.2 in
 \cite{BOROptimi2023},
which is based on Proposition \ref{prop:existenceMinQmu} providing the explicit solution of the first minimization subproblem $\underset{\Q \in \shk}{\inf} \shj_\epsilon(\Q, \P)$.
Indeed, thereby instead of considering a fixed terminal cost $g$,
in Step 1. of Algorithm 1, at each iteration $k$ we update the terminal cost $g_k$ based on \eqref{eq:densityQ*} by setting
	$$
	g_k(x) := \log\left(\frac{d\eta_T}{d\P_{k, T}}(x)\right) - \log\left(\E^{\P^k}\left[\exp\left(-\epsilon \int_0^T f(r, X_r, u^k(r, X_r))dr\right)\middle | X_T = x\right]\right)
	$$
	in Step 1 of the aforementioned algorithm.
	We ran the algorithm with penalization parameter $\epsilon = 5$, $K = 100$ iterations (and $ 10^5$ particles
        for the Monte-Carlo estimate required in the aforementioned Algorithm 1.). 
        As expected the penalized cost $\shj(\Q_k, \P_k)$ decreases and converges to a limit value. The cost $J(\P_k)$ decreases after a certain rank and gets very close from $\shj(\Q_k, \P_k)$. This is illustrated on Figure \ref{fig:costAlternate}. Figure \ref{fig:densities} compares the estimated marginal densities of $X_T$ with their respective target densities.
        
        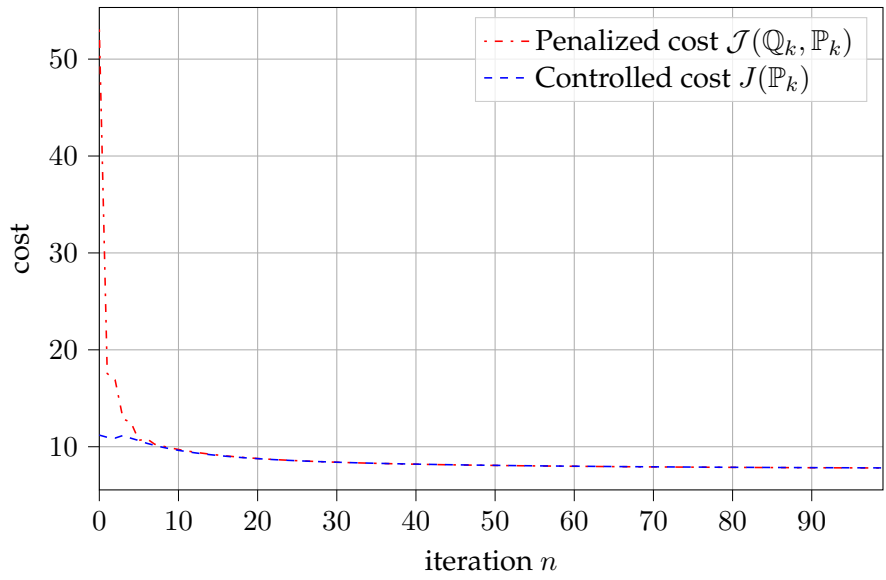
\begin{figure}[h]
        	\centering
\begin{tikzpicture}

\definecolor{darkgray176}{RGB}{176,176,176}
\definecolor{lightgray204}{RGB}{204,204,204}

\begin{axis}[
legend cell align={left},
legend style={fill opacity=0.8, draw opacity=1, text opacity=1, draw=lightgray204},
tick align=outside,
tick pos=left,
x grid style={darkgray176},
xlabel={iteration \(\displaystyle n\)},
xmajorgrids,
xmin=0, xmax=99,
xtick style={color=black},
y grid style={darkgray176},
ylabel={cost},
ymajorgrids,
ymin=5.5453498461204, ymax=55.3277152662473,
ytick style={color=black},
width=12cm,
height=8cm
]
\addplot [semithick, red, dash pattern=on 1pt off 3pt on 3pt off 3pt]
table {%
0 53.0648804744233
1 17.5561125834103
2 16.8805112023968
3 12.8714843307334
4 12.5060418614481
5 10.6553005907305
6 10.8285585444871
7 10.2573518335882
8 10.0665661683921
9 9.87785507207628
10 9.71677895558815
11 9.57536631017264
12 9.44769093853037
13 9.33384645153134
14 9.231058596479
15 9.13833790229706
16 9.05420059911219
17 8.9777666607305
18 8.90834811568573
19 8.8450005003607
20 8.78693837674594
21 8.73384719434776
22 8.68542030565132
23 8.64025441583922
24 8.59837889199946
25 8.55949191348899
26 8.523338181916
27 8.48960980121907
28 8.45807813628453
29 8.42857211513045
30 8.40092413379976
31 8.37495963222007
32 8.35050831236388
33 8.3274428123488
34 8.30566923631407
35 8.28510690334557
36 8.26567043806722
37 8.24720632181145
38 8.22963008834641
39 8.21288799244428
40 8.19691782124746
41 8.18166463900581
42 8.16708323282376
43 8.15312364137571
44 8.13974101448625
45 8.12689550358843
46 8.11455320785949
47 8.10268612319565
48 8.09126999222973
49 8.08028764829816
50 8.06972327047877
51 8.0595523145147
52 8.04974198921739
53 8.04026993453667
54 8.03111886303994
55 8.02227227115321
56 8.01371577938436
57 8.00543509457697
58 7.99741788731642
59 7.98965239544803
60 7.98212733784174
61 7.97483292941828
62 7.96775880567653
63 7.96089576982395
64 7.95423515758435
65 7.947768490852
66 7.94148868039292
67 7.93538822632555
68 7.92946114210218
69 7.92370141164779
70 7.91810391029034
71 7.91266372099388
72 7.90737604475204
73 7.9022364430942
74 7.89724072987298
75 7.89238467089529
76 7.88766407421532
77 7.88307477668281
78 7.8786127264445
79 7.87427387572563
80 7.87005446927983
81 7.86595086498082
82 7.86195967080918
83 7.85807731298122
84 7.85430072182756
85 7.85062695177803
86 7.84705323931692
87 7.84357534138802
88 7.84018921086314
89 7.83689182258323
90 7.83368055908209
91 7.83055251213827
92 7.82750443439521
93 7.82453336924838
94 7.82163675201561
95 7.81881209630547
96 7.81605715368783
97 7.8133690798085
98 7.81074566242146
99 7.80818463794435
};
\addlegendentry{Penalized cost $\mathcal{J}(\mathbb{Q}_k, \mathbb{P}_k)$}
\addplot [semithick, blue, dashed]
table {%
0 11.1991151615883
1 10.9727562887226
2 10.8933716336906
3 11.1622248746495
4 10.8777061507682
5 10.6347680755206
6 10.3558696416807
7 10.1410278537372
8 9.94354775526121
9 9.77829779391108
10 9.62853108578495
11 9.49869760972023
12 9.38074204080022
13 9.27616760570396
14 9.18118355949683
15 9.09562837020569
16 9.01776910921686
17 8.94720233905634
18 8.88285251532681
19 8.82393734426849
20 8.7697805938976
21 8.71989206488418
22 8.67382973305418
23 8.63116281145316
24 8.59154296578545
25 8.55467885963856
26 8.52029461390784
27 8.48815311047342
28 8.45804345066713
29 8.42980229816826
30 8.40325327792704
31 8.37822634018178
32 8.35458997161869
33 8.33225309045365
34 8.31112201150913
35 8.29111602603585
36 8.27213571780383
37 8.25405486510672
38 8.23682358761111
39 8.22037444141498
40 8.20466010675987
41 8.1896282618272
42 8.17523254460405
43 8.16142516884672
44 8.14816764756089
45 8.13542531075361
46 8.12316932027757
47 8.11137621618018
48 8.10002426055495
49 8.08909987160999
50 8.07858206062505
51 8.06844228792209
52 8.05865316540193
53 8.04919437869299
54 8.04005056251122
55 8.0312061242757
56 8.0226475862
57 8.01436201466945
58 8.00633754748722
59 7.99856326123598
60 7.99102873541703
61 7.98372438682125
62 7.97664110307551
63 7.96977036428981
64 7.9631040821926
65 7.95663467173693
66 7.95035490702762
67 7.9442580227581
68 7.93833760727594
69 7.93258755509717
70 7.92700208422887
71 7.92157565645137
72 7.916303039249
73 7.91117918857527
74 7.90619930702239
75 7.90135869756374
76 7.89665289525266
77 7.89207750803467
78 7.88762825719148
79 7.88330103936988
80 7.87909186526401
81 7.87499705270959
82 7.87101298719247
83 7.86713636839072
84 7.86336402514956
85 7.85969290456647
86 7.85611898301406
87 7.85263917278151
88 7.84925038640479
89 7.84594979264835
90 7.84273397172272
91 7.83960048105546
92 7.83654657823463
93 7.8335692826793
94 7.83066657533868
95 7.82783536717154
96 7.82507366683902
97 7.82237884198049
98 7.81974874866859
99 7.81718113259216
};
\addlegendentry{Controlled cost $J(\mathbb{P}_k)$}
\end{axis}

\end{tikzpicture}
        	\captionsetup{font=scriptsize}
        	\caption{Costs associated with the iterates generated by the entropy penalized Monte-Carlo algorithm in dimension $d = 5$ and $K = 100$.}
        	\label{fig:costAlternate}
        \end{figure} 
        
	
	\begin{figure}[H]
		\centering
		\includegraphics[width = 12cm, angle=0]{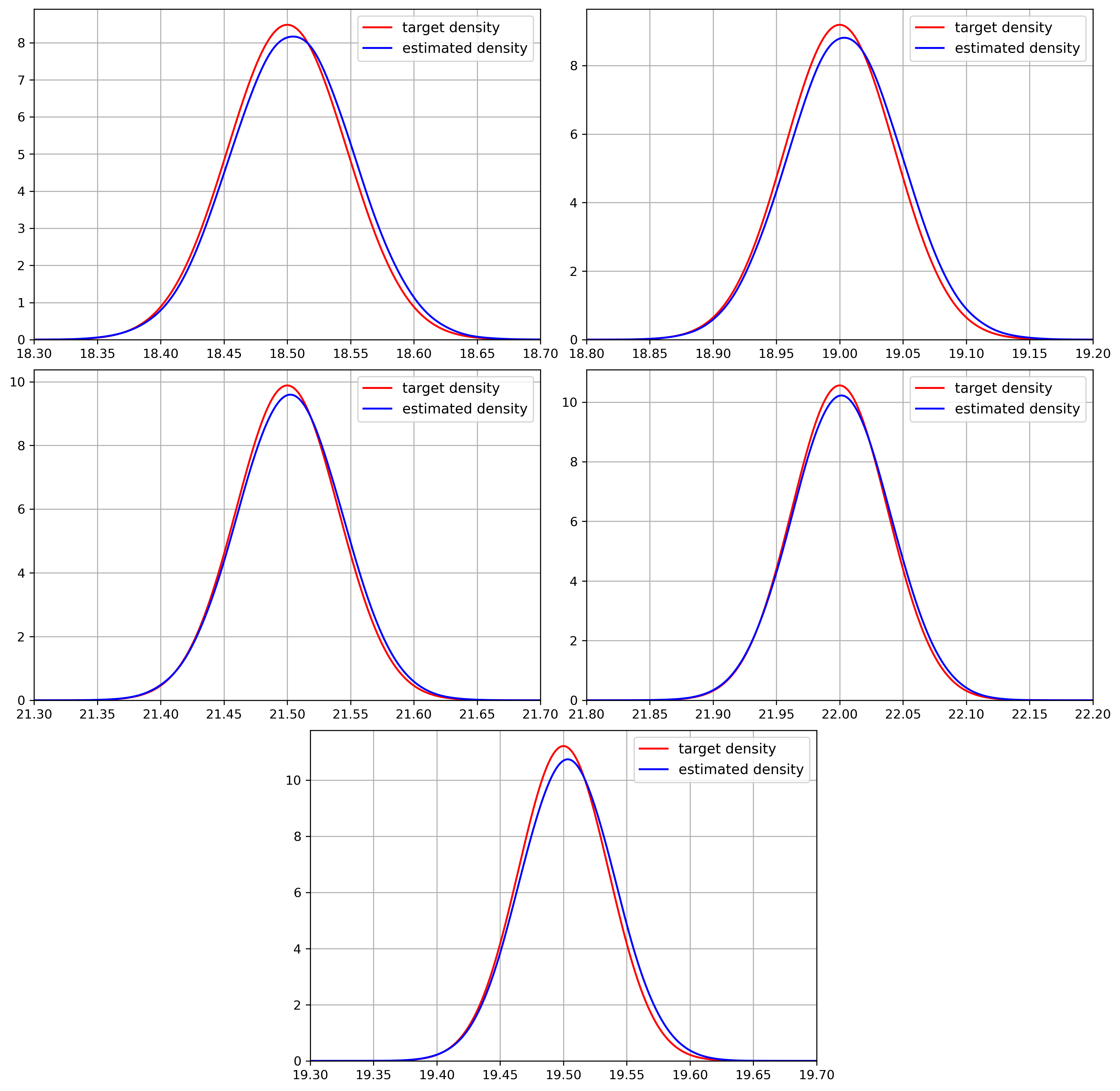}
		\caption{Comparison of the estimated marginal densities of $X_T$ with their respective target densities in dimension $d = 5$ with
$K = 100$.}
		\label{fig:densities}
	\end{figure}

	\appendix
%
	\section*{Appendices}
	\section{Relative entropy related results}
	\setcounter{equation}{0}
	\renewcommand\theequation{A.\arabic{equation}}
	\subsection{Girsanov theorem under finite entropy condition}

	The first  theorem of this Appendix  is a well-known result,
which can be established for instance by combining 
        the proofs of Theorems 2.1,2.3, 2.6 and 2.9 in \cite{GirsanovEntropy}.
	\begin{theorem}
          \label{th:entropyJumps}
          Let $b : [0, T] \times \Omega \rightarrow \R^d$ and $a : [0, T] \times \Omega \rightarrow S_d^+$ be two progressively measurable processes.
 Let $\P \in \shp(\Omega)$
          such that under $\P$ the canonical process has decomposition
		\begin{equation}
			\label{eq:decompPEntropy}
			X_t = X_0 + \int_0^t b_rdr + M_t^\P + (q\1_{\{|q| > 1\}})*\mu^X_t + (q\1_{\{|q| \le 1\}})*(\mu^X - \mu^L)_t,
		\end{equation}
		where $M^\P$ is a continuous local martingale such that $[ M^\P] = \int_0^\cdot a_rdr$ and $\mu^L(X, dt, dq) = L(t, X, dq)dt$ for some Lévy kernel $L$, see Definition \ref{def:levyKernel}.

                Let  $\Q \in \shp(\Omega)$ such that
                 $H(\Q | \P) < + \infty$.
  Then, there exists a progressively measurable process $\alpha$ and a $\tilde \shp$-measurable function $Y$ such that
  the following holds.
  \begin{itemize}
    \item
  Under $\Q$ the canonical process decomposes as
		\begin{equation}
			\label{eq:decompQEntropy}
			\begin{aligned}
				X_t & = X_0 + \int_0^t b_r dr + \int_0^t a_r\alpha_r dr + (q\1_{\{|q| \le 1\}}(Y - 1))*\mu^L_t + M_t^\Q\\
				& + (q\1_{\{|q| > 1\}})*\mu^X_t + (q\1_{\{|q| \le 1\}})*(\mu^X - (Y.\mu^L))_t,
			\end{aligned}
		\end{equation}
		where $M^\Q$ is a continuous $\Q$-local martingale such that $[ M^\Q] = \int_0^\cdot a_rdr$ and the $\Q$-compensator $Y.\mu^L$ of $\mu^X$ is given by $(Y.\mu^L)(X, dt, dq) = Y(X, t, q)\mu^L(X, dt, dq)$.
\item			
		\begin{equation}
			\label{eq:lowerBoundEntropy}
			H(\Q | \P) \ge H(\Q_0 | \P_0) + \frac{1}{2}\E^\Q\left[\int_0^T \alpha_r^\top a_r \alpha_rdr\right] + \E^\Q\left[(Y\log(Y) - Y + 1)*\mu^L_T\right],
		\end{equation}
		and equality holds in \eqref{eq:lowerBoundEntropy} if the martingale problem, in the sense Remark \ref{rmk:MartProb},
verified by $\P$ has a unique solution.
\end{itemize}
\end{theorem}
	\begin{remark} \label{rmk:char}
          Let us take into account the notion of
          characteristics of a semimartingale, see e.g.
          Definition 2.6, Chapter II in \cite{JacodShiryaev},
          with truncation function $k(x) = 1_{\{ \vert x \vert \le 1\}}.$
          
          By Remark \ref{rmk:MartProbStroock},
          Theorem \ref{th:entropyJumps} implies in particular that the new characteristics $(B^\Q, C^\Q, \mu^{\Q})$ of $X$ related to $\Q$ are
		\begin{equation*}
			\left\{
			\begin{aligned}
				& B^\Q_t = b_t + a_r\alpha_t +  (q\1_{\{|q| \le 1\}}(Y - 1))*\mu^L_t\\
				& C^\Q_t = a_t\\
				& \mu^\Q = Y\mu^L.
			\end{aligned}
			\right.
		\end{equation*}
	\end{remark}

	\subsection{Relative entropy between laws of c\`adl\`ag semimartingales}

        In this section we aim at computing the relative entropy between the laws of
	two càdlàg  semi-martingales,
	when it exists.
        In this framework we do not assume a priori the finiteness of the entropy
        between the two probabilities.
	The result below
        is a generalization  of Lemma 4.4 in \cite{LackerHierarchies}, where the considered semimartingales
        were continuous.
	Let $L$ be a Lévy kernel in the sense of Definition \ref{def:levyKernel} and set $\mu^L := L(t, X, dq)dt$. Let $b^1, b^2 : [0, T] \times \Omega \rightarrow \R^d$, and $a : [0, T] \times \Omega \rightarrow S_d^{++}$
        be progressively measurable processes. Let $\lambda^1, \lambda^2 : \Omega \times [0, T] \times \R^d \rightarrow \R^+_*$ be two non-negative $\tilde \shp$-measurable functions. Let $\P^1, \P^2 \in \shp(\Omega)$ such that for $i = 1, 2$, the canonical process
        decomposes under $\P^i$ 
	\begin{equation}
          \label{eq:decompP12}
          X_t = X_0 + \int_0^t b^i_r(r,X)dr + M^{\P^i}_t + (q\1_{\{|q| > 1\}})*\mu^X + (q\1_{\{|q| \le 1\}})*(\mu^X - \mu^i)_t,
	\end{equation}
	where
	\begin{enumerate}[label = (\roman*)]
		\item $M^{\P^i}$ is a continuous $\P^i$-local martingale verifying $[M^{\P^i}] = \int_0^\cdot a(r,X) dr$;
		\item $\mu^i(X, dt, dq) := \lambda^i(X, t, q)L(t, X, dq)dt$.
                \end{enumerate}
                From now on we use the notation $b_r:= b(r,X), \ a_r:= a(r,X),
                r \in [0,T].$

Below we will need  the following assumptions
for $(b^1, a, \lambda^1 L)$ and on the probability $\P^1$.
We set $\eta_0 = \P^1_0$.

            \begin{hyp} 
		\label{hyp:uniqueness}
                (Uniqueness and flow existence).
                \begin{itemize}
                \item
                  $\P^1$ is the unique probability under which  decomposition \eqref{eq:decompP12} holds
                  with $X_0 \sim \eta_0$. 
                \item 
The triplet $(b^1, a, \lambda^1L)$ verifies the flow existence property with associated progressively measurable kernel $(\P^{s, \omega})_{(s, \omega) \in [0, T] \times \Omega}$ in the sense of Definition \ref{def:flowExistence}.
\end{itemize}
\end{hyp}


                The assumption below concerns both probabilities $\P^1, \P^2$
                and the triplets  $(b^1, a, \lambda^1 L)$ and  $(b^2, a, \lambda^2 L)$.

	\begin{hyp}
		\label{hyp:almostSureBound}
		(Almost surely finite). We suppose
      		\begin{equation}
			\label{eq:driftAsFinite}
	\int_0^T (b^1_r - b^2_r)^\top a_r^{-1} (b^1_r - b^2_r) dr < + \infty\quad\P^2\text{-a.s.}
		\end{equation}
		and
		\begin{equation}
			\label{eq:driftJumpAsFinite}
			\int_{[0, T] \times \R^d}|q|\1_{\{|q| \le 1\}}|\lambda^2 - \lambda^1|L(r, X, dq)dr < + \infty\quad\P^2\text{-a.s.}
		\end{equation}

	\end{hyp}
        Under Hypothesis \ref{hyp:almostSureBound} we also introduce the following notations.
	\begin{equation}
		\label{eq:defYAlphaM}
		Y := \frac{\lambda^1}{\lambda^2}, \quad \alpha := a^{-1}\left(b^1 - b^2 - \int_{\R^d}q\1_{\{|q| \le 1\}}(\lambda^1 -
                  \lambda^2)L(\cdot, \cdot, dq)\right), \quad M := \int_0^\cdot \alpha_r^\top dM_r^{\P^2}.
	\end{equation}
	\begin{prop}
		\label{prop:entropyConverse}
		Let $a, b^i, \lambda^i, \P^i$, $i = 1, 2$ be as above.
                Assume
     Hypotheses \ref{hyp:uniqueness}, \ref{hyp:almostSureBound} and
		\begin{equation}
			\label{eq:finiteExpectation}
			\E^{\P^2}\left[\frac{1}{2}\int_0^T\alpha_r^\top a_r \alpha_rdr + \left(\frac{\lambda^2}{\lambda^1}\log\left(\frac{\lambda^2}{\lambda^1}\right) - \frac{\lambda^2}{\lambda^1} + 1\right)*\mu^1_T\right] < + \infty.
                      \end{equation}
                      Suppose $\P^2_0 \sim \eta_0$.
		Then $H(\P^2 | \P^1) < + \infty$ and
		\begin{equation}
			\label{eq:entropySM}
			H(\P^2| \P^1) = H(\P^2_0\vert  \P^1_0)+ \E^{\P^2}\left[\frac{1}{2}\int_0^T\alpha_r^\top a_r \alpha_rdr + \left(\frac{\lambda^2}{\lambda^1}\log\left(\frac{\lambda^2}{\lambda^1}\right) - \frac{\lambda^2}{\lambda^1} + 1\right)*\mu^1_T\right] < + \infty.
                      \end{equation}

                    \end{prop}

	The proof of Proposition \ref{prop:entropyConverse} requires some intermediate results. For simplicity we suppose $\eta_0 = \delta_x$ for some $x \in \R^d$ so that $H(\P^2_0\vert  \P^1_0) = 0$, see Remark \ref{rmk:extensionEquiv} below.
        Notice first that \eqref{eq:finiteExpectation} implies that
        \begin{equation*}
        	\begin{aligned}
        		\E^{\P^2}\left[\left(\frac{\lambda^2}{\lambda^1}\log\left(\frac{\lambda^2}{\lambda^1}\right) - \frac{\lambda^2}{\lambda^1} + 1\right)*\mu^1_T\right] & = \E^{\P^2}\left[\left(\lambda^2\log\left(\frac{\lambda^2}{\lambda^1}\right) - \lambda^2 + \lambda^1\right)*\mu^L_T\right]\\
        		& = \E^{\P^2}\left[\left(\frac{\lambda^1}{\lambda^2} - \log\left(\frac{\lambda^1}{\lambda^2}\right) - 1\right)*\mu^2_T\right]\\
        		& = \E^{\P^2}\left[\left(Y - \log(Y) - 1\right)*\mu^2_T\right] < + \infty,
        	\end{aligned}
        \end{equation*}
        the latter expectation being well-defined since $x \in ]0, + \infty[ \mapsto x - \log(x) - 1 \ge 0 $. 
        Hence the triplet $(Y, \P^2, \lambda^2 L)$ verifies  Hypothesis \ref{hyp:expoMart} below and by Lemma \ref{lemma:integExpMart}
$M + (Y - 1)*(\mu^X - \mu^2)$ is a well-defined $\P^2$-local martingale
and
       the Dol\'eans-Dade exponential $Z := \she\left(M + (Y - 1)*(\mu^X - \mu^2)\right)$
       is  well-defined 
        (see e.g. Theorem 4.61, Chapter I in \cite{JacodShiryaev}),
and
by Lemma \ref{lemma:rewriteZ}         rewrites 
	\begin{equation}
		\label{eq:densityZ}
		Z = \exp\left(M + \log(Y)*(\mu^X - \mu^2) - \frac{1}{2}[M] - \left(Y - \log(Y) - 1\right)*\mu^2\right).
	\end{equation}
	The local martingale $Z$ will constitute the basis for the proof of Proposition \ref{prop:entropyConverse}. Let then $(\tau_k)_{k \ge 0}$ be a localizing sequence for $Z$ as well as $M$ and $ \log(Y)*(\mu^X - \mu^2)$ appearing in \eqref{eq:densityZ}. For any $k \ge 0$, we also define the probability measure $\Q^k$ by
	\begin{equation}
		\label{eq:defQk}
		d\Q^k := Z^{\tau_k}_Td\P^2,
	\end{equation}
	where $Z^{\tau_k} := Z_{\cdot \wedge \tau_k}$.
	\begin{remark}
		\label{rmk:uiMartingale}
		Recall that, since $Z_t^{\tau_k} = \E^{\P^2}[Z_T^{\tau_k} | \shf_t]$ for all $t \in [0, T]$, the martingale $Z^{\tau_k}$ is uniformly integrable. 
                Indeed any martingale defined on a compact interval is uniformly integrable.
                This property easily follows from De la Vallée Poussin criteria of uniform integrability applied to the family $\{\E^\P[Z_T | \shf_t]~:~t \in [0, T]\}$, see e.g. T22, Chapter II in \cite{meyer} for a statement and a proof of this criteria.
	\end{remark}
	Lemma \ref{lemma:QIsP} below constitutes a crucial step for the proof of Proposition \ref{prop:entropyConverse}. Using Girsanov's theorem, we express the characteristics of the canonical process $X$ under $\Q_k$. Then, relying on decomposition \eqref{eq:decompP12}
        as well as Hypothesis \ref{hyp:uniqueness}, we prove
        the following.
	\begin{lemma}
          \label{lemma:QIsP}
          We formulate the same assumptions as in Proposition \ref{prop:entropyConverse}. 
         Let $\Q^k \in \shp(\Omega)$ be defined in \eqref{eq:defQk}. Then
    the restriction of
        $\Q^k$ to $\sigma$-field $\shf_{\tau_k}$
        equals $(\P^1)^{\tau_k}$.
        \end{lemma}
	\begin{proof}
      We keep in mind the notation \eqref{eq:defYAlphaM}. Since $Z^{\tau_k}$ is a uniformly integrable $\P^2$-martingale (see Remark \ref{rmk:uiMartingale}),
      Theorem 15.3.10 and Remark 15.3.11 in \cite{ElliottCohenStochasticCalculus}
      applied with $\Q = \Q_k$, $\P = \P^2$, $f = \alpha\1_{\llbracket 0, \tau_k \rrbracket}$, $\alpha_p = \left(Y\1_{\llbracket 0, \tau_k \rrbracket}\right)\mu^2$, $A = Y\1_{\llbracket 0, \tau_k \rrbracket}$, $\mu_p^X = \mu^2\1_{\llbracket 0, \tau_k \rrbracket}$, states that under $\Q^k$ the canonical process $X$ has decomposes as
      \begin{equation}
      	\begin{aligned}
      		\label{eq:elliottX}
      		\begin{aligned}
      			X_t & = x + \int_0^t \left(\1_{\llbracket 0, \tau_k \rrbracket}b^2\right)(r, X)dr + \bar M_t + \int_0^t(\1_{\llbracket 0, \tau_k \rrbracket}a\alpha)(r,X)dr + \left(q\1_{\{|q| \le 1\}}(Y - 1)\right)*\left(\1_{\llbracket 0, \tau_k \rrbracket}\mu^2\right)_t\\           
                            & + \left(q\1_{\{|q| > 1\}}\right)*\mu^X_t + \left(q\1_{\{|q| \le 1\}}\right)*(\mu^X - \left(Y\1_{\llbracket 0, \tau_k \rrbracket}\mu^2)\right)_t,
      		\end{aligned}
      	\end{aligned}
      \end{equation}
  	where $\bar M$ is a continuous local martingale under $\Q^k$ with $[\bar M] = \int_0^\cdot \1_{\llbracket 0, \tau_k \rrbracket}a_r dr.$ Stopping the process $X$ in \eqref{eq:elliottX} at $\tau_k$, we get the following decomposition for the stopped canonical process $X_{\cdot \wedge \tau_k}$,
     \begin{equation}
     	\label{eq:decompElliott}
     	\begin{aligned}
          X_{t \wedge \tau_k} & = x + \int_0^{t \wedge \tau_k}b^2(r, X)dr + \bar M_{t \wedge \tau_k} + \int_0^{t \wedge \tau_k}(a\alpha )(r,X)dr +
    \left(q\1_{\{|q| \le 1\}}(Y - 1)\right)*\mu^2_{t \wedge \tau_k}\\
     		& + \left(q\1_{\{|q| > 1\}}\right)*\mu^X_{t \wedge \tau_k} + \left(q\1_{\{|q| \le 1\}}\right)*(\mu^X - (Y.\mu^2))_{t \wedge \tau_k}.
     	\end{aligned}
     \end{equation}
 	Using the expression \eqref{eq:defYAlphaM} of $Y$ and $\alpha$, we have that
	\begin{equation}
		\label{eq:charac1}
		\begin{aligned}
			\int_0^{t \wedge \tau_k}(a\alpha )(r, X)dr & = \int_0^{t \wedge \tau_k}b^1(r, X)dr - \int_0^{t \wedge \tau_k}b^2(r, X)dr\\
			& - \int_{[0, t \wedge \tau_k] \times \R^d}q\1_{\{|q| \le 1\}}(\lambda^1 - \lambda^2)L(r, X, dq)dr,
		\end{aligned}
	\end{equation}
	\begin{equation}
		\label{eq:charac2}
		\left(q\1_{\{|q| \le 1\}}(Y - 1)\right)*\mu^2_{t \wedge \tau_k} = \int_{[0, t \wedge \tau_k] \times \R^d}q\1_{\{|q| \le 1\}}(\lambda^1 - \lambda^2)L(r, X, dq)dr,
	\end{equation}
	and
	\begin{equation}
		\label{eq:charac3}
		\left(q\1_{\{|q| \le 1\}}\right)*(\mu^X - (Y.\mu^2))_{t \wedge \tau_k} = (q\1_{\{|q| \le 1\}})*(\mu^X - \mu^1)_{t \wedge \tau_k}.
	\end{equation}
	Injecting \eqref{eq:charac1}, \eqref{eq:charac2} and \eqref{eq:charac3} in \eqref{eq:decompElliott} yields the decomposition
	\begin{equation}
		\label{eq:decompTaukj}
		X_{t \wedge \tau_k} = x + \int_0^{t \wedge \tau_k} b^1(r,X)dr + \bar M_{t \wedge \tau_k} + (q\1_{\{|q| > 1\}})*\mu^X_{t \wedge \tau_k} + (q\1_{\{|q| \le 1\}})*(\mu^X - \mu^1)_{t \wedge \tau_k}.
	\end{equation}
It follows from \eqref{eq:decompTaukj} that, under $\Q^k$, the canonical process has decomposition \eqref{eq:decompP12} with $i = 1$ stopped at $\tau_k$. However it is not immediate that the uniqueness of the decomposition \eqref{eq:decompP12} with $i = 1$ given by Hypothesis \ref{hyp:uniqueness} carries over to the random interval $[0, \tau_k]$. This is the object of the rest of the proof.
		
Let $(\P^{s, \omega})_{(s, \omega) \in [0, T] \times \Omega}$ the the the path-dependent class provided by Hypothesis \ref{hyp:uniqueness}.
We set $\omega \in \Omega \mapsto \Q_\omega := \P^{\tau_k(\omega), \omega_{\cdot \wedge \tau_k(\omega)}}$. The mapping $\omega \mapsto \Q_\omega$ satisfies the following properties.
\begin{enumerate}[label = (\roman*)]
\item
 We can show that, $A \in \shf$,  the map $\omega \mapsto \Q_\omega(A)$ is $\shf_{\tau_k}$-measurable for any $A \in \shf$. To see this, we introduce the map $\Psi : \omega \mapsto (\tau_k(\omega), \omega_{\cdot \wedge \tau_k(\omega)})$ which is $(\shf_{\tau_k}, \shb([0, T])\otimes \shf)$-measurable. Indeed, $\omega \mapsto \tau_k(\omega)$ is obviously $\shf_{\tau_k}$-measurable, and since $\tau_k$ is a stopping time and $\omega \in \Omega \mapsto \omega$ is càdlàg, $\omega \mapsto \omega_{\cdot \wedge \tau_k(\omega)}$ is also $\shf_{\tau_k}$-measurable, and the measurability of $\Psi$ follows. Denoting $\Phi : (s, \omega) \mapsto \P^{s, \omega}(A)$, the map $\omega \mapsto \Q_{\omega}(A)$ can be written as the composition $\Phi \circ \Psi$. Since $\Psi$ is $\shf_{\tau_k}$-measurable, and since $\Phi$ is measurable by definition, $\Phi \circ \Psi$ is $\shf_{\tau_k}$ measurable.
                \item
                  $\Q_\omega\left(X_{\tau_k(\omega)} = X_{\tau_k(\omega)}(\omega)\right) = \P^{\tau_k(\omega), \omega_{\cdot\wedge \tau_k(\omega)}}\left(X_r = \omega_r, 0 \le r \le \tau_k(\omega)\right) = 1$ for all $\omega \in \Omega$.
                        \end{enumerate}

                        Then by Theorem 6.1.2 in \cite{stroock} there exists a unique probability measure $\Q \in \shp(\Omega)$ such that $\Q = \Q^k$ on $\shf_{\tau_k}$ and $(\Q_\omega)_{\omega \in \Omega}$ is a \textit{regular conditional probability distribution} of $\Q$ given $\shf_{\tau_k}$, see the definition just before Theorem 1.1.6 in \cite{stroock}. We are going to prove that, under $\Q$, the canonical process has decomposition \eqref{eq:decompP12} with $i = 1$. To this aim, we take into account Remark \ref{rmk:MartProbStroock}.
		Let $h$ be a bounded $\shc^2$-function. We define
		\begin{equation*}
			\begin{aligned}
				N[h] & := h(X_\cdot) - h(x) - \int_0^\cdot \shl^1(h)(r, X)dr,
			\end{aligned}
		\end{equation*}
		where
		\begin{equation*}
			\begin{aligned}
				\shl^1(h)(t, X) & := \langle \nabla_x h(t, X_t), b^1(t, X) \rangle + \frac{1}{2}Tr[a(t, X)\nabla_x^2h(t, X_t)]\\
				& + \int_{\R^d} \left(h(t, X_{t-} + q) - h(t, X_{t-}) - \langle \nabla_xh(X_{t-}), q\rangle\1_{\{|q| \le 1\}}\right)\mu^1(t, X, dq).
			\end{aligned}
		\end{equation*}
		We need to prove that $N[h]$ is a local martingale under $\Q$. We fix $n \in \N^*$.
	We also introduce the sequence of stopping times
		$$
		\sigma_n := \inf\left\{t \in [0, T]~:~\int_0^t \shl^1(h)(r, X)dr \ge n\right\}, ~n \ge 1.$$
                We prove below that $N[h]^{\sigma_n}$ is a $\Q$-martingale. 

                We have the following two facts.
		\begin{enumerate}
                \item By decomposition \eqref{eq:decompTaukj} and Itô's formula, the process $N[h]_{\cdot \wedge \tau_k}$ is a local martingale under $\Q^k$.
                  Since $N[h]^{\sigma_n}_{\cdot \wedge \tau_k}$
                  is bounded, it is a martingale under $\Q_k$.

                \item
                  We recall Hypothesis \ref{hyp:uniqueness}
                  which states that $(P^{s,\omega})$ verifies
                  the flow existence property, see Definition
                  \ref{def:flowExistence}, with respect to
                  $(b^1, a, \lambda^1 L)$.
                  We recall that
                  $\Q_\omega = \P^{\tau_k(\omega), \omega_{\cdot \wedge \tau_k(\omega)}}$. By decomposition \eqref{eq:decompSOmega} applied with $b = b^1$, $L = \lambda^1L$ and $s = \tau_k(\omega)$, together with Itô's formula, the process $N[h]_{\cdot \wedge \tau_k} - N[h]_{\cdot \wedge \tau_k(\omega)}$ is a local martingale under $\Q_\omega$.
Consequently the stopped process
$N[h]^{\sigma_n}_{\cdot \wedge \tau_k} - N[h]^{\sigma_n}_{\cdot \wedge \tau_k(\omega)}$ is a local $\Q_\omega$-martingale, therefore
a genuine martingale since it is bounded.

                \end{enumerate}
                
                Using the final conclusion of Theorem 6.1.2 in \cite{stroock},
                see Remark \ref{rmk:stroockCadlag} below,
                it follows from previous items 1. and 2. that the process $N[h]^{\sigma_n}$ is a martingale under $\Q$.
               Finally $N[h]$ is indeed a local martingale under $\Q$, hence indeed,
               that under $\Q$ the canonical process has decomposition \eqref{eq:decompP12} with $i = 1$.

              By the uniqueness item of such decomposition, see Hypothesis \ref{hyp:uniqueness}, we get $\Q = \P^1$. This implies in particular that $\Q^{\tau_k} = (\P^1)^{\tau_k}$. By construction, we have $\Q = \Q^k$ on $\shf_{\tau_k}$, hence $\Q^{\tau_k} = (\Q^k)^{\tau_k}$ yielding by what precedes that $(\Q^k)^{\tau_k} = (\P^1)^{\tau_k}$. This concludes the proof.

	\end{proof}

	\begin{remark}
		\label{rmk:stroockCadlag}
		Theorem 6.1.2 in \cite{stroock} applies to the case where $\Omega = C([0, T], \R^d)$, but its proof remains valid without any modification when $\Omega = D([0, T], \R^d)$.
	\end{remark}

	We are now ready to prove Proposition \ref{prop:entropyConverse}.
	
	\begin{proof}[Proof of Proposition \ref{prop:entropyConverse}.]
		Recall the definition \eqref{eq:defQk} of the probability measure $\Q^k$. By Lemma \ref{lemma:QIsP}, we have $(\Q^k)^{\tau_k} = (\P^1)^{\tau_k}$. Then $d(\P^1)^{\tau_k} = Z_T^{\tau_k}d(\P^2)^{\tau_k}$ and since $Z_T^{\tau_k} > 0$, $(\P^2)^{\tau_k} \ll (\P^1)^{\tau_k}$ and $d(\P^2)^{\tau_k} = \left(1/Z^{\tau_k}_T\right)d(\P^1)^{\tau_k}$. We focus now on the process $1/Z^{\tau_k}$, which rewrites by \eqref{eq:densityZ}
		$$
		\frac{1}{Z^{\tau_k}} = \exp\left(- N_{\cdot \wedge \tau_k} + \frac{1}{2}[M]_{\cdot \wedge \tau_k} + \left((Y - \log(Y) - 1)\right)*\mu^2_{\cdot \wedge \tau_k}\right),
		$$
		where $N := M + \log(Y)*(\mu^X - \mu^2)$ is a $\P^2$-local martingale. $N_{\cdot \wedge \tau_k}$ is a true $\P^2$-martingale, 
         and we then have
		\begin{equation*}
			\begin{aligned}
				H((\P^2)^{\tau_k} | (\P^1)^{\tau_k}) & = - \E^{(\P^2)^{\tau_k}}\left[\log Z^{\tau_k}_T\right]\\
				& = \E^{(\P^2)^{\tau_k}}\left[\frac{1}{2}\int_0^{T \wedge \tau_k}\alpha_r^\top a_r \alpha_rdr + \left((Y - \log(Y) - 1)\right)*\mu^2_{T\wedge \tau_k}\right]\\
				& = \E^{\P^2}\left[\frac{1}{2}\int_0^{T \wedge \tau_k}\alpha_r^\top a_r \alpha_rdr + \left((Y - \log(Y) - 1)\1_{\llbracket 0, \tau_k \rrbracket}\right)*\mu^2_{T\wedge \tau_k}\right]\\
				& \le \E^{\P^2}\left[\frac{1}{2}\int_0^T\alpha_r^\top a_r \alpha_rdr + \left(Y - \log(Y) - 1\right)*\mu^2_T\right]\\
				& = \E^{\P^2}\left[\frac{1}{2}\int_0^T \alpha_r^\top a_r \alpha_rdr + \left(\frac{\lambda^2}{\lambda^1}\log\left(\frac{\lambda^2}{\lambda^1}\right) - \frac{\lambda^2}{\lambda^1} + 1\right)*\mu^1_T \right],
                        \end{aligned}
                      \end{equation*}
recalling that $\mu^i = \lambda^i \mu^L, i= 1,2.$
                      Now since $(\P^1)^{\tau_k} \underset{k \rightarrow + \infty}{\longrightarrow} \P^1$ and $(\P^2)^{\tau_k} \underset{k \rightarrow + \infty}{\longrightarrow} \P^2$
                weakly, the joint lower semi-continuity of the relative entropy for the weak* convergence on Polish spaces (see Remark \ref{rmk:relativeEntropy} item 2.) yields
		\begin{equation}
			\label{eq:upperBoundH}
			H(\P^2 | \P^1) \le \E^{\P^2}\left[\frac{1}{2}\int_0^T \alpha_r^\top a_r \alpha_rdr + \left(\frac{\lambda^2}{\lambda^1}\log\left(\frac{\lambda^2}{\lambda^1}\right) - \frac{\lambda^2}{\lambda^1} + 1\right)*\mu^1_T \right].
		\end{equation}
The inequality  \eqref{eq:upperBoundH} implies by
        \eqref{eq:finiteExpectation} that $H(\P^2 | \P^1) < + \infty$. Theorem \ref{th:entropyJumps} then applies and we directly deduce from \eqref{eq:lowerBoundEntropy}
		that
		$$
		H(\P^2 | \P^1) \ge \E^{\P^2}\left[\frac{1}{2}\int_0^T \alpha_r^\top a_r \alpha_rdr + \left(\frac{\lambda^2}{\lambda^1}\log\left(\frac{\lambda^2}{\lambda^1}\right) - \frac{\lambda^2}{\lambda^1} + 1\right)*\mu^1_T \right].
		$$
		This concludes the proof.
              \end{proof}
	\begin{remark}
		\label{rmk:extensionEquiv}
		In the more general setting where we only assume that
                $\P^2_0 \sim \eta_0$, we simply replace the density $Z$ in \eqref{eq:densityZ} by $\tilde Z := (d\P^1_0/d\P^2_0)Z$ and the proof follows exactly the same line as in the case $\eta_0 = \delta_x$.  
	\end{remark}
        In the corollary below we remove the assumption
	   $\P^2_0 \sim \eta_0$.
This will raise some technical difficulties.
           \begin{coro}
          \label{coro:entropyConverse}
          Assume that, under $\P^i$, $i = 1, 2$, the canonical process has decomposition \eqref{eq:decompP12}.
          
          Assume Hypotheses \ref{hyp:uniqueness} (which only concerns $\P^1$)
          and \ref{hyp:almostSureBound}.
          Assume moreover that \eqref{eq:finiteExpectation} holds.
          Then
          \begin{equation*}
			H(\P^2 | \P^1) = H(\P^2_0 | \P^1_0) + \E^{\P^2}\left[\frac{1}{2}\int_0^T\alpha_r^\top a_r \alpha_rdr + \left(\frac{\lambda^2}{\lambda^1}\log\left(\frac{\lambda^2}{\lambda^1}\right) - \frac{\lambda^2}{\lambda^1} + 1\right)*\mu^1_T\right].
		\end{equation*}
              \end{coro}
  %
            
              which raises some technical difficulties: in particular
the r.v. $\tilde Z$ in Remark \ref{rmk:extensionEquiv} is not properly defined.

              Indeed, the proof of  
              that Corollary is based on  Lemma \ref{lemma:desintegration}
              below,
             which states a disintegration result.
             The proof of that lemma
             can be found in \cite{BORMArtingaleProblems}
        and it is based on the equivalence criterion in Remark \ref{rmk:MartProbStroock}. Similar arguments are developed in the uniqueness part of Theorem 4.5
        in \cite{IssoglioRussoDistDriftBesov}.
	\begin{lemma}
		\label{lemma:desintegration}
		Let $\P \in \shp(\Omega)$ such that, under $\P$, the canonical process has decomposition \eqref{eq:decompPEntropy} with initial law $\eta_0$.
		\begin{enumerate}
			\item There is  a measurable kernel $(\P(x))_{x \in \R^d}$ such that
			$\P = \int_{\R^d}\P(x)\eta_0(dx)$, where for $\eta_0$-almost all $x \in \R^d$, under $\P(x)$, the canonical process decomposes as
			\begin{equation}
				\label{eq:decompDesintegration}
				X_t = x + \int_0^t b_rdr + M^{0, x}_t + \left(q\1_{\{|q| > 1\}}\right)*\mu^X_t + \left(q\1_{\{|q| \le 1\}}\right)*(\mu^X - \mu^L)_t,
			\end{equation}
			where $M^{0, x}$ is a continuous $\P(x)$-local martingale verifying $[M^{0, x}] = \int_0^\cdot a_rdr$.
			
\item Assume moreover that, given $b$ and $\eta_0$, the decomposition \eqref{eq:decompDesintegration} is unique in law. Then the decomposition \eqref{eq:decompDesintegration} is also unique in law for $\eta_0$-almost all $x \in \R^d$.
		\end{enumerate}
	\end{lemma}
	\begin{proof}[Proof of Corollary \ref{coro:entropyConverse}.]
		The result is trivial if  $H(\P^2_0 | \P^1_0) = + \infty$, so we can suppose $H(\P^2_0 | \P^1_0) < + \infty$.
		We set $\eta_0^1 := \shl^{\P^1}(X_0)$ and $\eta_0^2 := \shl^{\P^2}(X_0)$. We have the following. 
		\begin{enumerate}[label = (\roman*)]
                \item Let $i=1,2$. Lemma \ref{lemma:desintegration} item 1. applied with $\P = \P^i$ provides the existence of a measurable kernel
                  $(\P^i(x))_{x \in \R^d}$ and an $\eta_0^1$-null set $\shn_1$ such that $\P^i = \int_{\R^d}\P^i(x) \eta_0^i(dx)$ and for all $x \in \shn_i^c$,
                  the canonical process $X$ has decomposition \eqref{eq:decompDesintegration} with $b = b^i$ and $\mu^L = \mu^i$ under $\P^i(x)$.
\item
  Moreover, since decomposition \eqref{eq:decompP12} (given $\eta^1_0, b^1$ and $\mu^1$), is unique in law, item 2. of the same lemma states that one can chose
  $\shn_1$ such that for all $x \in \shn_1^c$, decomposition \eqref{eq:decompDesintegration} under $\P^1(x)$ is unique in law.
			
		\end{enumerate}
		Recall that $H(\eta_0^2 | \eta_0^1) = H(\P_0^2 | \P_0^1) < + \infty$. This implies in particular that $\eta_0^2 \ll \eta_0^1$, hence $\shn_1$ is also an $\eta_0^2$-null set and $\shn := \shn_1 \cup \shn_2$ is an $\eta_0^2$-null set.
		On the one hand, by Lemma 2.3 in \cite{VaradhanAsymptotic},
		\begin{equation}
			\label{eq:varadhan}
			H(\P^2 | \P^1) = H(\eta_0^2 | \eta_0^1) + \int_{\R^d}H(\P(x)^2 | \P(x)^1)\eta_0^2(dx).
		\end{equation}
		On the other hand, for all $x \in \shn^c$, a direct application of Proposition \ref{prop:entropyConverse} with $\P^1 = \P^1(x)$ and $\P^2 = \P^2(x)$ yields
		\begin{equation}
			\label{eq:entropyAtoms}
			H(\P^2(x) | \P^1(x)) = \E^{\P^2(x)}\left[\frac{1}{2}\int_0^T\alpha_r^\top a_r \alpha_rdr + 	\left(\frac{\lambda^2}{\lambda^1}\log\left(\frac{\lambda^2}{\lambda^1}\right) - \frac{\lambda^2}{\lambda^1} + 1\right)*\mu^1_T\right].
		\end{equation}
		The result follows replacing $H(\P^2(x) | \P^1(x))$ in \eqref{eq:varadhan} by its expression \eqref{eq:entropyAtoms}.
	\end{proof}
        
	\section{Entropy penalized optimization problem with a prescribed terminal law}
	\setcounter{equation}{0}
	\renewcommand\theequation{B.\arabic{equation}}
	\begin{prop}
		\label{prop:existenceMinQmu}
		Let $\P \in \shp(\Omega)$ be fixed. Let $\eta_T \in \shp(\R^d)$ be a probability measure such that $H(\eta_T | \P_T) < + \infty$. Let $\lambda := d\eta_T/d\P_T$ and $\gamma_\epsilon$ be a measurable function such that $\gamma_\epsilon(X_T) := \E^\P\left[\exp(-\epsilon\vphi(X))|X_T\right]$, where $\epsilon > 0$ and $\varphi : D([0, T], \R^d) \rightarrow \R_+$ is a bounded Borel function. Set $\shp(\Omega, \eta_T) := \{\Q \in \shp(\Omega) ~:~\Q_T = \eta_T\}$. Then
		\begin{equation}
			\label{eq:optimalCost}
                        \min_{\Q \in \shp(\Omega, \eta_T)} \shj_\epsilon(\Q, \P) = \frac{1}{\epsilon}\int_{\R^d}\log\frac{\lambda(x)}{\gamma_\epsilon(x)}\eta_T(dx), \quad \text{where} \quad \shj_\epsilon(\Q, \P) = \E^\Q\left[\vphi(X)\right] + \frac{1}{\epsilon} H(\Q | \P),
		\end{equation}
		and there exists a unique minimizer $\Q^* \in \shp(\Omega, \eta_T)$ given by
		\begin{equation}
			\label{eq:densityQ*}
			d\Q^* := \exp\left(- \epsilon\vphi(X)\right)\frac{\lambda(X_T)}{\gamma_\epsilon(X_T)}d\P.
		\end{equation}
              \end{prop}
              \begin{remark}
                \label{rmk:AC}
                $\eta_T$ is absolutely continuous with respect to $\P_T$ so that
             the Radon-Nikodym density  $\lambda$ in the statement makes sense.
                \end{remark}

	\begin{proof}
          The proof decomposes into three steps.
          \begin{enumerate}
            \item We first check that $\Q^*$ defined by \eqref{eq:densityQ*} is an element
              of $\shp(\Omega, \eta_T)$ and
              \begin{equation}
            \label{eq:mineq2}
            \shj_\epsilon(\Q^*, \P) 
 = \frac{1}{\epsilon}\E^{\Q^*}\left[\log\frac{\lambda(X_T)}{\gamma_\epsilon(X_T)}\right]
      =      \frac{1}{\epsilon}\int_{\R^d}\log \frac{\lambda(x)}{\gamma_\epsilon(x)}\eta_T(dx).
    \end{equation}
    Moreover $\shj_\varepsilon(\Q^*,\P) < \infty$.
          \item We suppose first $\eta_T \sim \P_T$ and we prove that
            \begin{equation}
            \label{eq:mineq1}
              \shj_\epsilon(\Q, \P) \ge \frac{1}{\epsilon}\int_{\R^d}\log \frac{\lambda(x)}{\gamma_\epsilon(x)}\eta_T(dx), \quad \forall   \Q \in \shp(\Omega, \eta_T).
\end{equation}

          \item We establish \eqref{eq:mineq1} when only $\eta_T \ll \P_T$.
            In particular, by \eqref{eq:mineq2}, this establishes \eqref{eq:optimalCost}.
          \end{enumerate}
We proceed  proving the previous points.
\begin{enumerate}
                \item Let us verify that $\Q^* \in \shp(\Omega, \eta_T)$. For $\psi \in \shc_b(\R^d)$
we have
			\begin{equation*}
				\begin{aligned}
					\E^{\Q^*}[\psi(X_T)] & = \E^{\P}\left[\psi(X_T)\exp\left(- \epsilon\vphi(X)\right)\frac{\lambda(X_T)}{\gamma_\epsilon(X_T)}\right] \\
					& = \E^{\P}\left[\psi(X_T)\frac{\lambda(X_T)}{\gamma_\epsilon(X_T)}\E^\P\left[\exp\left(- \epsilon\vphi(X)\right) \middle | X_T\right]\right]\\
					& = \E^\P\left[\psi(X_T)\frac{\lambda(X_T)}{\gamma_\epsilon(X_T)}\gamma_\epsilon(X_T)\right]\\
					& = \E^\P\left[\psi(X_T)\frac{d\eta_T}{d\P_T}(X_T)\right]\\
					& = \int_{\R^d}\psi(x)\eta_T(dx).
				\end{aligned}
                              \end{equation*}
                              Since the previous equality holds for all $\psi \in \shc_b(\R^d)$, $\Q^*_T = \eta_T$ and $\Q^* \in \shp(\Omega, \eta_T)$.

                              We now prove \eqref{eq:mineq2}. By \eqref{eq:densityQ*} we have
			\begin{equation}
				\label{eq:entropQ*}
				\log \frac{d\Q^*}{d\P} = - \epsilon \vphi(X) +
                                \log \frac{\lambda}{\gamma_\epsilon}(X_T) \quad \Q^*\text{-a.s.}
                              \end{equation}
since $\P$-a.s.
                  Consequently, taking the expectation with respect to $\Q^*$ we get
\begin{equation} \label{eq:QPstar}
                              H(\Q^*\vert \P) =  - \epsilon \E^{\Q^*}\vphi(X)) +
                              \E^{\Q^*} \left[\log \frac{\lambda}{\gamma_\epsilon}(X_T)\right].
                              \end{equation}
                  Since the left-hand side is well-defined by  Remark \ref{rmk:relativeEntropy} item 1.
                 also the second expectation on the right-hand side is well-defined.
		
  By Definition \eqref{eq:optimalCost} of $\shj_\epsilon$ and \eqref{eq:QPstar} we get
			\begin{equation}
				\label{eq:proofOptimalCost}
                                \shj_\epsilon(\Q^*, \P)  = \E^{\Q^*}\left[\vphi(X)\right] - \E^{\Q^*}\left[\vphi(X)\right] + \frac{1}{\epsilon}\E^{\Q^*}\left[\log\frac{\lambda(X_T)}{\gamma_\epsilon(X_T)}\right].
			\end{equation}
                        This shows the first equality of \eqref{eq:mineq2}, the second equality follows
                        because $\Q^*_T = \eta_T$.

                        Moreover $\shj_\varepsilon(\Q^*,\P) < \infty$ since
        $$ \E^{\Q^*} \left[\log \frac{\lambda}{\gamma_\epsilon}(X_T)\right] = H(\eta_T\vert \P_T) - \E^{\eta_T}\left[\log(\gamma_\varepsilon(X_T))\right],$$
        and the  expectation on the left-hand side is smaller than $+\infty$,
        since
        $\gamma_\varepsilon$ is lower bounded by a positive constant $\P_T$ a.e. therefore      $\eta_T$ a.e.     
                        
           \item Assume that $\eta_T \sim \P_T$. This implies in particular
                        that $\Q^* \sim \P$, taking into account \eqref{eq:densityQ*}.
                        Let $\Q \in \shp(\Omega, \eta_T)$ such that $\shj_\epsilon(\Q, \P) < + \infty$, which means equivalently
$H(\Q\vert \P) < +\infty$. 
                        Since $\Q \ll \P$ and $\Q^* \sim \P$, it holds that $\Q \ll \Q^*$ and we have
			\begin{equation*}
				\begin{aligned}
					\log \frac{d\Q}{d\Q^*} & = \log \frac{d\Q}{d\P} + \log \frac{d\P}{d\Q^*}\\
					& =  \log \frac{d\Q}{d\P} + \epsilon \vphi(X) - \log\frac{\lambda(X_T)}{\gamma_\epsilon(X_T)},
				\end{aligned}
			\end{equation*}
			where the previous equality holds $\Q$-a.s. since it holds $\P$-a.s. Taking the expectation under $\Q$ then yields
			\begin{equation}
\label{eq:minimizerTerminal1}
H(\Q | \Q^*) = H(\Q | \P) + \epsilon \E^\Q[\vphi(X)] - \E^{\Q}\left[\log\frac{\lambda(X_T)}{\gamma_\epsilon(X_T)}\right],
			\end{equation}
			and since $\Q_T = \eta_T$, \eqref{eq:minimizerTerminal1} rewrites
			\begin{equation}
                          \label{eq:minimizerTerminal2}
                          H(\Q | \Q^*) = H(\Q | \P) +  \epsilon \E^\Q[\vphi(X)] - \int_{\R^d}\log \frac{\lambda(x)}{\gamma_\epsilon(x)}\eta_T(dx)
                  = \shj_\epsilon(\Q, \P)     -  \int_{\R^d}\log \frac{\lambda(x)}{\gamma_\epsilon(x)}\eta_T(dx).
			\end{equation}
			Dividing both sides in \eqref{eq:minimizerTerminal2} by $\epsilon > 0$, we get
			$$
			\shj_\epsilon(\Q, \P) = \frac{1}{\epsilon}H(\Q | \Q^*) + \frac{1}{\epsilon}\int_{\R^d}\log \frac{\lambda(x)}{\gamma_\epsilon(x)}\eta_T(dx) \ge \frac{1}{\epsilon}\int_{\R^d}\log \frac{\lambda(x)}{\gamma_\epsilon(x)}\eta_T(dx).
			$$
                        This concludes the proof of \eqref{eq:mineq1}.

                      \item We now turn to the general case. For $\alpha \in ]0, 1[$, we set $\eta^\alpha_T := \alpha \P_T + (1 - \alpha)\eta_T$.
   By convexity of the relative entropy (see Remark \ref{rmk:relativeEntropy} item 2.), it holds that $H(\eta^\alpha_T | \P_T) \le (1 - \alpha) H(\eta_T | \P_T) < + \infty$.
       We apply now \eqref{eq:mineq1} in item 2. replacing $\shp(\Omega, \eta_T)$ by $\shp(\Omega, \eta^\alpha_T)$.
       This is possible because $\eta^\alpha_T \sim \P_T$, so that
       \begin{equation}
       	\label{eq:etaAlpha}
       	\shj_\epsilon(\Q^\alpha, \P) \ge \frac{1}{\epsilon}\int_{\R^d}\log \frac{\lambda^\alpha (x)}{\gamma_\epsilon(x)}\eta^\alpha_T(dx) = \frac{\alpha}{\epsilon}\int_{\R^d}\log \frac{\lambda^\alpha (x)}{\gamma_\epsilon(x)}\P_T(dx) + \frac{1 - \alpha}{\epsilon}\int_{\R^d}\log \frac{\lambda^\alpha (x)}{\gamma_\epsilon(x)}\eta_T(dx),
       \end{equation}
       where $\lambda^\alpha := d\eta^\alpha_T/d\P_T = \alpha + (1- \alpha)\lambda$. On the one hand, for any $\Q \in \shp(\Omega, \eta_T)$, $\Q^\alpha := \alpha \P + (1 - \alpha)\Q \in \shp(\Omega, \eta^\alpha_T)$, and again
     the convexity of the relative entropy yields $H(\Q^\alpha | \P) \le (1 - \alpha) H(\Q | \P)$.
		Hence
		\begin{equation}
			\label{eq:etaAlphaInter}
			\begin{aligned}
				\shj_\epsilon(\Q^\alpha, \P) & = \E^{\Q^\alpha}[\vphi(X)] + \frac{1}{\epsilon}H(\Q^\alpha| \P)\\
				& = (1 - \alpha)\E^\Q[\vphi(X)] + \alpha \E^\P[\vphi(X)] + \frac{1}{\epsilon}H(\Q^\alpha | \P)\\
				& \le (1 - \alpha)\E^\Q[\vphi(X)] + \alpha \E^\P[\vphi(X)] + \frac{1 - \alpha}{\epsilon}H(\Q |\P) \\
				& = (1 - \alpha)\shj_\epsilon(\Q, \P) + \alpha \E^\P[\vphi(X)],
			\end{aligned}
		\end{equation}
		and from \eqref{eq:etaAlphaInter} and \eqref{eq:etaAlpha}
                we get that
			\begin{equation}
				\label{eq:etaAlpha2}
				\begin{aligned}
				(1 - \alpha)\shj_\epsilon(\Q, \P) + \alpha \E^\P[\vphi(X)]
				\ge \frac{\alpha}{\epsilon}\int_{\R^d}\log \frac{\lambda^\alpha (x)}{\gamma_\epsilon(x)}\P_T(dx) + \frac{1 - \alpha}{\epsilon}\int_{\R^d}\log \frac{\lambda^\alpha (x)}{\gamma_\epsilon(x)}\eta_T(dx),
				\end{aligned}
			\end{equation}
			for all $\Q \in \shp(\Omega, \eta_T)$. On the other hand, by concavity 
of $x \mapsto \log(x)$, we have $\log \frac{\lambda^\alpha (x)}{\gamma_\epsilon(x)} \ge  (1 - \alpha) \log \frac{\lambda (x)}{\gamma_\epsilon(x)}$ and from \eqref{eq:etaAlpha2} we deduce that
			\begin{equation}
				\label{eq:etaAlpha3}
				(1 - \alpha)\shj_\epsilon(\Q, \P) + \alpha \E^\P[\vphi(X)] \ge \frac{\alpha(1- \alpha)}{\epsilon}\int_{\R^d}\log \frac{\lambda (x)}{\gamma_\epsilon(x)}\P_T(dx) + \frac{(1 - \alpha)^2}{\epsilon}\int_{\R^d}\log \frac{\lambda (x)}{\gamma_\epsilon(x)}\eta_T(dx).
			\end{equation}
	Letting $\alpha \rightarrow 0$ in \eqref{eq:etaAlpha3} yields
			$
			\shj_\epsilon (\Q, \P) \ge \frac{1}{\epsilon}\int_{\R^d}\log \frac{\lambda (x)}{\gamma_\epsilon(x)}\eta_T(dx)
			$
			for all $\Q \in \shp(\Omega, \eta_T)$ and we
finally establish \eqref{eq:mineq1} in the general case.
		\end{enumerate}
	\end{proof}
	We recall below Lemma B.6 in \cite{BORMarkov2023} applied with $g = - \log \frac{\lambda}{\gamma_\epsilon}$.
	\begin{lemma}
		\label{lemma:QMarkovian}
		Let $\P$ be a Markov probability measure in the sense of Definition \ref{def:markovProba}. Let $f \in \shb([0, T] \times \R^d, \R)$ be a non-negative bounded Borel function. Then $\Q^*$ defined
                in \eqref{eq:densityQ*} in Proposition \ref{prop:existenceMinQmu} statement, applied with $\vphi(X) = \int_0^Tf(r, X_r)dr$ is Markovian,
                see Definition \ref{def:markovProba}.

	\end{lemma}
	
	\section{Measurable selection of solutions of a variational inequality}
	\setcounter{equation}{0}
	\renewcommand\theequation{C.\arabic{equation}}
	\label{app:measurableSelection}

	This section is dedicated to the proof of Lemma \ref{lemma:existenceImpliesMeasurable}, which states the existence of measurable selectors for the set of solutions
        of mixed variational inequalities.
	The proof of the result below borrows some ideas from the proof of Theorem A.9 in \cite{HaussmanLepeltierOptimal}. We recall the Definition \ref{def:correspondence} of a correspondence and a measurable selector.
        
	\begin{prop}
		\label{prop:measurableSelectionMVI}
		Let $K$ be a compact metric space.
                Let $(S, d_S)$ be a metric space. Let $\vphi : S \times K \times K \rightarrow \R$ be a continuous function. Assume that for all $s \in S$ there exists $y(s) \in K$ such that $\vphi(s, x, y(s)) \ge 0$ for all $x \in K$. Then there exists a Borel function $\bar y : S \rightarrow K$ such that for all $s \in S$, $\vphi(s, x, \bar y(s)) \ge 0$ for all $x \in K$.
	\end{prop}
	\begin{proof}
		For all $s \in S$ we set
		\begin{equation}
			\label{eq:defTs}
			\sht(s) := \{y~\in K~:~\vphi(s, x, y) \ge 0 ~\text{for all}~x \in K\}.
		\end{equation}
                Under this assumption we want to apply Kuratowski–Ryll-Nardzewski selection theorem, see e.g. Theorem 18.13 in \cite{chara}, to prove that the correspondence $\sht$ has a measurable selector, where the corresponding spaces are equipped with their Borel $\sigma$-algebra.
                We remark that the compact space $K$ is necessarily separable.
 We have to verify the following.
			\begin{enumerate}[label =(\roman*)]
                        \item $\sht$ takes values in the non-empty closed sets of $K$.
				\item $\sht$ is weakly measurable in the sense that $\{s \in S~:~\sht(s) \cap G \neq \emptyset\}$ is an element of $\shb(S)$ for each open set $G$ of $K$.
			\end{enumerate}
			We first check item (i). We fix $s \in S$. By assumption $\sht(s) \neq \emptyset$. Let then $(y_n)_{n \ge 1}$ be a sequence of elements of $\sht(s)$ which converges towards a limit $y \in K$. By definition \eqref{eq:defTs} of $\sht(s)$, for all $n \ge 1$, for all $x \in K$, $\vphi(s, x, y_n) \ge 0$. Since $\vphi$ is continuous, letting $n \rightarrow + \infty$ yields $\vphi(s, x, y) \ge 0$, hence $y \in \sht(s)$ and $\sht(s)$ is closed. This proves item (i).
			
			\noindent We now turn to item (ii). By Lemma 18.2 item 1. in \cite{chara} it is enough to show that
                        $\sht^\ell(F) := \{s \in S \vert \sht(s) \cap F \neq \emptyset\}$ is an element of $\shb(S)$ for each closed set $F$ of $K$. Let then $F$ be a closed set of $K$. We are going to show that $\sht^\ell(F)$ is closed in $S$. Let $(s_n)_{n \ge 1}$ be a sequence of elements of $\sht^\ell(F)$ converging towards some $s \in S$. For all $n \ge 1$, let $y_n$ be an element of $\sht(s_n) \cap F$. $(y_n)_{n \ge 1}$ is a sequence of elements of $K$ which is compact, hence it admits a converging subsequence still denoted $(y_n)_{n \ge 1}$. We set $K \ni y := \underset{n \rightarrow + \infty}{\lim} y_n$. Then for all $n \ge 1$, $x \in K$, $\vphi(s_n, x, y_n) \ge 0$ and letting $n \rightarrow + \infty$ we get that $\vphi(s, x, y) \ge 0$ by continuity of $\vphi$. It follows that $y \in \sht(s)$ and since $F$ is closed, we also have $y \in F$, that is $y \in \sht(s) \cap F$. Hence $s \in \sht^\ell(F)$, $\sht^\ell(F)$ is closed and in particular $\sht^\ell(F) \in \shb(S)$ for all closed subset $F$ of $S$. The correspondence $\sht$ is weakly measurable so item (ii) is also verified.

              \end{proof}
	We are now able to prove Lemma \ref{lemma:existenceImpliesMeasurable}.
	\begin{proof}[Proof of Lemma \ref{lemma:existenceImpliesMeasurable}.]
		We set $S := [0, T] \times \R^d \times \R^d$. Let $\vphi: S \times \U \times \U  \rightarrow \R$  be defined for all $s = (t, x, \delta) \in S$, $u, \nu \in \U$ by
		\begin{equation}
			\label{eq:vphiMvi}
			\vphi(s, u, \nu) := f(t, x, u) - f(t, x, \nu) + \frac{1}{\epsilon}\langle \Sigma^{-1}(t, x)(b(t, x, \nu) - \delta), b(t, x, u) - b(t, x, \nu)\rangle.
		\end{equation}
		By Hypothesis \ref{hyp:MVIPointwise},
                \eqref{eq:MVIPointwise} has a solution for any $(t, x, \delta) \in [0, T] \times \R^d \times \R^d$. This is equivalent
                to say that for all $s \in S$ there exists $\bar u \in \U$ such that for all $u \in \U$, $\vphi(s, u, \bar u) \ge 0$. Since $f$, $b$ and $\sigma$ are continuous, $\vphi$ is continuous and we can apply Proposition \ref{prop:measurableSelectionMVI} with $K = \U$.
		there exists a  Borel function $\shu : S \rightarrow \U$ such that for all $s \in S$, $u \in \U$, $\vphi(s, u, \shu(s)) \ge 0$.

                Let now $\delta \in \shb([0, T] \times \R^d, \R^d)$. For all $(t, x) \in [0, T] \times \R^d$, the function $\bar u$ defined by $\bar u(t, x) := \shu(t, x, \delta(t, x))$, is Borel measurable being
                the composition of two Borel functions, and for all $(t, x) \in [0, T] \times \R^d$, $\bar u(t, x)$ verifies \eqref{eq:MVIPointwise} with $\delta = \delta(t, x)$.
	\end{proof}
        \begin{remark} \label{rmk:MVI_PathDep}.
          It is possible to write a path-dependent version of previous proof in order to justify Remark \ref{rmk:pathDependentMeasSelec}.
          Let $\Omega_0$ as in the aforementioned Remark \ref{rmk:pathDependentMeasSelec}.
          We can follow previous proof setting $  S = \Omega_0 \times \R^d$. We define $\varphi: S \times \U \times \U \rightarrow \R$
          by
         	\begin{equation}
			\label{eq:vphiMviPD}
			\vphi(s, u, \nu) := f(t, X_t, u) - f(t, X_t, \nu) + \frac{1}{\epsilon}\langle \Sigma^{-1}(t, X_t)(b(t, X_t, \nu) - \delta), b(t, X_t, u) - b(t, X_t, \nu)\rangle.
		\end{equation}
                At the end of the proof we consider $\delta:\Omega_0 \rightarrow \R^d$ be a Borel functional and we define
  $\bar u(t, X) := \shu(t, X_t, \delta(t, X))$, which is Borel measurable being
                the composition of two Borel functions, and for all $(t, X) \in \Omega_0$, $\bar u(t, X)$ verifies \eqref{eq:MVIPointwise} with $\delta = \delta(t, X)$.
              \end{remark}

	\section{Exponential martingales}
	\label{sec:expoMart}
	\setcounter{equation}{0}
	\renewcommand\theequation{D.\arabic{equation}}
	In this section we gather and prove some results on exponential martingales. Let $\P \in \shp(\Omega)$ be the law of a semimartingale and let
        $M^\P$ be its continuous local martingale part (see Proposition 4.27, Chapter I in \cite{JacodShiryaev})
        verifying $[M^\P] = \int_0^\cdot a_rdr$ for some progressively measurable process $a : [0, T] \times \Omega \rightarrow S_d^+$.
$L$ is a Lévy kernel in the sense of Definition \ref{def:levyKernel}.
Let $Y$ be a strictly positive $\tilde \shp$-measurable function. Let $\alpha : [0, T] \times \Omega \rightarrow \R^d$ be a progressively measurable process.
We will assume in all this section that Hypothesis \ref{hyp:expoMart}
below is in force for the triplet
$(Y, \P, L)$.
\begin{hyp}
\label{hyp:expoMart}
		\begin{enumerate}
			\item The compensator $\mu^L$ of $\mu^X$ under $\P$ is given by $\mu^L = L(t, X, dq)dt$ where $L$ is a Lévy kernel in the sense of Definition \ref{def:levyKernel}.
			
			\item $\E^\P\left[\left(Y - \log(Y) - 1\right)*\mu^L_T\right] < + \infty$.
		\end{enumerate}
	\end{hyp}
	We start by a preliminary result. We observe that $y \mapsto y - \log(y) - 1 $ is a non-negative function.

        \begin{lemma}
          \label{lemma:integExpMart}
          We have the following.
          \begin{enumerate}
          \item $\log(Y) \in \shg_{loc}^\P(\mu^X)$ and $Y - 1 \in \shg_{loc}^\P(\mu^X)$.
          \item
            \begin{equation}
			\label{eq:integLog}
			\left(|\log(Y)|\1_{\{|\log(Y)| > 1 \}}\right)*\mu^L_T + \left(|\log(Y)|^2\1_{\{|\log(Y)| \le 1 \}}\right)*\mu^L_T \in \sha_{loc}^+(\P).
                      \end{equation}
                      \end{enumerate}

          \end{lemma}
	\begin{proof}
          Set $\theta : z \in \R \mapsto e^z - z - 1$. Item 2. of Hypothesis \ref{hyp:expoMart} rewrites
		 \begin{equation}
			\label{eq:finiteExp}
			\E^\P\left[\theta(\log(Y))*\mu^L_T\right] < + \infty.
		\end{equation}
		\eqref{eq:finiteExp} directly implies \eqref{eq:integLog}
                by Lemma \ref{lemma:usefulIneq} item 1.
		We then conclude that $\log(Y) \in \shg_{loc}^\P(\mu^X)$ by applying Proposition \ref{prop:characGloc} with $b_0 = 1$.

		On the other hand, since $Y > 0$, we have $\1_{\{|Y - 1| > 1\}} = \1_{\{Y - 1 > 1\}}$. Then Lemma \ref{lemma:usefulIneq} item 2. together with \eqref{eq:finiteExp} yields
		$$
		\left(|Y - 1|\1_{\{|Y - 1| > 1 \}}\right)*\mu^L_T + \left(|Y - 1|^2\1_{\{|Y - 1| \le 1 \}}\right)*\mu^L_T \in \sha_{loc}^+(\P).
		$$
		We similarly conclude that $Y - 1 \in \shg_{loc}^\P(\mu^X)$ by applying Proposition \ref{prop:characGloc} with $ W =Y-1, b_0 = 1$.
	\end{proof}
	Lemma \ref{lemma:integExpMart}
        tells that $M + \left(Y - 1\right)*(\mu^X - \mu^L)$ is a well-defined
        local martingale and the
Doléans-Dade exponential martingale
	\begin{equation}
          \label{eq:expMartingale}
		Z := \she\left(M + \left(Y - 1\right)*(\mu^X - \mu^L)\right)
              \end{equation}
              is well-defined.
	\begin{lemma}
		\label{lemma:rewriteZ}
		The local martingale $Z$ defined
                by \eqref{eq:expMartingale} rewrites
		\begin{equation}
	\label{eq:rewriteZ}
			Z = \exp\left(M^\P + \log(Y)*(\mu^X - \mu^L) - \frac{1}{2}[M^\P] - \left(Y - \log(Y) - 1\right)*\mu^L\right).
		\end{equation}
	\end{lemma}
	\begin{proof}
          We set $\ell^- := \log(Y)\1_{\{|\log(Y)| \le 1\}}$ and $\ell^+ := \log(Y)\1_{\{|\log(Y)| > 1\}}$. We first establish  that
		\begin{equation}
			\label{eq:interZ}
			Z = \exp\left(M^\P - \frac{1}{2}[M^\P] + \left(\ell^+\right)*\mu^X + \left(\ell^-\right)*(\mu^X - \mu^L) - \left((Y - \ell^- - 1),
                          \right)*\mu^L\right).
		\end{equation}
              Now, the equality \eqref{eq:interZ} is a consequence of
              Proposition IV.5 in \cite{LepingleExpoMart} applied
              with $z = \log(Y)$, provided that
		\begin{equation}
			\label{eq:condLepingle}
			\left(|\ell^+|\right)*\mu^X_T <  +\infty, \quad (\ell^-)^2*\mu^L_T < + \infty, \quad \text{and} \quad \left|e^{\ell^+} - 1\right|*\mu^L_T < + \infty
                        \quad \P\text{-a.s.}
		\end{equation}
                So let us verify \eqref{eq:condLepingle}.
                The first two conditions therein
                are a direct consequence of \eqref{eq:integLog} in Lemma \ref{lemma:integExpMart}.
                Moreover, we have
		\begin{equation}
			\label{eq:interDolean}
			\begin{aligned}
				\left|e^{\ell^+} - 1\right| & = |Y - 1|\1_{\{|\log(Y)| > 1\}}\\
& \le \vert Y - 1\vert \1_{\{\vert Y - 1\vert > 1 - e^{-1}\}}.
			\end{aligned}
		\end{equation}
		Now since $Y - 1 \in \shg_{loc}^\P(\mu^X)$ by Lemma \ref{lemma:integExpMart},
                Proposition \ref{prop:characGloc} applied with
                $W = Y - 1,
                b_0 = 1 - e^{-1}$ yields $\left(|Y - 1|\1_{\{|Y - 1| > 1 - e^{-1}\}}\right)*\mu^L \in \sha_{loc}^+(\P)$, which directly implies by \eqref{eq:interDolean} that $\left|e^{\ell^+} - 1\right|*\mu^L_T < + \infty$ $\P$-a.s. Consequently
                \eqref{eq:condLepingle} is verified and so also \eqref{eq:interZ} holds.
		
                We finally  conclude the proof of
                \eqref{eq:rewriteZ}.
                This follows
                from \eqref{eq:interZ}
because $(\ell^+)*\mu^X = (\ell^+)*(\mu^X - \mu^L) + (\ell^+)*\mu^L$
taking into account  $|\ell^+|*\mu^L \in \sha_{loc}^+(\P)$, by \eqref{eq:condLepingle},
Proposition 1.28, Chapter III in \cite{JacodShiryaev}.
	\end{proof}

%

	\section{Miscellaneous}
	\setcounter{equation}{0}
	\renewcommand\theequation{E.\arabic{equation}}

        The proposition below states the celebrated minimization problem,
        known as {\it exponential twist} or {\it Donsker-Varadhan},
        see Proposition 2.5 in \cite{EntropyWeighted},
        recalled also in Proposition 3.13 of \cite{BOROptimi2023}.

	\begin{prop}
		\label{prop:minimizerUnconstrained}
		Let $\vphi : \Omega \rightarrow \R$ be a Borel function and $\P \in \shp(\Omega)$. Assume that $\vphi$ is bounded from below. Then
		\begin{equation}
			\label{eq:klOpti}
			{\rm min}_{\Q \in \shp(\Omega)} \E^{\Q}[\vphi(X)] + \frac{1}{\epsilon} H(\Q | \P) = - \frac{1}{\epsilon}\log \E^{\P}\left[\exp(-\epsilon\vphi(X))\right].
		\end{equation}
		Moreover there exists a unique minimizer $\Q^* \in \shp(\Omega)$ given by
		\begin{equation}
			\label{eq:expoTwist}
			d\Q^* = \frac{\exp(-\epsilon\vphi(X))}{\E^{\P}[\exp(-\epsilon\vphi(X))]}d\P.
		\end{equation}
	\end{prop}
	We restate below Lemma F.1 in \cite{BOROptimi2023}.
	\begin{lemma}
		\label{lemma:squareIntVar}
		Let $\eta$ be a real square integrable  random variable satisfying $\E\left[\exp(-\epsilon \eta)\right] < + \infty$. Then for all $\epsilon > 0$,
		$
		0 \le \E[\eta] - \left(-\frac{1}{\epsilon}\log\E\left[\exp(-\epsilon \eta)\right]\right) \le
		\frac{\epsilon}{2}Var[\eta].
		$
	\end{lemma}
	Below we reformulate Lemma F.2 in \cite{BOROptimi2023}.
	\begin{lemma}
		\label{lemma:nelsonDerivative}
		Let $(X_t)_{t \in [0, T]}$ be an $(\shf_t)$-adapted process of the form
		$
		X_t = x + \int_0^t b_rdr + M_t,
		$
		where $\E\left[\int_0^T |b_r|^pdr\right] < + \infty$ for some $p > 1$ and where $M$ is a martingale. For Lebesgue almost all $0 \le t < T$
		$
		\lim_{h \downarrow 0}\E\left[\frac{X_{t + h} - X_t}{h}~\middle|~\shf_t\right] = b_t~\text{in}~L^1(\P).
		$
	\end{lemma}

        \begin{lemma}
     	\label{lemma:4PointsIneq}
		Let $a, b, c, d \in \R^d$. Then
		$
		\frac{1}{2}|a - b|^2 - \frac{1}{2}|a - c|^2 + \frac{1}{2}|a - d|^2 \ge \langle c - b, d - c \rangle
		$
	\end{lemma}
	\begin{proof}
		Applying the algebraic inequality $|\alpha|^2 - |\beta|^2 = |\alpha - \beta|^2 + 2\langle \alpha - \beta, \beta \rangle$ to $\alpha = a - b$ and $\beta = a - c$, we have
		\begin{equation}
			\label{eq:algebraicManip}
			\begin{aligned}
				\frac{1}{2}|a - b|^2 - \frac{1}{2}|a - c|^2 & = \frac{1}{2}|c - b|^2 + \langle c - b, a - c \rangle\\
				& = \frac{1}{2}|c - b|^2 + \langle c - b, a - d\rangle + \langle c - b, d - c \rangle.
			\end{aligned}
		\end{equation}
		As $\frac{1}{2}|c - b|^2 + \langle c - b, a - d \rangle = \frac{1}{2}|(c - b) + (a - d)|^2 - \frac{1}{2}|a - d|^2$, we get from \eqref{eq:algebraicManip} that
		\begin{equation*}
			\begin{aligned}
				\frac{1}{2}|a - b|^2 - \frac{1}{2}|a - c|^2 + \frac{1}{2}|a - d|^2 & = \frac{1}{2}|(c - b) + (a - d)|^2 + \langle c - b, d - c \rangle \ge \langle c - b, d - c \rangle.
			\end{aligned}
		\end{equation*}
	\end{proof}
	\begin{lemma}
          \label{lemma:convexOptimality}
Let $g: \R^d
  \rightarrow \R$ are convex and $h$ is differentiable.
  Let $U \subset \R^d$ be a convex set and set $F = g + h$.
  Assume that $F$
  restricted to $U$ has a minimum $x^*$.
  Then for all $x \in U$,
\begin{equation}
			\label{eq:convexOptimality}
			g(x) - g(x^*) + \langle \nabla_x h(x^*), x - x^*\rangle \ge 0.
		\end{equation}
\end{lemma}
          
	\begin{proof}
		Let $\lambda \in ]0, 1]$. By definition of $x^*$, for all $x \in U$,
		$$
		F(\lambda x + (1 - \lambda)x^*) - F(x^*) \ge 0,
		$$
		that is
		\begin{equation}
			\label{eq:convex1}
			g(\lambda x + (1 - \lambda)x^*) - g(x^*) + h(\lambda x + (1 - \lambda)x^*) - h(x^*) \ge 0.
		\end{equation}
		Since $g$ is convex,
		$$
		g(\lambda x + (1 - \lambda) x^*) - g(x^*) \le \lambda g(x) + (1 - \lambda)g(x^*) - g(x^*) = \lambda(g(x) - g(x^*)),
		$$
		and \eqref{eq:convex1} then implies
		\begin{equation}
			\label{eq:convex2}
			g(x) - g(x^*) + \frac{1}{\lambda}(h(\lambda x + (1 - \lambda)x^*) - h(x^*)) \ge 0.
		\end{equation}
		Letting $\lambda \rightarrow 0$ in \eqref{eq:convex2} yields \eqref{eq:convexOptimality}.
	\end{proof}

        We gather in the following result some inequalities which are useful to prove the results of Section \ref{sec:expoMart}.
		\begin{lemma}
		\label{lemma:usefulIneq}
		Set $\theta : z \in \R \mapsto e^z - z - 1$.
		\begin{enumerate}
			\item For all $z \in \R$,
                  $	\theta(z) \ge \frac{1}{e}|z|\1_{\{|z| > 1\}} + \frac{1}{2}|z|^2\1_{\{|z| \le 1\}}.
			$
			\item For all $z \in ]0, + \infty[$,
			$
			\theta(\log(z)) \ge (1 - \log(2))(z - 1)\1_{\{z \ge 2 \}} + \frac{1}{8} (z - 1)^2\1_{\{ 0 < z < 2\}}.
			$
		\end{enumerate}
	\end{lemma}
	\begin{proof}
          We first prove that
\begin{equation} \label{eq:>1}
          \vert z \vert >1 \Rightarrow \theta(z) \ge \frac{\vert z\vert}{e}.  
        \end{equation}
        This follows because the functions 
        $z \mapsto \theta(z) - \frac{ z}{e}$ can be proved to be increasing for $z > 1$
        and $z \mapsto \theta(z) + \frac{z}{e}$ to be decreasing for $z < - 1$
        by direct evaluation of the derivatives.
        Moreover,
\begin{equation} \label{eq:<1}
          \vert z \vert \le 1 \Rightarrow \theta(z) \ge \frac{\vert z\vert^2}{2},  
        \end{equation}
follows because, by the Taylor expansion, there is $\xi \in [0,1]$ such that 
$$ \theta(z) = e^{\xi z}\frac{z^2}{2} \ge  \frac{z^2}{2}.$$ 
Concerning item 2. we remark that  $\theta(\log(z)) = z - \log(z) -1$.
We first prove that
$$ z  \ge 2 \Rightarrow \theta(z) \ge (1-\log(2))(z-1).$$
The result follows
since $z \mapsto \theta(\log(z)) - (1-\log(2))(z-1)$ vanishes at $z = 2$
and is increasing for $z \ge 2$, since the derivative is positive.

It remains to show
$$ 0 < z < 2 \Rightarrow \theta(\log(z)) \ge  \frac{1}{8}(z - 1)^2. $$
By a first order Taylor expansion around $1$  for $0 < z < 2$ we get
$$ \theta(\log(z)) = (z-1)^2 \frac{1}{2 \xi^2}  \ge  (z-1)^2 \frac{1}{8},$$
  where $ 0 < \xi < 2$ and the result follows.

	\end{proof}
	\section*{Acknowledgments}


The research of the first named author is supported by a doctoral fellowship
PRPhD 2021 of the Région Île-de-France.
The research of the second and third named authors was partially
supported by the  ANR-22-CE40-0015-01 project SDAIM.
	\bibliographystyle{plain}
	\bibliography{../../../../BIBLIO_FILE/ThesisBourdais}
\end{document}